\documentclass[11pt,letterpaper]{amsart}
\usepackage[
    paperwidth=8.5in,
    paperheight=11in,
    margin=1.25in,
    marginparwidth=0.75in,
    marginparsep=0.1in,
    footskip=0.25in
]{geometry}

\usepackage[shortlabels]{enumitem}
\setlist{
    itemsep=1ex,
    listparindent=\parindent,
    parsep=0em,
    topsep=2ex,
}
\allowdisplaybreaks[1]

\usepackage{xpatch}
\makeatletter
\xpatchcmd{\@tocline}
{\hfil\hbox to\@pnumwidth{\@tocpagenum{#7}}\par}
{\ifnum#1<0\hfill\else\dotfill\fi\hbox to\@pnumwidth{\@tocpagenum{#7}}\par}
{}{}
\makeatother

\usepackage{amsfonts}
\usepackage{amsmath}
\usepackage{amssymb}
\usepackage{amsthm,thmtools}
\usepackage{bm}
\usepackage[mathscr]{euscript}
\usepackage{mathtools}

\usepackage{booktabs}
\usepackage{comment}
\usepackage{float}
\usepackage{graphicx}
\usepackage{tikz}
\usepackage{pgfplots}
\usepgfplotslibrary{groupplots}
\usepackage{multirow}

\usepackage{caption}
\usepackage[
    backend=biber,
    style=numeric,
    giveninits=true,
    sorting=nyt,
    maxnames=100,
    backref=false
]{biblatex}
\renewbibmacro{in:}{}
\DeclareFieldFormat[
	article,
	inbook,
	incollection,
	inproceedings,
	patent,
	thesis,
	unpublished
]{title}{#1}

\newcommand\linkcolor{black}
\newcommand\pref[1]{\textcolor{\linkcolor}{(\ref{#1})}}

\usepackage{hyperref}
\hypersetup{
    colorlinks,
    linktoc=all,
    citecolor=black,
    linkcolor=\linkcolor,
    urlcolor=\linkcolor
}
\usepackage[capitalize,nameinlink,noabbrev]{cleveref}
\crefname{fact}{Fact}{Facts}
\crefname{inequality}{Inequality}{Inequalities}
\creflabelformat{inequality}{#2(#1)#3}
\crefname{lemma}{Lemma}{Lemmas}

\theoremstyle{plain}
\newtheorem{theorem}{Theorem}[section]
\newtheorem*{theorem*}{Theorem}

\newtheorem{corollary}{Corollary}[theorem]

\newtheorem{lemma}[theorem]{Lemma}
\newtheorem*{lemma*}{Lemma}
\newtheorem{proposition}[theorem]{Proposition}

\theoremstyle{definition}
\newtheorem{definition}[theorem]{Definition}

\newtheorem{conjecture}[theorem]{Conjecture}

\newtheorem{question}[theorem]{Question}

\newtheorem*{remark*}{Remark}

\newcommand\C{{\mathbb C}}
\newcommand\E{{\mathbb E}}

\newcommand\N{{\mathbb N}}
\newcommand\bbP{{\mathbb P}}

\newcommand\R{{\mathbb R}}

\newcommand\bsone{\boldsymbol{1}}
\newcommand\cB{{\mathcal B}}
\newcommand\cC{{\mathcal C}}
\newcommand\cD{{\mathcal D}}
\newcommand\cE{{\mathcal E}}
\newcommand\cF{{\mathcal F}}
\newcommand\cG{{\mathcal G}}
\newcommand\cH{{\mathcal H}}
\newcommand\cI{{\mathcal I}}

\newcommand\cM{{\mathcal M}}
\newcommand\cN{{\mathcal N}}

\newcommand\cP{{\mathcal P}}
\newcommand\cQ{{\mathcal Q}}
\newcommand\cR{{\mathcal R}}
\newcommand\cS{{\mathcal S}}
\newcommand\cT{{\mathcal T}}
\newcommand\cU{{\mathcal U}}
\newcommand\cUbar{\overline{\cU}}
\newcommand\cUhat{\widehat{\cU}}
\newcommand\cV{{\mathcal V}}
\newcommand\cW{{\mathcal W}}
\newcommand\cX{{\mathcal X}}
\newcommand\cY{{\mathcal Y}}
\newcommand\cZ{{\mathcal Z}}

\newcommand\sB{{\mathscr B}}

\newcommand\sD{{\mathscr D}}

\newcommand\sK{{\mathscr K}}

\newcommand\sP{{\mathscr P}}

\newcommand\erdos{Erd\H{o}s}
\newcommand\renyi{R\'enyi}
\newcommand\erdosrenyi{\erdos--\renyi}

\makeatletter
\newcommand{\vast}{\bBigg@{3}}
\makeatother

\usepackage{scalerel,stackengine}
\newcommand\pig[1]{\scalerel*[5pt]{\big#1}{%
	\ensurestackMath{\addstackgap[1.1pt]{\big#1}}}}

\newcommand\blambda{\boldsymbol{\lambda}}
\newcommand\bsd{\boldsymbol{d}}
\newcommand\bsm{\boldsymbol{m}}

\newcommand\close{{\ensuremath\mathsf{close}}}
\newcommand\far{{\ensuremath\mathsf{far}}}
\newcommand\Covnm{\mathrm{Cov}_{n,m}}
\newcommand\cb{{\ensuremath\mathsf{cb}}}

\newcommand\bt{{\ensuremath\mathsf{btw}}}
\newcommand\sparse{{\ensuremath\mathsf{sp}}}
\newcommand\low{{\ensuremath\mathsf{lo}}}
\newcommand\high{{\ensuremath\mathsf{hi}}}
\newcommand\High{{\ensuremath\mathsf{High}}}

\DeclareMathOperator{\argmax}{argmax}

\DeclareMathOperator{\image}{image}
\DeclareMathOperator{\ind}{ind}
\DeclareMathOperator{\homind}{hom_{\ind}}
\DeclareMathOperator{\Cov}{Cov}
\DeclareMathOperator{\extremal}{ex}
\DeclareMathOperator{\rand}{rand}

\DeclarePairedDelimiter\ceil{\lceil}{\rceil}
\DeclarePairedDelimiter\floor{\lfloor}{\rfloor}

\DeclarePairedDelimiter\norm{\lVert}{\rVert}
\renewcommand\geq{\geqslant}
\renewcommand\leq{\leqslant}

\title[The typical structure of dense claw-free graphs]{The typical structure of dense claw-free graphs}

\author{Will Perkins}
\author{Sam van der Poel}
\address{Georgia Institute of Technology}
\email{wperkins3@gatech.edu, samvanderpoel@gatech.edu}

\date{\today}

\begin{document}

\begin{abstract}
We analyze the asymptotic number and typical structure of claw-free graphs at constant edge densities. The first of our main results is a formula for the asymptotics of the logarithm of the number of claw-free graphs of edge density $\gamma \in (0,1)$. We show that the problem exhibits a second-order phase transition at edge density $\gamma^\ast=\frac{5-\sqrt{5}}{4}$.  The asymptotic formula arises by solving a variational problem over graphons. For  $\gamma\geq\gamma^\ast$ there is a unique optimal graphon, while for  $\gamma<\gamma^\ast$ there is an infinite set of optimal graphons. By analyzing more detailed structure, we prove that for  $\gamma<\gamma^\ast$, there is in fact a unique  graphon $W$ such that almost all claw-free graphs at edge density $\gamma$ are close in cut metric to $W$.

We also analyze the probability of claw-freeness in the \erdosrenyi{} random graph $G(n,p)$ for constant $p$, obtaining a formula for the large-deviation rate function for claw-freeness.  In this case, the problem exhibits a first-order phase transition at $p^\ast=\frac{3-\sqrt{5}}{2}$, separating distinct structural regimes.  At the critical point $p^\ast$, the corresponding graphon variational problem has infinitely many solutions, and we again pinpoint a unique optimal graphon that describes the typical structure of $G(n,p^\ast)$ conditioned on being claw-free.
\end{abstract}

\maketitle

\section{Introduction}
\label{sec:intro}
A graph is \emph{claw-free} if none of its induced subgraphs is isomorphic to $K_{1,3}$. In this paper we study the \emph{evolution of structure} of dense claw-free graphs: how does the  structure of a typical claw-free graph change as its edge density varies? First, we derive formulas for the asymptotics of the logarithm of the number of claw-free graphs at constant edge densities and the asymptotics of the logarithm of the probability that the \erdosrenyi{} random graph $G(n,p)$ is claw-free for constant $p$. The formulas arise from variational problems over the space of graphons, and we show that the solutions to these problems exhibit phase transitions. The optimal graphons provide a rough structural description of claw-free graphs, and after a more detailed analysis, we make finer structural statements and derive more precise asymptotic formulas.

A classical example of asymptotic enumeration and typical structure of constrained graphs is the case of triangle-free graphs. \citeauthor{mantel1907vraagstuk} proved in \citeyear{mantel1907vraagstuk} that the complete balanced bipartite graph is the extremal triangle-free graph, i.e. the one with the most edges \cite{mantel1907vraagstuk}. In fact, the logarithm of the number of subgraphs of the extremal graph is asymptotic to the logarithm of the number of triangle-free graphs. \citeauthor{erdos1976enumeration} proved the much stronger result that almost all\footnote{All except a fraction that tends to zero as $n\to\infty$.} triangle-free graphs are bipartite \cite{erdos1976enumeration}; and so by asymptotically enumerating bipartite graphs, one obtains an asymptotic formula for the  number of triangle-free graphs. Many works have also studied how the structure of a typical triangle-free graph depends on the edge density \cite{promel1996asymptotic,luczak2000triangle,osthus2003densities,balogh2016typical}. Related questions ask for the probability $G(n,p)$ is triangle-free and the typical structure of $G(n,p)$ conditioned on being triangle-free, the answers to which are known for constant $p$ and some regimes of $p=o(1)$ \cite{janson1990exponential,promel1996asymptotic,luczak2000triangle}.

To derive typical structure results, it has proven useful to study graphons, the limit objects of large dense graphs \cite{lovasz2006limits,borgs2008convergent}. \citeauthor{chatterjee2011large} proved an important result in this regard: a large deviation principle for the \erdosrenyi{} random graph in which the rate function is given by a variational problem over graphons \cite{chatterjee2011large}. The optima and optimizers of this variational problem are informative about the number and structure of graphs in a broad range of properties defined by subgraph density constraints. An important example, and one that predates the study of graphons, is the case of $H$-free graphs where $H$ is a constant-sized subgraph and $\chi(H)=r+1\geq3$. A typical $H$-free graph of edge density $\gamma\in\big(0,\frac{r-1}{r}\big)$ is within $o(n^2)$ edit distance of a nearly balanced $r$-partite graph with edge density approximately $\frac{r}{r-1}\gamma$ between color classes (see \cite{bottcher2012perfect} for a proof based on the earlier work of \citeauthor{erdos1986asymptotic} \cite{erdos1986asymptotic}), and note that the range $\gamma>\frac{r-1}{r}$ is excluded by the \erdos--Stone theorem. In the language of graphons, the variational problem over $H$-free graphons with edge density $\gamma$ has a unique\footnote{Uniqueness here is meant up to equivalence of graphons, which is discussed in more detail in \Cref{subsec:mainresults}.} optimizer with an $r\times r$ block structure, density $\frac{r}{r-1}\gamma$ in off-diagonal blocks, and density 0 in diagonal blocks. Notably, the solution to this variational problem maintains this structure and the optimum varies smoothly with $\gamma$; that is, there is no phase transition.

More recently, the lower- and upper-tail large deviation problems have received much attention: for $\delta>0$, what is the probability that the number of copies of a fixed graph $H$ in $G(n,p)$ is at most $1-\delta$ or at least $1+\delta$ times the expected number? For constant $p$, the logarithmic asymptotics of the lower- and upper-tail probabilities, i.e. the large deviation rate functions, are given by the variational problem of \citeauthor{chatterjee2011large}. For the upper-tail variational problem, \citeauthor{lubetzky2015replica} identified the `replica symmetric' regime for cliques and more generally $d$-regular $H$; that is, the set of parameters for which the optimizing graphon is constant~\cite{lubetzky2015replica} (see also work on the variational problem for $p =o(1)$~\cite{lubetzky2017variational}). For the lower-tail variational problem for triangles, \citeauthor{zhao2017lower} showed that for small $\delta$ the optimal graphon is constant while for larger $\delta$ it is not~\cite{zhao2017lower}.  In both cases, these solutions show that a phase transition (in the sense of a non-analyticity of the rate function)  occurs as $\delta$ (or $p$) varies. A related line of research studies the variational problem with fixed subgraph densities; for example, \citeauthor{radin2015singularities} proved that edge- and triangle-constrained graphs exhibit a phase transition and fully characterized the optimizers for certain positive edge and triangle density pairs \cite{radin2013phase,radin2015singularities}. Some works have shown that in the sparse case $p=o(1)$, the lower- and upper-tail large deviation problems still reduce to certain variational problems \cite{chatterjee2016nonlinear,eldan2018gaussian,cook2020large,augeri2020nonlinear,harel2022upper,kozma2023lower}; in this case the upper-tail problem is completely solved for cliques~\cite{lubetzky2017variational}, while the lower-tail problem is still widely open~\cite{zhao2017lower}.

In this paper, we fully characterize the entropy density of claw-free graphs for constant edge densities and the large deviation rate function for claw-freeness in $G(n,p)$ for constant $p$.  In contrast to the problem of $H$-freeness described above, phase transitions occur in both problems; in the case of entropy density the phase transition is second-order, while in the case of the rate function, the phase transition is first-order.  We obtain these results by solving variational problems  over claw-free graphons.  In both cases, for certain parameter regimes there are infinitely many distinct optimizers, yet in all cases can still identify the unique typical structure (in the sense of cut metric) of claw-free graphs using a finer analysis of the counts and probabilities. These  more fine-grained results advance the research direction set in motion by \citeauthor{promel1991excluding} of proving strong typical structure results for induced-$H$-free graphs \cite{promel1991excluding}.

\subsection{Main Results}\label{subsec:mainresults}
Let $\cC(n)$ be the set of claw-free graphs on $n$ vertices, and let $\cC(n,m)$ be the set of claw-free graphs on $n$ vertices and $m$ edges. The binary entropy $H:[0,1]\to\R$ is the function $H(x)=-x\log_2(x)-(1-x)\log_2(1-x)$ with the convention $0\log_20=0$. Define $r^\ast:[0,1]\to\R$ by
\begin{equation}
r^\ast(\gamma)=\begin{cases}\frac{5+\sqrt{5}}{10}H\pig(\frac{3-\sqrt{5}}{2}\pig)\gamma&\gamma\in\big[0,\frac{5-\sqrt{5}}{4}\big)\\[5pt]\frac{1}{2}H(2\gamma-1)&\gamma\in\big[\frac{5-\sqrt{5}}{4},1\big]\end{cases}\,.\label{eqn:ratefunction}
\end{equation}
If $\cP(n)$ is a set of graphs on $n$ vertices, defined for all $n\in\N$, then the \emph{entropy density} of $\cP(n)$ is defined to be the number $\lim_{n\to\infty}\binom{n}{2}^{-1}\log_2|\cP(n)|$, provided the limit exists\footnote{The limit exists in particular for hereditary properties \cite{alekseev1993entropy,bollobas1997hereditary} and a wide range of properties defined by subgraph density constraints \cite{chatterjee2011large}.}. Our first main result states that $r^\ast(\gamma)$ is the entropy density of $\cC(n,m)$.

\begin{theorem}\label{thm:CnmRateFunction}
Let $\gamma\in(0,1)$ be a fixed constant and let $n,m\in\N$. If $m\sim\gamma\binom{n}{2}$ then
$$\lim_{n\to\infty}\frac{1}{\binom{n}{2}}\log_2|\cC(n,m)|=r^\ast(\gamma)\,.$$
\end{theorem}

The graph of $r^\ast(\gamma)$ is depicted in the first plot of \Cref{fig:rate-functions}. Note $r^\ast(3/4)=1/2$ is the entropy density of $\cC(n)$ (known by prior work of \citeauthor{promel1992excluding} \cite{promel1992excluding}); a theorem of \citeauthor{balogh2011excluding} \cite{balogh2011excluding} implies almost all claw-free graphs are co-bipartite, so the value $r^\ast(3/4)=1/2$ reflects the fact that almost all co-bipartite graphs have edge density approximately $3/4$. 

The entropy density $r^\ast(\gamma)$ has a continuous first derivative and discontinuous second derivative, so there is a second-order phase transition at $\gamma^\ast=\frac{5-\sqrt{5}}{4}$ (see \cite{jaeger1998ehrenfest} for background on the classification of phase transitions).

Next, define $r_\ast(p):[0,1]\to\R$ by
\begin{equation}
r_\ast(p)=\begin{cases}-\log_2(1-p)&p\in\big[0,\frac{3-\sqrt{5}}{2}\big)\\[5pt]-\frac{1}{2}\log_2p&p\in\big[\frac{3-\sqrt{5}}{2},1\big]\end{cases}\,.\label{eqn:ratefunctionGNP}
\end{equation}
Our second main result states that $r_\ast(p)$ is the large deviation rate function for the event that $G(n,p)$ is claw-free.

\begin{theorem}\label{thm:LDPforGNP}
If $p\in(0,1)$ is a fixed constant, $n\in\N$, and $G\sim G(n,p)$ then 
$$\lim_{n\to\infty}\frac{1}{\binom{n}{2}}\log_2\bbP\{G\in\cC(n)\}=-r_\ast(p)\,.$$
\end{theorem}

The graph of $r_\ast(p)$ is depicted in the second plot of \Cref{fig:rate-functions}. The case $p=1/2$ corresponds with counting claw-free graphs, so the rate function recovers the known entropy density of $\cC(n)$ as a special case. Since $r_\ast(p)$ has a discontinuous first derivative, there is a first-order phase transition at $p^\ast=\frac{3-\sqrt{5}}{2}$. As we will show in the structural results below,  the first-order phase transition is also reflected in a discontinuity at $p^\ast$ in the typical edge density of the conditioned \erdosrenyi{} random graph. We also note that the rate function $r_\ast(p)$ is non-monotone, which differs from what one observes for  monotone graph properties.

\begin{figure}
	\begin{center}
	\pgfplotsset{compat=1.17}
\pgfmathsetmacro{\a}{(5-sqrt(5))/4}
\pgfmathsetmacro{\b}{(3-sqrt(5))/2}
\pgfmathsetmacro{\c}{(5+sqrt(5))/10}
\pgfmathsetmacro{\Hb}{0.9594187} 
\newcommand{\scale}{1.0}

\makebox[\textwidth]{%
\begin{tikzpicture}[scale=\scale]
	\pgfplotsset{set layers}
	
	\begin{groupplot}[
		group style={
			group size=2 by 1,
			horizontal sep=2cm,
			vertical sep=0cm
		},
		width=7cm,
		height=5cm,
		grid=both,
		major grid style={solid,gray!30},
		every axis plot/.append style={thick,black},
		xticklabel style={/pgf/number format/fixed},
		yticklabel style={/pgf/number format/fixed}
	]

	\nextgroupplot[
		xmin=0, xmax=1.05,
		ymin=0, ymax=0.65,
		xlabel={\(\gamma\)},
		axis x line=bottom,
		axis y line=middle,
		xtick={0,0.2,0.4,0.6,0.8,1},
		ytick={0,0.2,0.4,0.6},
		label style={font=\small},
		tick label style={font=\small},
		line cap=round,
	]

	\addplot[smooth,domain=0:\a,samples=50] {\c * \Hb * x};
	\addplot[smooth,domain=\a:1,samples=100] {0.5 * (-(2*x-1)*log2(2*x-1) - (1-(2*x-1))*log2(1-(2*x-1)))};
	\node[right,fill=white] at (rel axis cs:0.025,0.925) {\small$r^\ast(\gamma)$};

	\nextgroupplot[
		xmin=0, xmax=1.05,
		ymin=0, ymax=0.8125,
		xlabel={$p$},
		axis x line=bottom,
		axis y line=middle,
		xtick={0,0.2,0.4,0.6,0.8,1},
		ytick={0,0.25,0.5,0.75},
		xlabel near ticks,
		label style={font=\small},
		tick label style={font=\small},
		line cap=round,
	]

	\addplot[smooth,domain=0:\b,samples=50] {-log2(1-x)};
	\addplot[smooth,domain=\b:1,samples=100] {-0.5 * log2(x)};
	\node[right,fill=white] at (rel axis cs:0.025,0.925) {\small$r_\ast(p)$};

	\end{groupplot}
  
\end{tikzpicture}
}
	\end{center}
	\vspace{-2mm}
	\caption{The entropy density $r^\ast(\gamma)$ and rate function $r_\ast(p)$.}
	\label{fig:rate-functions}
\end{figure}
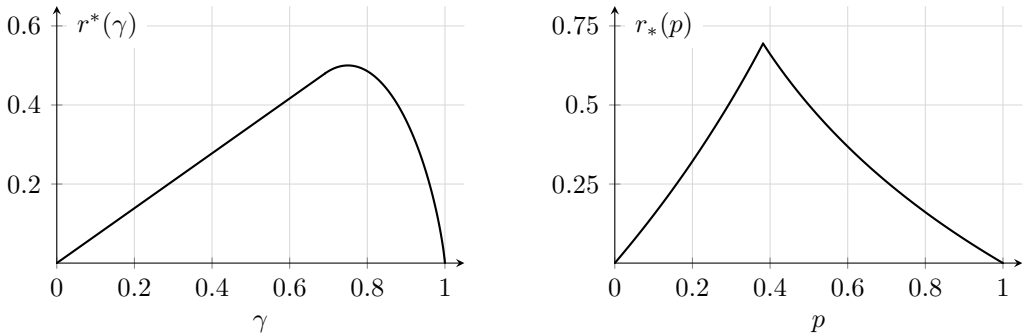

We briefly provide some intuition for the forms of the functions $r^\ast(\gamma)$ and $r_\ast(p)$. For  $\gamma\in\big[\frac{5-\sqrt{5}}{4},1\big]$, the quantity $r^\ast(\gamma)$ is entropy density of the set $\cB_c(n,m)$ of co-bipartite graphs on $n$ vertices and $m$ edges. For  $\gamma\in\big[0,\frac{5-\sqrt{5}}{4}\big)$, the quantity $r^\ast(\gamma)$ is the entropy density of the set of graphs $G$ on $n$ vertices and $m$ edges that are the disjoint union of a co-bipartite graph with parts of sizes approximately $\big(\frac{5+\sqrt{5}}{20}\gamma\big)^{1/2}n$ and an empty graph on the remaining vertices. The density of edges between the two cliques in such a graph $G$ is approximately $\frac{3-\sqrt{5}}{2}$. For  $p\in\big[\frac{3-\sqrt{5}}{2},1\big]$, the quantity $r_\ast(p)$ is the large deviation rate function for the event that $G(n,p)$ is co-bipartite, and for  $p\in\big[0,\frac{3-\sqrt{5}}{2}\big)$, $r_\ast(p)$ is the rate function for the event that $G(n,p)$ has no edges.

The functions $r^\ast(\gamma)$ and $r_\ast(p)$ are the solutions to  variational problems over graphons, which are measurable functions $W:[0,1]^2\to[0,1]$ satisfying $W(x,y)=W(y,x)$ for all $x,\,y$. The set of all graphons is denoted $\cW$. Graphons are representations of \emph{graph limits}, which are the limits of large graphs with respect to the \emph{cut metric} (these terms are introduced formally in \Cref{sec:prelim}; see~\cite{lovasz2012large} for more background). We say that two graphons are \emph{equivalent} if they represent the same graph limit\footnote{\citeauthor{borgs2010moments} proved \cite[Corollary~2.2]{borgs2010moments} that graphons $W_1$ and $W_2$ are equivalent if and only if there exist measure-preserving maps $\sigma_i:[0,1]\to[0,1]$ such that $W_1(\sigma_1(x),\sigma_1(y))=W_2(\sigma_2(x),\sigma_2(y))$ almost everywhere. See also \cite[Theorem~13.10]{lovasz2012large}.\label{ftnt:GraphonUniquness}}.

For  $\gamma\in(0,1)$, define the following variational problem over graphons $W\in\cW$ :
\[\arraycolsep=3pt
\begin{array}{lll}
\phi(\gamma) \hspace{0.5mm}:= & \sup & \displaystyle\int_{[0,1]^2}H(W(x,y)) \, dx \, dy \\[16pt]
&\text{s.t.} & \displaystyle\int_{[0,1]^2}W(x,y) \, dx \, dy=\gamma \,, \\[12pt]
&& \displaystyle\int_{[0,1]^4}\Bigg(\prod_{i=2}^4W(x_1,x_i)\Bigg)\Bigg(\prod_{2\leq i<j\leq4}(1-W(x_i,x_j))\Bigg) \, dx_1 \, dx_2 \, dx_3 \, dx_4  =0 \,.
\end{array}\]
The first constraint asserts that $W$ has edge density $\gamma$. The second constraint asserts that $W$ has zero claw-density; for the reader familiar with the $W$-random graph $G(n,W)$, defined formally later, this is equivalent to saying that $G(n,W)$ is claw-free with probability 1. Our next result states the solution to the variational problem $\phi(\gamma)$ and describes the number of graphons achieving the supremum.

\begin{theorem}\label{thm:FixedDensityVarProblem}
For all $\gamma\in(0,1)$ we have $\phi(\gamma)=r^\ast(\gamma)$. For all $\gamma\in\big[\frac{5-\sqrt{5}}{4},1\big)$, there is a unique graphon (up to equivalence) achieving $\phi(\gamma)$, and for all $\gamma\in\big(0,\frac{5-\sqrt{5}}{4}\big)$, there is an infinite number of distinct (nonequivalent) graphons achieving $\phi(\gamma)$.
\end{theorem}

By a slight extension of the Chatterjee--Varadhan framework to fixed-density induced-$H$-free graphons, \Cref{lemma:FnmGraphonOptAsymps} proves $\phi(\gamma)$ is the entropy density of $\cC(n,m)$. Thus from \Cref{thm:FixedDensityVarProblem} and \Cref{lemma:FnmGraphonOptAsymps} we can deduce \Cref{thm:CnmRateFunction}.

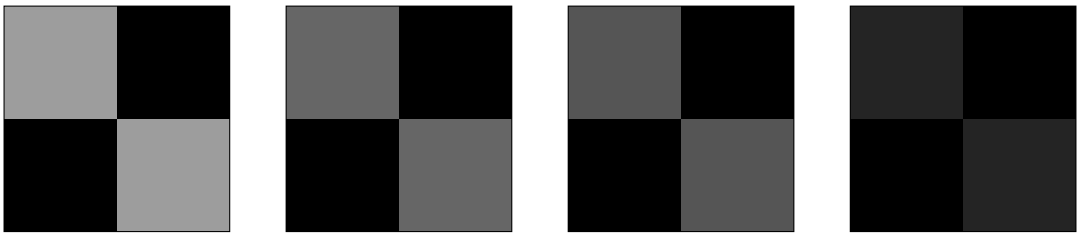
\begin{figure}
	\begin{center}
	\pgfmathsetmacro{\density}{0.125}
\pgfmathsetmacro{\RHS}{2*\density*(1+sqrt(5)/5)}

\begin{tikzpicture}[baseline={(0,0)},scale=3]
	\tikzset{cobipartite/.code={
		\pgfkeys{/cobipartite/.cd,
			x/.initial=0,
			y/.initial=0,
			intensity/.initial=0
		}    
        \pgfkeys{/cobipartite/.cd, #1}
        \def\xpos{\pgfkeysvalueof{/cobipartite/x}}
        \def\ypos{\pgfkeysvalueof{/cobipartite/y}}
        \def\intensity{\pgfkeysvalueof{/cobipartite/intensity}}
        
        \begin{scope}[shift={(\xpos,\ypos)}, scale=1]
            \fill[black] (0,0) -- (0.5,0) -- (0.5,0.5) -- (0,0.5) -- cycle;
            \fill[black!\intensity] (0.5,0) -- (1,0) -- (1,0.5) -- (0.5,0.5) -- cycle;
            \fill[black] (0.5,0.5) -- (1,0.5) -- (1,1) -- (0.5,1) -- cycle;
            \fill[black!\intensity] (0,0.5) -- (0.5,0.5) -- (0.5,1) -- (0,1) -- cycle;
        \end{scope}
    }}
    
    \begin{scope}[shift={(-2.375,0)}]
    	\pgfmathsetmacro{\crossdensity}{(3-sqrt(5))/2}
		\pgfmathsetmacro{\rho}{100*\crossdensity}
		\draw[cobipartite={x=0, y=0, intensity=\rho}];
		\draw (0,0) rectangle (1,1);
	\end{scope}
	
	\begin{scope}[shift={(-1.125,0)}]
		\pgfmathsetmacro{\crossdensity}{3/5}
		\pgfmathsetmacro{\rho}{100*\crossdensity}
		\draw[cobipartite={x=0, y=0, intensity=\rho}];
		\draw (0,0) rectangle (1,1);
	\end{scope}
	
	\begin{scope}[shift={(0.125,0)}]
		\pgfmathsetmacro{\crossdensity}{2/3}
		\pgfmathsetmacro{\rho}{100*\crossdensity}
		\draw[cobipartite={x=0, y=0, intensity=\rho}];
		\draw (0,0) rectangle (1,1);
	\end{scope}
	
	\begin{scope}[shift={(1.375,0)}]
		\pgfmathsetmacro{\crossdensity}{6/7}
		\pgfmathsetmacro{\rho}{100*\crossdensity}
		\draw[cobipartite={x=0, y=0, intensity=\rho}];
		\draw (0,0) rectangle (1,1);
	\end{scope}
	
\end{tikzpicture}
	\end{center}
	\caption{Unique optimizers of the variational problem $\phi(\gamma)$ at edge densities $\gamma=\frac{5-\sqrt{5}}{4}$, $\frac{4}{5}$, $\frac{5}{6}$, and $\frac{13}{14}$. Each gray area represents a density in $\big[\frac{3-\sqrt{5}}{2},1\big]$ and black areas represent density 1.}
	\label{fig:graphons-super}
\end{figure}

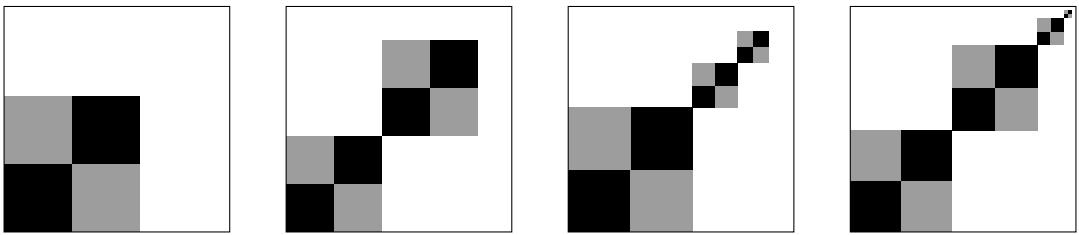
\begin{figure}
	\begin{center}
	\pgfmathsetmacro{\density}{0.125}
\pgfmathsetmacro{\crossdensity}{(3-sqrt(5))/2}
\pgfmathsetmacro{\intensity}{100*\crossdensity}
\pgfmathsetmacro{\RHS}{2*\density*(1+sqrt(5)/5)}

\begin{tikzpicture}[baseline={(0,0)},scale=3]
	\tikzset{cobipartite/.code={
		\pgfkeys{/cobipartite/.cd,
			x/.initial=0,
			y/.initial=0,
			scale/.initial=1
		}    
        \pgfkeys{/cobipartite/.cd, #1}
        \def\xpos{\pgfkeysvalueof{/cobipartite/x}}
        \def\ypos{\pgfkeysvalueof{/cobipartite/y}}
        \def\scale{\pgfkeysvalueof{/cobipartite/scale}}
        
        \begin{scope}[shift={(\xpos,\ypos)}, scale=\scale]
            \fill[black] (0,0) -- (0.5,0) -- (0.5,0.5) -- (0,0.5) -- cycle;
            \fill[black!\intensity] (0.5,0) -- (1,0) -- (1,0.5) -- (0.5,0.5) -- cycle;
            \fill[black] (0.5,0.5) -- (1,0.5) -- (1,1) -- (0.5,1) -- cycle;
            \fill[black!\intensity] (0,0.5) -- (0.5,0.5) -- (0.5,1) -- (0,1) -- cycle;
        \end{scope}
    }}
    
	\begin{scope}[shift={(-2.375,0)}]
		\draw[cobipartite={x=0, y=0, scale=sqrt(\RHS)}];
		\draw (0,0) rectangle (1,1);
	\end{scope}
	
	\begin{scope}[shift={(-1.125,0)}]
		\pgfmathsetmacro{\lambdaOne}{sqrt(\RHS/2)}
		\pgfmathsetmacro{\lambdaTwo}{sqrt(\RHS-pow(\lambdaOne,2))}
		\draw[cobipartite={x=0, y=0, scale=\lambdaOne}];
		\draw[cobipartite={x=\lambdaOne, y=\lambdaOne, scale=\lambdaTwo}];
		\draw (0,0) rectangle (1,1);
	\end{scope}
	
	\begin{scope}[shift={(0.125,0)}]
		\pgfmathsetmacro{\lambdaOne}{0.55}
		\pgfmathsetmacro{\lambdaTwo}{0.2}
		\pgfmathsetmacro{\lambdaThree}{sqrt(\RHS-pow(\lambdaOne,2)-pow(\lambdaTwo,2))}
		\draw[cobipartite={x=0, y=0, scale=\lambdaOne}];
		\draw[cobipartite={x=\lambdaOne, y=\lambdaOne, scale=\lambdaTwo}];
		\draw[cobipartite={x=\lambdaOne+\lambdaTwo, y=\lambdaOne+\lambdaTwo, scale=\lambdaThree}];
		\draw (0,0) rectangle (1,1);
	\end{scope}
	
	\begin{scope}[shift={(1.375,0)}]
		\pgfmathsetmacro{\lambdaOne}{0.45}
		\pgfmathsetmacro{\lambdaTwo}{0.38}
		\pgfmathsetmacro{\lambdaThree}{0.117}
		\pgfmathsetmacro{\lambdaFour}{sqrt(\RHS-pow(\lambdaOne,2)-pow(\lambdaTwo,2)-pow(\lambdaThree,2))}
		\draw[cobipartite={x=0, y=0, scale=\lambdaOne}];
		\draw[cobipartite={x=\lambdaOne, y=\lambdaOne, scale=\lambdaTwo}];
		\draw[cobipartite={x=\lambdaOne+\lambdaTwo, y=\lambdaOne+\lambdaTwo, scale=\lambdaThree}];
		\draw[cobipartite={x=\lambdaOne+\lambdaTwo+\lambdaThree, y=\lambdaOne+\lambdaTwo+\lambdaThree, scale=\lambdaFour}];
		\draw (0,0) rectangle (1,1);
	\end{scope}

\end{tikzpicture}
	\end{center}
	\caption{Examples of optimizers of the variational problem $\phi(\gamma)$ at edge density $\gamma=\frac{1}{4}$. The white, gray, and black areas represent densities $0$, $\frac{3-\sqrt{5}}{2}$, and $1$, respectively.}
	\label{fig:graphons-sub}
\end{figure}

In the range $\gamma\in\big[\frac{5-\sqrt{5}}{4},1\big)$, the unique graphon achieving $\phi(\gamma)$ has the structure of a co-bipartite graph; see \Cref{fig:graphons-super} for a depiction and \Cref{eqn:deffstarlambdagraphon} for a formal definition. In the range $\gamma\in\big(0,\frac{5-\sqrt{5}}{4}\big)$, the optimal graphons have the structure of a graph that is the vertex-disjoint union of several co-bipartite graphs and a graph with $o(n^2)$ edges; see \Cref{fig:graphons-sub} for a depiction and \Cref{eqn:flambdaGraphonDef} for a formal definition of these graphons.

For all $p\in(0,1)$ define the relative entropy
\begin{equation}
I_p(x):=x\log_2\frac{p}{x}+(1-x)\log_2\frac{1-p}{1-x}=H(x)+x\log_2\left(\frac{p}{1-p}\right)+\log_2(1-p)\,,\label{eqn:defIpx}
\end{equation}
and define the variational problem
\[\arraycolsep=3pt
\begin{array}{lll}
\psi(p) \hspace{0.5mm}:= & \sup & \displaystyle\int_{[0,1]^2}I_p(W(x,y))\,dx\,dy \\[16pt]
& \text{s.t.} & \displaystyle\int_{[0,1]^4}\Bigg(\prod_{i=2}^4W(x_1,x_i)\Bigg)\Bigg(\prod_{2\leq i<j\leq4}(1-W(x_i,x_j))\Bigg)\,dx_1\,dx_2\,dx_3\,dx_4=0 \,.
\end{array}\]
Our next result states that $r_\ast$ is the solution to the variational problem $\psi$, and describes the number of graphons achieving the supremum.

\begin{theorem}\label{thm:GNPVarProblem}
For all $p\in(0,1)$ we have $\psi(p)=-r_\ast(p)$. For all $p\in\big(0,\frac{3-\sqrt{5}}{2}\big)$, there is a unique graphon (up to equivalence) achieving $\psi(p)$, for $p=\frac{3-\sqrt{5}}{2}$, there is an infinite number of distinct (nonequivalent) graphons achieving $\psi(p)$, and for all $p\in\big(\frac{3-\sqrt{5}}{2},1\big)$, there is a unique graphon (up to equivalence) achieving $\psi(p)$.
\end{theorem}

In the range $p\in\big(\frac{3-\sqrt{5}}{2},1\big)$, the unique graphon achieving $\psi(p)$ resembles the adjacency of a large co-bipartite graph with edge density $(1+p)/2$ (see \Cref{eqn:CpStarFormalDef} for a formal definition). In the range $p\in\big(0,\frac{3-\sqrt{5}}{2}\big)$, the zero graphon is the unique optimal graphon. Finally, for $p=\frac{3-\sqrt{5}}{2}$, the optimal graphons have various edge densities and resemble those in \Cref{fig:graphons-sub}.

\citeauthor{chatterjee2011large} showed that in addition to providing the rate function for $G(n,p)$, optimal graphons also give a rough structural description of the conditional \erdosrenyi{} random graph. By a slight modification of their proof, we show that if $m\sim\gamma\binom{n}{2}$ then almost every $G\in\cC(n,m)$ is close in cut metric of an optimal graphon. Since the optimal graphons achieving $\phi(\gamma)$ have a simple description---they have the structure of a disjoint union of co-bipartite graphs and a sparse graph---it is desirable to estimate the number of claw-free graphs that have precisely this simple structure. For example, does a statement analogous to the \erdos--Kleitman--Rothschild theorem hold for claw-free graphs with edge density greater than $\frac{5-\sqrt{5}}{4}$? The next theorem answers this affirmatively and describes structure in the subcritical regime.

Let $\cB_c(n,m)$ denote the set of co-bipartite graphs on $n$ vertices and $m$ edges. For functions $f,g:\N\to\R$, the notation $f(n)\sim g(n)$ means $\lim_{n\to\infty}f(n)/g(n)=1$.

\begin{theorem}\label{thm:almostallCB}
Let $\gamma\in(0,1)$ be a fixed constant, let $n,m\in\N$, and assume $m\sim\gamma\binom{n}{2}$. Then the following hold.
\begin{enumerate}[(\emph{\roman*})]
	\item\label{thm:almostallCBsupercritical} If $\gamma\in\big(\frac{5-\sqrt{5}}{4},1\big)$ then almost all claw-free graphs on $n$ vertices and $m$ edges are co-bipartite, that is,
	$$|\cC(n,m)|\sim|\cB_c(n,m)|\sim\left(\frac{r+1}{2}+\sum_{k=1}^\infty(2\gamma-1)^{k^2+rk}\right)\binom{n}{\floor{n/2}}\binom{\floor{n^2/4}}{m-\binom{\floor{n/2}}{2}-\binom{\ceil{n/2}}{2}}\,,$$
	where $r=n\bmod{2}$.
	\item\label{thm:almostallCBcritical} If $\gamma=\frac{5-\sqrt{5}}{4}$ and $m=\floor{\gamma \binom{n}{2}}$ then almost every $G\in\cC(n,m)$ is the vertex-disjoint union of a co-bipartite graph and a graph with at most $\log n$ vertices.
	\item\label{thm:almostallCBsubcritical} If $\gamma\in\big(0,\frac{5-\sqrt{5}}{4}\big)$ then almost every $G\in\cC(n,m)$ is the vertex-disjoint union of two graphs $G_1\cup G_2$ where $G_1$ is a co-bipartite graph with parts of sizes
	$$\Bigg(\sqrt{\frac{5+\smash[t]{\sqrt{5}}\rule{0pt}{9.25pt}}{20}\gamma}+o(1)\Bigg)n$$
	and $G_2$ has $\big(1-\big(\frac{5+\sqrt{5}}{5}\gamma\big)^{1/2}+o(1)\big)n$ vertices and $\Omega(n)\leq e(G_2)\leq o(n^2)$ edges.
\end{enumerate}
\end{theorem}

The asymptotic formula for $|\cB_c(n,m)|$ in \Cref{thm:almostallCB} \ref{thm:almostallCBsupercritical} is obtained by a standard calculation, included  in \Cref{sec:counting}.

In the subcritical regime $\gamma\in\big(0,\frac{5-\sqrt{5}}{4}\big)$, \Cref{thm:almostallCB} \ref{thm:almostallCBsubcritical} implies there is a unique graphon $W^\ast$ (up to equivalence) such that almost every $G\in\cC(n,m)$ is within $o(1)$ in cut metric of $W^\ast$. In fact, in \Cref{sec:subcritical} we show that the number of claw-free graphs with edge density $\gamma\in\big(0,\frac{5-\sqrt{5}}{4}\big)$ is at least $n^{\Theta(n)}$ times larger than the number of graphs that are the disjoint union of two or more co-bipartite graphs each having $\Theta(n)$ vertices and a sparse graph, which are graphs close in cut metric to an optimal graphon besides $W^\ast$. For example, among the optimal graphons in \Cref{fig:graphons-sub}, there are at least $n^{\Theta(n)}$ times as many claw-free graphs of edge density $1/4$ that are close to the leftmost graphon as there are to any of the other three graphons.

Our last main result describes the typical structure of $G(n,p)$ conditioned on being claw-free.

\begin{theorem}\label{thm:GnpTypicalStructure}
Let $p\in(0,1)$ be a fixed constant and let $G$ be the \erdosrenyi{} random graph $G(n,p)$ conditioned on being claw-free. Then the following hold.
\begin{enumerate}[(\emph{\roman*})]
	\item If $p\in\big(\frac{3-\sqrt{5}}{2},1\big)$ then $G$ is co-bipartite with high probability.
	\item If $p\in\big(0,\frac{3-\sqrt{5}}{2}\big]$ then $G$ has $o(n^2)$ edges with high probability.
\end{enumerate}
\end{theorem}

The second part of the theorem shows that at the critical point $p^\ast=\frac{3-\sqrt{5}}{2}$, although there are infinitely many optimal graphons as proven in \Cref{thm:GNPVarProblem}, the conditional \erdosrenyi{} converges in cut metric to the all-zero graphon in probability.

\subsection{Related Work}\label{subsec:RelatedWork}
The study of asymptotic enumeration and typical structure has a long history in graph theory. Beyond the many works on triangle-free graphs mentioned above, there has been much research on analogous questions about $H$-free graphs for general graphs $H$ with $r+1=\chi(H)\geq3$. Using Szemer\'edi's regularity lemma, \citeauthor{erdos1986asymptotic} proved that the log-number of $H$-free graphs is asymptotic to the extremal number $\extremal(n,H)$ \cite{erdos1986asymptotic}. \citeauthor{promel1992asymptotic} proved that almost all $H$-free graphs are $r$-colorable if and only if $H$ is edge-critical (i.e. the removal of an edge reduces the chromatic number) \cite{promel1992asymptotic}, generalizing the \erdos--Kleitman--Rothschild theorem. \citeauthor{balogh2009typical} proved that if $\cF$ is a finite set of graphs and $r+1=\min_{H\in\cF}\chi(H)\geq3$, then almost all $\cF$-free graphs are within $o(n^2)$ edit distance of an $r$-colorable graph \cite{balogh2009typical}.

Several results for $H$-free graphs have induced-$H$-free analogues. For example, \citeauthor{alon2011structure} proved a generalization of the Balogh--Bollob\'as--Simonovits theorem for hereditary properties \cite{alon2011structure}. Induced analogues of the \erdos--Stone theorem and \erdos--Frankl--R\"odl theorem, proven by \citeauthor{promel1993excluding} \cite{promel1992excluding,promel1993excluding}, give an asymptotic formula for the log-number of induced-$H$-free graphs.

The study of typical structure of induced-$H$-free graphs was initiated by \citeauthor{promel1991excluding}, who examined induced-$C_4$-free graphs \cite{promel1991excluding}. Recent works have also examined \emph{sparse} induced-$C_4$-free graphs \cite{kalvari2021typical,morris2024asymmetric}. \citeauthor{balogh2011excluding} proved an induced analogue to \citeauthor{promel1992asymptotic}'s result about edge-critical graphs: they defined a notion of criticality that characterizes the graphs $H$ for which almost all induced-$H$-free graphs have a simple structure \cite{balogh2011excluding}. \citeauthor{mckay2003asymptotic} proved first-order asymptotics for the number of cubic claw-free graphs \cite{mckay2003asymptotic}, which plays an important role in the proof of \Cref{thm:almostallCB} \ref{thm:almostallCBsubcritical}. \citeauthor{bottcher2012perfect} analyzed induced-$C_5$-free graphs, deriving the entropy density and showing it has two critical points \cite{bottcher2012perfect}.

Prior works have also examined the probability that $G(n,p)$ belongs to a general hereditary property $\cH$. Using the extremal theory of 2-colored multigraphs, \citeauthor{marchant2011structure} proved that there is always an ``elementary property'' $\cH'$ contained in $\cH$ such that the rate function for the event $G(n,p)\in\cH'$ coincides with the rate function for the event $G(n,p)\in\cH$ \cite{marchant2011structure}. \citeauthor{marchant2011structure} also calculated the rate function for some specific hereditary properties including the set of induced-$C^\ast_6$-free graphs (where $C^\ast_6$ is the cycle $C_6$ with a diagonal chord). See also the related works \cite{bollobas2000structure,marchant2010extremal}.

Graphons have been used extensively to study graph properties, random graphs, and typical structure. Many works have sought to solve the variational problem over graphons that are limits of constrained graph classes \cite{radin2013phase,radin2014asymptotics,aristoff2015asymptotic,lubetzky2015replica,kenyon2017multipodal,kenyon2017asymptotics,zhao2017lower,zhu2017asymptotic,neeman2023typical}. In this paper, we prove the existence of a phase in which there is more than one optimal graphon, but where the typical structure of the underlying property is still captured by a unique graphon; this phenomenon has previously been observed in string graphs \cite{janson2017string,pach2020almost}.

Besides the topics of this paper, there is broad interest in claw-free graphs in graph theory and theoretical computer science. Historically, interest in claw-free graphs originated with \citeauthor{beineke1968derived}'s characterization of line graphs in terms of forbidden induced subgraphs \cite{beineke1968derived}.
Later, \citeauthor{chudnovsky2005structure} proved a detailed characterization of claw-free graphs through decomposition theorems  \cite{chudnovsky2005structure,chudnovsky2008clawV}, and showed that the independence polynomial of claw-free graphs is real-rooted (extending the Heilmann--Lieb theorem) \cite{chudnovsky2007roots}, which is relevant in statistical physics.
Claw-free graphs also have favorable algorithmic properties as several hard problems (e.g. maximum independent set and approximating the independence polynomial) can be solved in polynomial time in claw-free graphs \cite{minty1980maximal,patel2017deterministic}. The survey \cite{faudree1997claw} describes many other properties of claw-free graphs and the areas in which they appear.

\subsection{Overview of Proofs}
The solutions to the variational problems $\phi$ and $\psi$ are presented in \Cref{sec:GraphonOpt}. The proofs combine a version of the regularity lemma (\Cref{lemma:TypeLemma}) with the solution to an extremal problem over edge colorings (\Cref{sec:edgecolorings}). Every optimizer $W$ in the variational problem $\phi$ is the limit of a sequence of regular partitions of claw-free graphs. The solution to the extremal problem allows us to deduce these regular partitions closely resemble (in a Hamming distance sense) the extremal structure, which can be informally described as a union of co-bipartite graphs and an empty graph. It then follows that $W$ has the same structure subject to an edge density constraint, and examples of these graphons are depicted in \Cref{fig:graphons-super,fig:graphons-sub}.

\Cref{thm:almostallCB} is proven in \Cref{sec:supercritical,sec:subcritical}. The proofs use the Kleitman--Rothschild method, a technique for asymptotic enumeration that was first applied to the enumeration of partially-ordered sets \cite{kleitman1975asymptotic}. In this paper, the method works by defining several subsets of $\cC(n,m)$, showing that the graphs $G\in\cC(n,m)$ not contained in those subsets have a special structure, and proving that the cardinality of those subsets is much smaller than that of $\cC(n,m)$. The first step is to define $\cC_\far\,$, the set of all $G\in\cC(n,m)$ that are far in cut metric from the set of optimal graphons $\cX^\ast_\gamma$ for the variational problem $\phi(\gamma)$. \Cref{prop:CNMtypstrucconddistr} proves that almost all $G\in\cC(n,m)$ are contained in $\cC_\close:=\cC(n,m)\setminus\cC_\far$. In \Cref{lemma:SuperCloseStructure,lemma:WtoWtildeMetrics} we prove that every graph $G\in\cC_\close$ is within $o(n^2)$ edit distance of a graph $G'$ that is the disjoint union of several co-bipartite graphs with parts denoted $\Pi_i=\{A_i,B_i\}$, $i=1,\dots,\ell$, and a sparse graph. The graph $T(G)$ with edge set $E(G)\,\Delta\,E(G')$ is called the \emph{defect graph} since we think of its edges as defects in $G$.

The bulk of the proof of \Cref{thm:almostallCB} deals with classifying defect graphs according to various graph theoretic conditions and bounding the number of graphs $G$ for which a given defect graph $T$ is optimal. The key tool in these steps is the version of Janson's inequality due to \citeauthor{riordan2015janson} (\Cref{thm:JansonsInequality}). The inequality is applied by first setting up a random graph $H$ that is the disjoint union of random co-bipartite graphs on the parts $\Pi_i=\{A_i,B_i\}$, and including the defect graph. We then define a class $\sK$ of subgraphs $C\subseteq K_V$ isomorphic to $K_{1,3}$ and uses Janson's inequality to obtain an exponential tail bound for the probability $H$ contains none of the induced copies in $\sK$. The tail bounds are written independently in \Cref{sec:JansonPenalties}.

To prove \Cref{thm:almostallCB} \ref{thm:almostallCBsubcritical} we first fix an optimizer $W$ for $\phi(\gamma)$ that has more than one co-bipartite block (one can visualize this by inspecting \Cref{fig:graphons-sub}). A detailed analysis shows almost every $G\in\cC(n,m)$ that is close in cut metric to $W$ is in fact a disjoint union of several co-bipartite graphs (where the number of co-bipartite graphs is roughly given by the number of such blocks in $W$) and a sparse graph. If $W^\ast$ is the optimizer with a single co-bipartite block (i.e. the leftmost graphon in \Cref{fig:graphons-sub}), then the graphs $G\in\cC(n,m)$ close to $W^\ast$ in cut metric have at least $cn$ more vertices in the sparse part of the graph than those close to $W$, where $c>0$ is a constant. We can then lower bound the number of claw-free graphs close to $W^\ast$ as follows: for every $G\in\cC(n,m)$ that is close to $W$, replace the co-bipartite subgraphs of $G$ with a single co-bipartite graph similar in size to the co-bipartite block in $W^\ast$; this frees up a set $U$ of at least $cn$ vertices in $G$; finally, place a cubic claw-free graph on $U$. Since the main result of \cite{mckay2003asymptotic} implies the number of cubic claw-free graphs is $n^{\Theta(n)}$, we conclude there are at least $n^{\Theta(n)}$ as many graphs $G\in\cC(n,m)$ that are close to $W^\ast$ in cut metric than that are close to $W$.

\subsection{Future Work}
A natural direction for future research is to examine the typical structure of sparse claw-free graphs, which could give a more precise understanding of the sparse graph $G_2$ in \Cref{thm:almostallCB} \ref{thm:almostallCBsubcritical}. The proof of \Cref{thm:almostallCB} \ref{thm:almostallCBsubcritical} lower-bounds the number of claw-free graphs close to the optimal graphon $W^\ast$ using cubic claw-free graphs, so the following is a pertinent question.

\begin{question}
What are the asymptotics (of the logarithm) of $|\cC(n,cn)|$ for fixed  $c>0$?
\end{question}

The use of an asymmetric container lemma, as done by \citeauthor{morris2024asymmetric} to study sparse induced-$C_4$-free graphs \cite{morris2024asymmetric}, could help analyze sparse claw-free graphs.

Another direction for future research could examine the typical structure of fixed-density induced-$H$-free graphs for other graphs $H$, and analyze whether phase transitions occur in those cases as well.

\begin{question}
For  graphs $H$ with coloring number\footnote{The coloring number was defined by \citeauthor{promel1992excluding} \cite{promel1992excluding} and is the suitable replacement for chromatic number in the induced setting.} at least three, does the variational problem over induced-$H$-free graphons exhibit a phase transition in the edge density? Is there always a phase in which the optimizer is nonunique?
\end{question}

The proofs for typical structure in claw-free graphs suggest similar results could hold for the class of induced-$K_{1,r+1}$-free graphs for $r\geq3$. To establish \Cref{conj:Kr1rplus1}, it may be useful to prove asymptotics for the number of $(r+1)$-regular induced-$K_{1,r+1}$-free graphs, similar to the cubic claw-free case.

\begin{conjecture}\label{conj:Kr1rplus1}
For all $r\geq3$, the class $\cF(n,m)$ of induced-$K_{1,r+1}$-free graphs on $n$ vertices and $m\sim\gamma\binom{n}{2}$ edges exhibits a phase transition in the edge density $\gamma\in(0,1)$. In one phase, almost every $G\in\cF(n,m)$ is co-$r$-partite, and in the other phase, almost every $G\in\cF(n,m)$ is the disjoint union of a co-$r$-partite graph and a sparse graph.
\end{conjecture}

Further directions of research also include solving the lower-tail variational problem for claws, or the fixed (non-zero) claw-density variational problem, which are likely to require new techniques.

\section{Preliminaries}
\label{sec:prelim}
This section describes probabilistic tools and background on graph regularity and graph limits that we will need. Unless stated otherwise, graphs are labeled and simple on a vertex set $V$ of size $n$.

\subsection{Notation}
Let $\cG(n)$ denote the set of graphs on $V$. Let $\cG:=\bigcup_{n\in\N}\cG(n)$ denote the set of all graphs. Let $\cC(n)$ denote the set of claw-free graphs on $n$ vertices and let $\cB(n)$ denote the set of bipartite graphs on $n$ vertices. Let $\cG(n,m)$ denote the set of graphs on $n$ vertices with $m$ edges, and define the sets $\cC(n,m):=\cC(n)\cap\cG(n,m)$ and $\cB(n,m):=\cB(n)\cap\cG(n,m)$. For all $n,m\in\N:=\{1,2,\dots\}$ let $\cU(n)$ denote the set of unlabeled graphs on $n$ vertices. Let $\cU(n,m)$ denote those unlabeled graphs on $n$ vertices and $m$ edges. Let $\cU:=\bigcup_{n\in\N}\cU(n)$ denote the set of all unlabeled graphs.

If $G$ is a graph and $u,v\in V$ are distinct vertices then $uv$ denotes the unordered pair $\{u,v\}$. For a set $S$ and $k\in\N$, let $\binom{S}{k}:=\{T\subseteq S:|T|=k\}$. We denote $v(G):=|V(G)|$ and $e(G):=|E(G)|$. For all $A,B\subseteq V(G)$, denote $E(A):=E(G)\cap\binom{A}{2}$, $e(A):=|E(A)|$, $E(A,B):=\{uv\in E(G):u\in A,v\in B\}$, and $e(A,B):=|E(A,B)|$. The \emph{density} between $A$ and $B$ is the number $d(A,B):=e(A,B)/(|A|\cdot|B|)$. For all $v\in V(G)$, the neighborhood (resp. degree) of $v$ in $A$ is defined by $N(v,A):=N(v)\cap A$ (resp. $d(v,A):=|N(v,A)|$). The complementary neighborhood (resp. degree) of $v$ in $A$ is defined by $\overline{N}(v,A):=A\setminus N(v)$ (resp. $\overline{d}(v,A):=|\overline{N}(v,A)|$). Define $G[A,B]$ to be the graph with edge set $E=\{ab\in E(G):a\in A,\,b\in B\}$ and vertex set $\bigcup E$. For each of the previous definitions, we sometimes use subscripts (e.g. $E_G(A)$, $e_G(A)$, $d_G(A,B)$, etc) to emphasize the underlying graph. The complement of a graph $G$ is denoted $\overline{G}$ or $G^c$. For all $n\in\N$, denote $[n]:=\{1,\dots,n\}$. For functions $f,g:\N\to\R$, the formula $f(n)\ll g(n)$ means $\lim_{n\to\infty}f(n)/g(n)=0$.

\subsection{Probabilistic Preliminaries}
In \Cref{sec:JansonPenalties} we use the following version of Janson's inequality due to \citeauthor{riordan2015janson} \cite{riordan2015janson} to prove exponential penalties for defect edges. The reason we need this general version is that the event of containing an induced $K_{1,3}$ is not edge-monotone, but if the copies of $K_{1,3}$ are chosen judiciously then the events of containing those copies satisfy the general up-set conditions.

\begin{theorem}[\normalfont{Janson's inequality, \cite{riordan2015janson}}]\label{thm:JansonsInequality}
Let $(\Omega,\cF,\bbP)$ be a probability space. Let $\cI\subseteq\cF$ be a family of events such that for all $A,B\in\cI$ we have $\bbP\{A\cap B\}\geq\bbP\{A\}\bbP\{B\}$, $A\cap B\in\cI$, and $A\cup B\in\cI$. Let $B_1,\dots,B_n\in\cI$ and define
$$\mu:=\sum_{i=1}^n\bbP\{B_i\}\hspace{7mm}\text{and}\hspace{7mm}\Delta:=\sum_{i\sim j}\bbP\{B_i\wedge B_j\}\,,$$
where the second sum is over unordered pairs $i$, $j$ such that $B_i$, $B_j$ are not independent. Then
$$\bbP\left\{\bigwedge_{i=1}^n\overline{B_i}\right\}\leq\exp\left(-\min\left\{\frac{\mu}{2}\,,\,\frac{\mu^2}{4\Delta}\right\}\right)\,.$$
\end{theorem}

\subsection{Preliminaries on Graph Regularity}\label{subsec:graphregularity}
Let $G$ be a graph and $A,B\subseteq V(G)$. The pair $\{A,B\}$ is said to be \emph{$\epsilon$-regular} if for all subsets $A'\subseteq A$ and $B'\subseteq B$ with $|A'|\geq\epsilon|A|$ and $|B'|\geq\epsilon|B|$, we have $|d(A',B')-d(A,B)|\leq\epsilon$. For $\epsilon>0$, a partition $\{V_0,\dots,V_k\}$ of $V(G)$ (where $V_0$ is possibly empty) is said to be an \emph{$\epsilon$-regular partition} of $G$ if the following conditions hold:
\begin{enumerate}[(\emph{\roman*})]
    \item $|V_0|\leq\epsilon|V(G)|$,
    \item $|V_i|=|V_j|$ for all $i,j\in[k]$,
    \item all but at most $\epsilon k^2$ pairs $\{V_i,V_j\}$, $1\leq i<j\leq k$, are $\epsilon$-regular.
\end{enumerate}
The sets $V_1,\dots,V_k$ are called the \emph{clusters} or \emph{parts} and $V_0$ the \emph{exceptional set} of the $\epsilon$-regular partition.

In our proofs we will need regular partitions equipped with additional information, namely labels on clusters and pairs of clusters. When applying induced embedding lemmas, these labels describe whether edges, non-edges, or both edges and non-edges of a fixed graph $H$ can be embedded in a given cluster or between a pair of clusters. These kinds of regular partitions have appeared in several past works including \cite{alon2000efficient,bollobas2000structure,alon2008characterization,marchant2011structure,bottcher2012perfect}. We will use the definitions and embedding lemmas from \cite{bottcher2012perfect}, which we record now.

For a graph $G$ and constants $\mu>0$, $\epsilon>0$, and $k\in\N$, a \emph{$(\mu,\epsilon,k)$-subpartition} of $G$ is a set of pairwise disjoint subsets $W_1,\dots,W_k\subseteq V(G)$ such that $|W_i|\geq\mu|V(G)|$ for all $i\in[k]$, and every pair $\{W_i,W_j\}$ with $i\neq j$ is $\epsilon$-regular. A $(\mu,\epsilon,k)$-subpartition of a graph $G$ is said to be \emph{dense} if every pair $\{W_i,W_j\}$ of distinct subsets has $d(W_i,W_j)\geq\frac{1}{2}$, and \emph{sparse} if every pair $\{W_i,W_j\}$ of distinct subsets has $d(W_i,W_j)<\frac{1}{2}$.

\begin{lemma}\label{lemma:RegSubpartition}{\normalfont(\cite[Corollary~3.4]{alon2000efficient})}
For all $\epsilon>0$ and $k\in\N$ there exists $\mu=\mu(\epsilon,k)>0$ such that every graph on at least $\mu^{-1}$ vertices has a dense or sparse $(\mu,\epsilon,k)$-subpartition.
\end{lemma}

\begin{definition}[Type coloring and type of a graph]\label{defn:typecoloringtype}
A \emph{type coloring} of a graph $R=(V,E)$ is a mapping $\sigma:V\cup E\to\{0,\frac{1}{2},1\}$ such that $\sigma(V)\subseteq\{0,1\}$. Let $G$ be a graph and $V(G)=V_0\cup V_1\cup\cdots\cup V_k$ an $\epsilon$-regular partition of $G$. Let $V_R:=[k]$ and $E_R:=\{ij:\{V_i,V_j\}\text{ is }\epsilon\text{-regular}\}$. For all constants $\delta>0$, $\epsilon'>0$, and $k'\in\N$, an \emph{$(\epsilon,\epsilon',\delta,k')$-type} associated with the $\epsilon$-regular partition $V(G)=V_0\cup V_1\cup\cdots\cup V_k$ is a type-colored graph $R=(V_R,E_R,\sigma)$, where $\sigma:V_R\cup E_R\to\{0,\frac{1}{2},1\}$ is a mapping such that for all $ij\in E_R$,
$$\sigma(ij)=\begin{cases}0&d(V_i,V_j)\leq \delta\\\frac{1}{2}&\delta<d(V_i,V_j)<1-\delta\\1&d(V_i,V_j)\geq1-\delta\end{cases}$$
and for all $i\in V_R$,
$$\sigma(i)=\begin{cases}0&G[V_i]\text{ has a sparse }(\mu,\epsilon',k')\text{-subpartition}\\1&G[V_i]\text{ has a dense }(\mu,\epsilon',k')\text{-subpartition}\end{cases}\,,$$
where $\mu=\mu(\epsilon',k')$ is the constant from \Cref{lemma:RegSubpartition}. A type $R'$ is said to \emph{refine} a type $R$ if the partition $P'$ associated to $R'$ is a refinement of the partition associated to $R$, that is, for all $S'\in P'$ there exists $S\in P$ such that $S'\subseteq S$.
\end{definition}

We refer to the three types of edges $\sigma(ij)=0$, $\frac{1}{2}$, and $1$ as \emph{sparse}, \emph{random-like}, and \emph{dense}, respectively. Similarly, we refer to the two types of edges $\sigma(i)=0$ and $1$ as \emph{sparse} and \emph{dense}, respectively.

The following version of the regularity lemma asserts that every sufficiently large graph has a type.

\begin{lemma}[\normalfont{Type Lemma, \cite[Lemma~2.7]{bottcher2012perfect}}]\label{lemma:TypeLemma}
For all $\epsilon,\epsilon'>0$ and all $l,k'\in\N$, there exist positive integers $u$ and $n_0$ such that for all $\delta>0$, every graph $G$ on at least $n_0$ vertices has an $(\epsilon,\epsilon',\delta,k')$-type $R=(V_R,E_R,\sigma)$ with $l\leq|V_R|\leq u$.
\end{lemma}

Let $H=(V_H,E_H)$ be a graph and $R=(V_R,E_R,\sigma)$ a type-colored graph. A \emph{colored homomorphism} from $H$ to $R$ is a mapping $\phi:V_H\to V_R$ such that the following conditions hold:
\begin{enumerate}[(\emph{\roman*})]
	\item For all distinct $u,v\in V_H$, if $\phi(u)\neq\phi(v)$ then $\phi(u)\phi(v)\in E_R$.
	\item If $uv\in E_H$ then either $\phi(u)\neq\phi(v)$ and $\sigma(\phi(u)\phi(v))\in\{\frac{1}{2},1\}$, or $\phi(u)=\phi(v)$ and $\sigma(\phi(u))=1$.
	\item If $uv\not\in E_H$ then either $\phi(u)\neq\phi(v)$ and $\sigma(\phi(u)\phi(v))\in\{0,\frac{1}{2}\}$, or $\phi(u)=\phi(v)$ and $\sigma(\phi(u))=0$.
\end{enumerate}

\begin{lemma}{\normalfont(Induced Embedding Lemma, \cite[Lemma~2.9]{bottcher2012perfect})}\label{lemma:InducedEmbedding}
For all $k'\in\N$ and all $\delta>0$ there exist $\epsilon,\epsilon'>0$ such that the following holds. Let $G$ be a graph with an $(\epsilon,\epsilon',\delta,k')$-type $R$. If $H$ is a graph on $v(H)\leq k'$ vertices and there exists a colored homomorphism from $H$ to $R$, then $G$ contains $H$ as an induced subgraph.
\end{lemma}

The following is a simple consequence of \Cref{lemma:InducedEmbedding} stating that if a type $R$ of a graph $G$ has certain type-colored triangles (namely, three random-like edges, or two random-like edges and one sparse edge), then $G$ necessarily contains $K_{1,3}$ as an induced subgraph.

\begin{lemma}\label{lemma:ClawForbiddenColorings}
For all $k'\geq4$ and $\delta>0$, there exist constants $\epsilon_0,\epsilon'_0>0$ such that the following holds for all $0<\epsilon\leq\epsilon_0$ and $0<\epsilon'\leq\epsilon'_0$. Let $G$ be a graph with an $(\epsilon,\epsilon',\delta,k')$-type $R=(V_R,E_R,\sigma)$. Suppose there exist distinct vertices $u,v,w\in V(R)$ such that either (or both) of the following cases hold: (\emph{i}) $\sigma(uv)=\sigma(vw)=\sigma(wu)=\frac{1}{2}$; or (\emph{ii}) $\sigma(uv)=\sigma(uw)=\frac{1}{2}$ and $\sigma(vw)=0$. Then $G$ contains $K_{1,3}$ as an induced subgraph.
\end{lemma}
\begin{proof}
Let $\epsilon_0=\epsilon_{\ref{lemma:InducedEmbedding}}>0$ and $\epsilon'_0=\epsilon'_{\ref{lemma:InducedEmbedding}}>0$ be the constants provided when \Cref{lemma:InducedEmbedding} is applied with $k'$ and $\delta$. Let $0<\epsilon\leq\epsilon_0$ and $0<\epsilon'\leq\epsilon'_0$ and assume $R=(V_R,E_R,\sigma)$ is an $(\epsilon,\epsilon',\delta,k')$-type for $G$ satisfying the hypotheses of the lemma. By \Cref{lemma:InducedEmbedding}, it suffices to prove there exists a colored homomorphism from $K_{1,3}$ to $R$. Let $V(K_{1,3})=\{a,b,c,d\}$ be a labeling of the vertices of $K_{1,3}$ where $a$ is the vertex of degree three. First, if $\sigma(u)=0$, then in both cases (\emph{i}) and (\emph{ii}), the mapping given by $a\mapsto v$ and $b,c,d\mapsto u$ is a colored homomorphism. Otherwise we have $\sigma(u)=1$, and the mapping given by $a,b\mapsto u$, $c\mapsto v$, and $d\mapsto w$ is a colored homomorphism.
\end{proof}

\subsection{Preliminaries on Graph Limits and the Cut Metric}\label{subsec:graphlimits}
We record some facts about graph limits that we will need. See \cite{lovasz2012large} for a general reference on the subject.

Let $F$ and $G$ be graphs. Let $\hom(F,G)$ denote the set of homomorphisms from $F$ to $G$, that is, maps $\phi:V(F)\to V(G)$ such that $uv\in E(F)$ implies $\phi(uv)\in E(G)$. Let $\homind(F,G)$ denote the set of induced homomorphisms from $F$ to $G$, that is, maps $\phi:V(F)\to V(G)$ such that $uv\in E(F)$ if and only if $\phi(uv)\in E(G)$. The \emph{homomorphism density} of $F$ in $G$ is defined by
$$t(F,G):=\frac{|\hom(F,G)|}{v(G)^{v(F)}}\,,$$
and, assuming $v(F)\leq v(G)$, the \emph{induced homomorphism density} of $F$ in $G$ is defined by
$$t_{\ind}(F,G):=\frac{|\homind(F,G)|}{v(G)(v(G)-1)\cdots(v(G)-v(F)+1)}\,.$$
The definitions of $t$ and $t_{\ind}$ make sense for unlabeled graphs by fixing an arbitrary labeling of the vertex sets.

A sequence $\{G_n\}_{n\in\N}$ of graphs is said to be \emph{convergent} if $\lim_{n\to\infty}t(F,G_n)$ exists for all fixed graphs $F$. We obtain an equivalence relation on the set of convergent graph sequences by considering two sequences equivalent if and only if they have the same limiting homomorphism densities for all graphs $F$. For all subsets $\cV\subseteq\cU$ of unlabeled graphs, let $\overline{\cV}$ denote the set of equivalence classes $\Gamma\subseteq\cU^\N$ of convergent graph sequences such that there exists a graph sequence $\{G_n\}_{n\in\N}\in\cV^\N\cap\Gamma$. The set $\cUbar$ is a completion of $\cU$, where $\cU$ is isometrically embedded in $\cUbar$ with respect to the cut metric (defined below). For all subsets $\cV\subseteq\cU$ denote $\widehat{\cV}:=\overline{\cV}\setminus\cV$, and call the elements of $\cUhat$ \emph{graph limits}.

If $\{G_n\}_{n\in\N}\in\Gamma\in\cUhat$ then write $G_n\to\Gamma$ and define, for all graphs $F$,
$$t(F,\Gamma):=\lim_{n\to\infty}t(F,G_n)\,.$$
Let $\cW$ denote the set of symmetric measurable functions $W:[0,1]^2\to[0,1]$ (symmetric here means $W(x,y)=W(y,x)$). The elements of $\cW$ are called \emph{graphons}. For all $W\in\cW$ and all graphs $F$, define the \emph{homomorphism density} of $F$ in $W$
$$t(F,W):=\int_{[0,1]^{v(f)}}\prod_{uv\in E(F)}W(x_u,x_v)\prod_{u\in V(F)}dx_u\,,$$
and the \emph{induced homomorphism density} of $F$ in $W$
$$t_{\ind}(F,W):=\int_{[0,1]^{v(F)}}\prod_{uv\in E(F)}W(x_u,x_v)\prod_{uv\in E(F^c)}(1-W(x_u,x_v))\prod_{u\in V(F)}dx_u\,.$$
\citeauthor{lovasz2006limits} proved in \cite{lovasz2006limits} that for all graph limits $\Gamma$ there exists a graphon $W\in\cW$ such that $t(F,W)=t(F,\Gamma)$ for all graphs $F$, and conversely for all $W\in\cW$ there is a unique graph limit $\Gamma$ such that $t(F,W)=t(F,\Gamma)$ for all $F$. We say that the graphon $W$ \emph{represents} $\Gamma$ and write $\widehat{W}:=\Gamma$. Similarly if $\cZ\subseteq\cW$ then we denote $\widehat{\cZ}:=\{\widehat{W}:W\in\cZ\}$. Two graphons are \emph{equivalent} if they represent the same graph limit; see \Cref{ftnt:GraphonUniquness} for a characterization of equivalent graphons. If $W$ is a graphon representing $\Gamma$ and $G_n\to\Gamma$ then we also write $G_n\to W$.

On several occasions we will work with graphons that are naturally associated to matrices, graphs, and regular partitions. For a matrix $A\in\R^{n\times n}$ let $W_A\in\cW$ denote the graphon defined by $W_A(x,y)=A_{\ceil{nx},\ceil{ny}}$ for $\min\{x,y\}>0$ and $W_A(x,y)=0$ for $\min\{x,y\}=0$. For all graphs $G\in\cG(n)$ let $A_G$ denote the adjacency matrix of $G$ with all ones on the diagonal, and if $G$ is a weighted graph then let $A_G$ denote the weights matrix of $G$. For all (weighted or unweighted) graphs $G$, denote $W_G:=W_{A_G}$. Let $R$ be an $(\epsilon,\epsilon',\delta,k')$-type associated to an $\epsilon$-regular partition $V(G)=V_0\cup V_1\cup\cdots\cup V_k$ of $G$ (as in \Cref{defn:typecoloringtype}). The \emph{density matrix} of $R$ is defined to be the $k\times k$ matrix $Q$ with $Q_{ij}=d(V_i,V_j)$. The graphon $W_R\in\cW$ is defined as follows: let $a:=|V_0|/n$ and let $b:=|V_1|/n$; for all $x\in[0,1]$ let $i_x=0$ if $x\in[0,a]$ or $\ceil*{\frac{x-a}{b}}$ if $x\in(a,1]$; let $W_R(x,y)=d(V_{i_x},V_{i_y})$. For all graph limits $\Gamma\in\cUhat$ let $W_\Gamma\in\cW$ be a canonically chosen graphon that represents $\Gamma$.

The \emph{cut norm} of a graphon $W\in\cW$ is defined to be
$$\norm{W}_\square:=\sup_{S,T\subseteq[0,1]}\left|\int_{S\times T}W(x,y)\,dx\,dy\right|\,.$$
For all $W,W'\in\cW$ define
$$\delta_\square(W,W'):=\inf_\phi\norm{W-W'\circ(\phi\otimes\phi)}_\square\,,$$
where the infimum is over all measure-preserving bijections $\phi:[0,1]\to[0,1]$. Let $\cG_w$ denote the set of weighted graphs with weights in $[0,1]$, and we view $\cG$ as being a subset of $\cG_w$. For all $\Gamma\in\cUbar\cup\cG_w\cup\cW$, define $K(\Gamma)$ to be the graphon $W_\Gamma$ if $\Gamma\in\cUbar\cup\cG_w$, and define $K(\Gamma):=\Gamma$ if $\Gamma\in\cW$. The \emph{cut metric} $\delta_\square$ is then defined on pairs $(\Gamma_1,\Gamma_2)\in(\cUbar\cup\cG_w\cup\cW)^2$ by
\begin{equation}
\delta_\square(\Gamma_1,\Gamma_2):=\delta_\square(K(\Gamma_1),K(\Gamma_2))\,.\label{eqn:cutmetricWholeDomain}
\end{equation}
The \emph{cut norm} of a real $n\times n$ matrix $A$ is defined to be
$$\norm{A}_\square:=\frac{1}{n^2}\max_{S,T\subseteq[n]}\left|\sum_{i\in S,j\in T}A_{ij}\right|\,.$$
The \emph{cut distance} between two weighted graphs $G,G'$ on $n$ vertices is defined as
$$d_\square(G,G'):=\norm{A_G-A_{G'}}_\square\,,$$
and one also defines the distance
$$\widehat{\delta}_\square(G,G'):=\min_{P}\norm{A_G-P^TA_{G'}P}_\square\,,$$
where the minimum is over $n\times n$ permutation matrices $P$.

The following lemma, almost immediate from the proof of \cite[Theorem~2.2]{lovasz2006limits} is a key preliminary step in our solution to the graphon variational problem since it provides, for a given claw-free graphon, a sequence of types of claw-free graphs. 

\begin{lemma}\label{lemma:UFHatGraphonSeq}
For all $W\in\cW$ such that $t_{\ind}(K_{1,3},W)=0$ and all $\delta>0$, there exists a sequence $\{W_m\}_{m\in\N}$ of graphons such that the following conditions hold:
\begin{enumerate}[(\emph{\roman*})]
	\item $\norm{W_m-W}_1\to0$ as $m\to\infty$.
	\item For all $m$ there exist constants $\eta_m,\eta'_m>0$ and $k'\geq4$ such that $W_m$ is the graphon associated with an $(\eta_m,\eta'_m,\delta/5,k')$-type $R_m$ of an claw-free graph. Additionally, the constants $\eta_m,\eta'_m,\delta/5,k'$ satisfy the hypotheses of \Cref{lemma:InducedEmbedding} for all $F\in\cF$.
	\item $\eta_m\to0$ as $m\to\infty$.
	\item The number of clusters of $R_m$ diverges as $m\to\infty$.
\end{enumerate}
\end{lemma}

\begin{definition}[$W$-random graph]\label{defn:GammaRandomGraph}
For all $n\in\N$ and $W\in\cW$, let $G(n,W)$ be the random graph on the vertex set $[n]$ sampled as follows. Let $X_1,\dots,X_n$ and $Y_{ij}$, $\{i,j\}\in\binom{[n]}{2}$, be i.i.d. uniformly distributed on $[0,1]$. For all $\{i,j\}\in\binom{[n]}{2}$, form an edge between vertices $i$ and $j$ if and only if $Y_{ij}\leq W(X_i,X_j)$. If $\Gamma$ is a graph limit then let $G(n,\Gamma)$ denote the random graph $G(n,W)$ for any graphon $W$ representing $\Gamma$ (and note the distribution of $G(n,\Gamma)$ is independent of $W$).
\end{definition}

\begin{proof}[Proof of \Cref{lemma:UFHatGraphonSeq}]
The sequence of random graphs $G_n:=G(n,W)$ converges to $W$ in cut metric with probability 1 (see \cite[Theorem~2.5]{lovasz2006limits}), and each $G_n$ is claw-free almost surely. Let $\epsilon,\epsilon'$ be the constants given when \Cref{lemma:InducedEmbedding} is applied with $k'$ and $\delta$. By applying the type lemma (\Cref{lemma:TypeLemma}) instead of weak regularity in the proof of \cite[Theorem~2.2]{lovasz2006limits}, we obtain a subsequence $\{G'_n\}_{n\in\N}$ of $\{G_n\}_{n\in\N}$ such that the following conditions hold. First, for all $n\in\N$ and $j\in[n]$, the graph $G'_n$ has an $(\eta,\epsilon',\delta,k')$-type $R_{n,j}$, where $\eta:=\min\{\epsilon,1/j\}$, such that $R_{n,j}$ has $k_j$ clusters and $k_j\times k_j$ density matrix $Q_{n,j}$. Second, for all $j\in\N$, there exists a $k_j\times k_j$ matrix $Q_j$ such that $\lim_{n\to\infty}Q_{n,j}\to Q_j$. Finally, letting $W_j:=W_{Q_j}\in\cW$ denote the natural graphon associated with $Q_j$, we have $W_j\to W$ almost everywhere, hence $\norm{W_m-W}_1\to0$ as $m\to\infty$.
\end{proof}

\subsection{Preliminaries on the Entropy of Graph Limits}
\label{subsec:entropygraphlimits}
For all graphons $W\in\cW$, the entropy of $W$ is defined to be
$$H(W):=\int_{[0,1]^2}H(W(x,y))\,dx\,dy\,.$$
The entropy of a graph limit $\Gamma\in\cUhat$ is defined to be $H(\Gamma):=H(W)$ for any $W\in\cW$ representing $W$ (and this is independent of $W$). For all $W\in\cW$ and $p\in(0,1)$ the relative entropy of $W$ is defined by
$$I_p(W):=\int_{[0,1]^2}I_p(W(x,y))\,dx\,dy\,,$$
where $I_p$ is the function defined in \pref{eqn:defIpx}. The relative entropy of a graph limit $\Gamma\in\cUhat$ is defined by $I_p(\Gamma):=I_p(W)$ for any $W\in\cW$ representing $\Gamma$. For a graphon $W\in\cW$ define the quantity
$$\rand(W):=|\{(x,y):0<W(x,y)<1\}|\,,$$
where $|\cdot|$ denote Lebesgue measure. For a graph limit $\Gamma\in\cUhat$ let $\rand(\Gamma):=\rand(W)$ for any $W$ representing $\Gamma$.

If $\cQ$ is a set of graphs and $n,m\in\N$ then let $\cQ(n):=\{G\in\cQ:v(G)=n\}$ and $\cQ(n,m):=\{G\in\cQ(n):e(G)=m\}$. The following lemma is a simple consequence of \cite[Theorem~1]{hatami2018graph}, proven in \Cref{app:deferred}.

\begin{lemma}\label{lemma:limsupentropyupperbound}
If $\cQ\neq\emptyset$ is a set of graphs, $p\in(0,1)$ is a constant, and $G\sim G(n,p)$ then
$$\limsup_{n\to\infty}\frac{1}{\binom{n}{2}}\log_2\bbP\{G\in\cQ(n)\}\leq\sup_{\Gamma\in\overline{\cQ}}I_p(\Gamma)\,.$$
\end{lemma}

If $H$ is a fixed graph and $n,m\in\N$ then let $\cF(n,H)$ denote the set of induced-$H$-free graphs on $n$ vertices, and let $\cF(n,m,H):=\cF(n,H)\cap\cG(n,m)$. The following propositions, proven in \Cref{app:deferred}, state that the rate function and entropy density for induced-$H$-free graphs is given by a variational problem over a set of induced-$H$-free graphons.

\begin{proposition}\label{lemma:LDPforGNPkl}
Let $p\in(0,1)$ be a constant and $G\sim G(n,p)$. If $H$ is a fixed graph and
$$\cF:=\{W\in\cW:t_{\ind}(H,W)=0,\,\rand(W)>0\}$$
is nonempty then
$$\lim_{n\to\infty}\frac{1}{\binom{n}{2}}\log_2\bbP\{G\in\cF(n,H)\}=\sup_{W\in\cF}I_p(W)\,.$$
\end{proposition}

\begin{proposition}\label{lemma:FnmGraphonOptAsymps}
Let $\gamma\in(0,1)$, $n,m\in\N$, and $m\sim\gamma\binom{n}{2}$. If $H$ is a fixed graph and
$$\cF_\gamma:=\{W\in\cW:t(K_2,W)=\gamma,\,t_{\ind}(H,W)=0,\,\rand(W)>0\}$$
is nonempty then
$$\lim_{n\to\infty}\frac{1}{\binom{n}{2}}\log_2|\cF(n,m,H)|=\sup_{W\in\cF_\gamma}H(W)\,.$$
\end{proposition}

\section{Solving the Graphon Variational Problem}
\label{sec:GraphonOpt}
In this section we solve the variational problem over claw-free graphons with edge density $\gamma$. Several definitions, notations, and lemmas used in this section were recorded in \Cref{subsec:graphlimits,subsec:entropygraphlimits}. For all $\gamma\in(0,1)$, define the set of graphons
\begin{equation}
\cX_\gamma:=\big\{W\in\cW:t(K_2,W)=\gamma\,,\,t_{\ind}(K_{1,3},W)=0\big\}\,,\label{eqn:defnXgamma}
\end{equation}
and the variational problem
\begin{equation}\phi(\gamma):=\sup\{H(W):W\in\cX_\gamma\} \,.\label{eqn:variationalproblem}
\end{equation}
Let $\cX_\gamma^\ast$ denote the set of all $W\in\cX_\gamma$ that achieve the supremum in \pref{eqn:variationalproblem}. Note $\cX^\ast_\gamma$ is nonempty since $\cX_\gamma$ is compact (see \cite[Theorem~3.7]{borgs2008convergent}) and $H$ is upper semicontinuous (see \cite[Lemma~2.1]{chatterjee2011large}).

\subsection{Overview of Proof}
The solution to the variational problem \pref{eqn:variationalproblem} uses types (introduced in \Cref{sec:prelim}) and an extremal problem over edge colorings (\Cref{lemma:c4ATMOSTc3,lemma:EdgeColStab}). \Cref{lemma:ClawForbiddenColorings} shows that if $R$ is a type for a claw-free graph, then the type coloring of $R$ avoids certain substructures: $R$ contains no triangle with three random-like edges, nor any triangle with two random-like edges and one sparse edge. \Cref{lemma:c4ATMOSTc3} proves that a 3-edge-coloring avoiding these substructures has at most as many random-like edges as dense edges, and \Cref{lemma:EdgeColStab} proves extremal structure and stability. We use these results to solve \pref{eqn:variationalproblem} as follows.

Every optimal graphon $W$ for $\phi$ is the limit of a sequence $R_n$ of types associated with claw-free graphs on $n$ vertices. The concavity of $I_p$ implies that $W$ takes three values $\{0,a,1\}$ almost everywhere for some $a\in(0,1)$. The three values of $W$ approximately correspond with the three types of edge labels in $R_n$. Since $\int W=\gamma$, we have $\gamma=x+ay$ and $H(W)=yH(a)$, where $x$ is the measure of the set where $W=1$ and $y$ is the measure of the set where $W=a$. Since $R_n$ has at most as many random-like edges as dense edges, we show that $y\leq x$. Further, $yH(a)$ is maximum if and only if $x=y$, and there is a threshold at $\gamma^\ast=\frac{5-\sqrt{5}}{4}$\,: for all $\gamma\geq\gamma^\ast$, we have $x=y=\frac{1}{2}$ while for $\gamma<\gamma^\ast$, we have $x=y=\frac{5-\sqrt{5}}{20}\gamma$. For large $n$, the equality $x=y$ means the number of random-like edges in $R_n$ is very close to the number of dense edges. The stability result \Cref{lemma:EdgeColStab} then implies $R_n$ is closed in Hamming distance to a certain extremal structure and allows us to deduce $W$ is equivalent to a graphon in a set $\cV_\gamma$ of well-structured graphons defined below.

For all $\gamma\geq\gamma^\ast$, the set $\cV_\gamma$ contains one element (up to equivalence), the unique optimal graphon at edge density $\gamma$. For $\gamma<\gamma^\ast$, the set $\cV_\gamma$ has infinitely many (nonequivalent) graphons. Since every graphon in $\cV_\gamma$ achieves the same entropy, has edge density $\gamma$, and is claw-free, we prove $\cV_\gamma$ is precisely the set of optimizers for $\phi$.
\medskip\smallskip

We now define the graphons that we will prove comprise $\cX^\ast_\gamma$. Let $\xi:\R^2\to[0,1]$ be the step function defined by
$$\xi(x,y)=\begin{cases}1&(x,y)\in\big(0,\frac{1}{2}\big)^2\cup\big(\frac{1}{2},1\big)^2\\[3pt]\frac{3-\sqrt{5}}{2}&(x,y)\in\big(0,\frac{1}{2}\big)\times\big(\frac{1}{2},1\big)\cup\big(\frac{1}{2},1\big)\times\big(0,\frac{1}{2}\big)\\[3pt]0&\text{otherwise}\end{cases}\,.$$
Let $\Lambda$ denote the set of all sequences $(\lambda_0,\lambda_1,\dots)$ of at least two real numbers (the sequence may be finite or infinite) such that three conditions hold: (\emph{i}) $\lambda_0=0$, (\emph{ii}) $\lambda_i<\lambda_{i+1}\leq1$, and (\emph{iii}) the numbers $\lambda_{i+1}-\lambda_i$ are nonincreasing in $i$. For all $\blambda=(\lambda_0,\lambda_1,\dots)\in\Lambda$, define the graphon $W_{\blambda}\in\cW$ by
\begin{equation}
W_{\blambda}(x,y)=\sum_{0\leq i<|\blambda|}\xi\left(\frac{x-\lambda_{i}}{\lambda_{i+1}-\lambda_{i}},\frac{y-\lambda_{i}}{\lambda_{i+1}-\lambda_{i}}\right)\,,\label{eqn:flambdaGraphonDef}
\end{equation}
where $|\blambda|$ denotes the cardinality of $\blambda$ as a set. (\Cref{fig:graphons-super,fig:graphons-sub} depict such graphons $W_{\blambda}$.) Let $\cV:=\{W_{\blambda}:\blambda\in\Lambda\}$ denote the set of all such graphons. For all $\gamma\in\big(0,\frac{5-\sqrt{5}}{4}\big)$ let
\begin{align*}
\cV_\gamma &:= \left\{W_{\blambda}:\int W_{\blambda}(x,y)\,dx\,dy=\gamma\,,\,\blambda\in\Lambda\right\} \\
&= \left\{W_{\blambda}:\sum_{0\leq i<|\blambda|}(\lambda_{i+1}-\lambda_{i})^2=\gamma\left(1+\frac{\sqrt{5}}{5}\right)\,,\,\blambda=(\lambda_0\,,\,\lambda_1\,,\,\dots)\in\Lambda\right\}\,,
\end{align*}
and for all $\gamma\in\big[\frac{5-\sqrt{5}}{4},1\big)$ let $\cV_\gamma$ be the set containing the single graphon $W^\ast_\gamma\in\cX_\gamma$ defined by
\begin{equation}
W^\ast_\gamma(x,y)=\begin{cases}1&(x,y)\in\big(0,\frac{1}{2}\big)^2\cup\big(\frac{1}{2},1\big)^2\\[3pt]2\gamma-1&(x,y)\in\big(0,\frac{1}{2}\big)\times\big(\frac{1}{2},1\big)\cup\big(\frac{1}{2},1\big)\times\big(0,\frac{1}{2}\big)\\[2pt]0&\text{otherwise}\end{cases}\,.\label{eqn:deffstarlambdagraphon}
\end{equation}
Notice that for all $\gamma\in(0,1)$ we have $\cV_\gamma\subseteq\cX_\gamma$. Also note that for all $\gamma\in\big(0,\frac{5-\sqrt{5}}{4}\big)$, there are graphons $W_{\blambda}\in\cV_\gamma$ for which $\blambda=(\lambda_0,\lambda_1,\dots)$ is an infinite sequence (in fact, there are infinitely many non-equivalent such graphons). For example, for edge density $\gamma=\frac{5-\sqrt{5}}{12}$, the set $\cV_\gamma$ contains the graphon $W_{\blambda}$ where $\blambda=(0,\frac{1}{2},\frac{3}{4},\frac{7}{8},\dots)$.

The following proposition states the solution to the variational problem \pref{eqn:variationalproblem}. In the statement, recall that if $\cZ\subseteq\cW$ is a set of graphons then $\widehat{\cZ}\subseteq\widehat{\cU}$ denotes the set of graph limits represented by elements of $\cZ$. This notation allows us to state that two sets of graphons are the same up to equivalence (in the sense of \Cref{ftnt:GraphonUniquness}).

\begin{proposition}\label{prop:FsubsetLambda}
For all $\gamma\in(0,1)$ we have $\widehat{\cX}^\ast_\gamma=\widehat{\cV}_\gamma$.
\end{proposition}

The proof of \Cref{prop:FsubsetLambda} uses extremal and stability properties of edge colorings, which we state now and prove in \Cref{sec:edgecolorings}. If $\varphi$ is a red--green--blue coloring of $E(K_n)$, let $e_r(\varphi)$, $e_g(\varphi)$, and $e_b(\varphi)$ denote the number of red, green, and blue edges of $\varphi$, respectively.

\begin{lemma}\label{lemma:c4ATMOSTc3}
Let $\varphi$ be a red--green--blue coloring of $E(K_n)$. If $\varphi$ has no triangle colored (red, red, red) nor any colored (red, red, green) then
\begin{equation}
e_r(\varphi)\leq e_b(\varphi)+\floor*{\frac{n}{2}}\,.\label[inequality]{ineq:thmredblueineq}
\end{equation}
\end{lemma}

The next proposition gives a structural characterization of the colorings $\varphi$ for which \pref{ineq:thmredblueineq} is close to equality. We introduce the following definitions to make the extremal structure precise. Let $\cE(n)$ be the set of red--green--blue colorings $\varphi$ of $E(K_n)$ such that for some partition of the vertex set $V(K_n)=H_1\cup\cdots\cup H_l$, all of the following conditions hold: if $n$ is odd then exactly one of the numbers $|H_i|$ is odd, otherwise all $|H_i|$ are even; for all $i\in[l]$ there is a balanced partition $V(H_i)=X_i\cup Y_i$ (i.e. the sizes of $X_i$ and $Y_i$ differ by at most one) such that the edges $E(X_i)\cup E(Y_i)$ are all blue and the edges $E(X_i,Y_i)$ are all red; for all $i\neq j$ the edges $E(H_i,H_j)$ are all green. The \emph{Hamming distance} between colorings $\varphi$ and $\psi$ of $E(K_n)$ is
$$d(\varphi,\psi):=|\{e\in E(K_n):\varphi(e)\neq\psi(e)\}|\,.$$
Suppose $\varphi\in\cE(k)$, $k\leq n$, and $\sigma:V(K_k)\to V(K_n)$ is injective. Let $\varphi_\sigma$ be the red--green--blue coloring of $E(K_n)$ defined by
$$\varphi_\sigma(e)=\begin{cases}\varphi(\sigma^{-1}(e))&e\subseteq\image(\sigma)\\\text{green}&\text{otherwise}\end{cases}\,,$$
and define the set
\begin{equation}
\cF(n):=\big\{\varphi_\sigma:k\in[n],\,\varphi\in\cE(k),\,\sigma:V(K_k)\to V(K_n)\text{ injective}\big\}\,.\label{eqn:FnColoringsSet}
\end{equation}
The Hamming distance from $\varphi$ to $\cF(n)$ is defined by
$$d(\varphi,\cF(n)):=\min\{d(\varphi,\psi):\psi\in\cF(n)\}\,,$$
so $d(\varphi,\cF(n))$ is the least number of edits one can make to the colors of $\varphi$ to obtain a member of $\cF(n)$.

\begin{proposition}\label{lemma:EdgeColStab}
For all $\epsilon>0$ there exists $\delta>0$ such that the following holds. Let $\varphi$ be a red--green--blue coloring of $E(K_n)$. Assume $\varphi$ has no triangle colored (red, red, red) nor any colored (red, red, green). If
$$e_r(\varphi)\geq e_b(\varphi)+\floor*{\frac{n}{2}}-\delta n^2$$
then $d(\varphi,\cF(n))\leq\epsilon n^2$.
\end{proposition}

We will now state and prove several more lemmas in preparation for the proof of \Cref{prop:FsubsetLambda}. For all $W\in\cW$ define the following subsets of $[0,1]^2$:
\begin{align*}
R_W &:=\{(x,y)\in[0,1]^2:0<W(x,y)<1\}\,, \\
O_W &:=\{(x,y)\in[0,1]^2:W(x,y)=1\} \,.
\end{align*}
For a measurable set $A\subseteq\R^2$, the symbol $|A|$ denotes the Lebesgue measure of $A$.

\begin{lemma}\label{lemma:C4AtMostC3}
If $W\in\cW$ and $t_{\ind}(K_{1,3},W)=0$ then $|R_W|\leq|O_W|$.
\end{lemma}
\begin{proof}
Fix $\delta>0$ and let $\{W_m\}_{m\in\N}$ be the sequence of graphons obtained by applying \Cref{lemma:UFHatGraphonSeq} with $W$ and $\delta$. Then each $W_m$ is the graphon associated with an $(\eta_m,\eta_m',\delta/5,k')$-type $R=(V_R,E_R,\sigma)$ of a claw-free graph $G$, where the constants $\eta_m$, $\eta_m'$, $\delta/5$, $k'$ satisfy the hypotheses of the induced embedding lemma (\Cref{lemma:InducedEmbedding}). Since $G$ is claw-free, \Cref{lemma:ClawForbiddenColorings} implies $R$ does not have three distinct vertices $u,v,w\in V_R$ such that either of the following cases holds: (\emph{i}) $\sigma(uv)=\sigma(vw)=\sigma(wu)=\frac{1}{2}$; or (\emph{ii}) $\sigma(uv)=\sigma(uw)=\frac{1}{2}$ and $\sigma(vw)=0$. Let $k:=|V_R|$ and define a red--green--blue coloring $\varphi$ of $E(K_k)$ as follows: an edge $uv\in E(K_k)$ is red if $\sigma(ij)=\frac{1}{2}$; green if $\sigma(uv)=0$; and blue if $\sigma(uv)=1$ or $uv\not\in E_R$. It follows that $\varphi$ has no triangle colored (red, red, red) nor any colored (red, red, green), hence \Cref{lemma:c4ATMOSTc3} implies $e_r(\varphi)\leq e_b(\varphi)+\floor{n/2}$, where $e_r(\varphi)$ and $e_b(\varphi)$ denote the number of red and blue edges, respectively. Since $R$ has at most $\eta_mk^2$ non-edges,
$$|\{(x,y):\delta<W_m(x,y)<1-\delta\}|\leq|\{(x,y):W_m(x,y)\geq1-\delta\}|+\eta_m+\frac{2}{k}\,.$$
Since $W_m\to W$ in $L^1$, $\eta_m\to0$, and $\delta>0$ was arbitrary, it follows that $|R_W|\leq|O_W|$.
\end{proof}

\begin{lemma}\label{lemma:threevaluegraphon}
If $W\in\cX_\gamma^\ast$ then for some $a\in(0,1)$, $W$ takes the values $\{0,a,1\}$ almost everywhere.
\end{lemma}
\begin{proof}
Fix $W\in\cX_\gamma^\ast$ and let $R:=R_W$. Note that $R$ has positive Lebesgue measure since otherwise $H(W)=0$. Define the graphon $W'$ by
$$W'(x,y)=\begin{cases}\frac{1}{|R|}\int_RW(x,y)\,dx\,dy&(x,y)\in R\\W(x,y)&\text{otherwise}\end{cases}.$$
Notice that $W'\in\cX_\gamma$ since the relations $t(K_2,W)=\gamma$ and $t_{\ind}(K_{1,3},W)=0$ imply $t(K_2,W')=\gamma$ and $t_{\ind}(K_{1,3},W')=0$. The inequality $H(W)\geq H(W')$ holds by assumption. The concavity of $H$ implies
$$H(W)=\int_RH(W(x,y))\,dx\,dy\,\leq|R|\cdot H\left(\frac{1}{|R|}\int_RW(x,y)\,dx\,dy\right)=H(W')\,,$$
with equality holding if and only if $W$ is constant almost everywhere in $R$.
\end{proof}

The following lemma is proven in \Cref{app:deferred}.

\begin{lemma}\label{lemma:kktopt}
Fix $c\in(0,1)$. Let $D\subseteq[0,1]^2$ be the domain of points $(x,y)$ such that $0<y\leq x\leq1-y$ and $0\leq\frac{c-x}{y}\leq1$. Let $f:D\to\R$ be defined by $f(x,y)=yH(\frac{c-x}{y})$. The function $f$ attains a unique maximum $(x^\ast,y^\ast)$, given in terms of $c$ by
\begin{equation}
x^\ast=y^\ast=\begin{cases}\frac{5+\sqrt{5}}{10}c\phantom{...}&c\in\big(0,\frac{5-\sqrt{5}}{4}\big)\\[5pt]\frac{1}{2}&c\in\big[\frac{5-\sqrt{5}}{4},1\big)\end{cases}\,.\label{eqn:kktlemma}
\end{equation}
Additionally, the mapping $c\mapsto f(x^\ast,y^\ast)$ has a continuous first derivative.
\end{lemma}

\begin{lemma}\label{lemma:LambdasubsetF}
For all $\gamma\in(0,1)$ we have $\widehat{\cV}_\gamma\subseteq \widehat{\cX}_\gamma^\ast$.
\end{lemma}
\begin{proof}
It suffices to prove $\cV_\gamma\subseteq\cX_\gamma^\ast$. We will establish an upper bound on $\phi(\gamma)$ and observe that all $W\in\cV_\gamma$ meet this upper bound. By \Cref{lemma:threevaluegraphon} we may restrict the domain of the variational problem \pref{eqn:variationalproblem} to the set $\cY$ of all $W\in\cX_\gamma$ that take three values almost everywhere. Notice that $\cV_\gamma\subseteq\cY$. Let $W\in\cY$ and assume $W$ takes the values $\{0,a,1\}$ almost everywhere for some $a\in(0,1)$. We then have $H(W)=|R_W|\cdot H(a)$, $\int W=\gamma$, and $\gamma=a|R_W|+|O_W|$. Let $r^\ast:(0,1)\to\R$ be the function defined in \pref{eqn:ratefunction}. Since \Cref{lemma:C4AtMostC3} proves $|R_W|\leq|O_W|$, it follows from \Cref{lemma:kktopt} that $H(W)\leq r^\ast(\gamma)$ (specifically, apply \Cref{lemma:kktopt} by substituting $c$ with $\gamma$, $x$ with $|O_W|$, and $y$ with $|R_W|$). Noticing that $H(W)=r^\ast(\gamma)$ for all $W\in\cV_\gamma$, the inclusion $\cV_\gamma\subseteq\cX_\gamma^\ast$ follows immediately.
\end{proof}

\begin{proof}[Proof of \Cref{prop:FsubsetLambda}]
By \Cref{lemma:LambdasubsetF}, it suffices to show that for all $W\in\cX^\ast_\gamma$, there is some $W'\in\cV_\gamma$ that is equivalent to $W$. Fix $W\in\cX_\gamma^\ast$ and $\epsilon\in(0,\frac{1}{16})$. Let $\delta'>0$ be the constant provided when \Cref{lemma:EdgeColStab} is applied with $\epsilon/4$.

First let $\gamma\in\big(0,\frac{5-\sqrt{5}}{4}\big)$ and $\delta:=\min\{\epsilon/4,\delta'\}$. We will show that $\delta_\square(W,\cV)\leq\epsilon$, and this fact together with $t(K_2,W)=\gamma$ proves $W$ is equivalent to a graphon in $\cV_\gamma$. By \Cref{lemma:threevaluegraphon,lemma:kktopt}, $W$ takes the three values $\{0,a,1\}$ almost everywhere for some $a\in(0,1)$, and additionally $|R_W|=|O_W|=\frac{5+\sqrt{5}}{10}\gamma$. Hence the equation $\gamma=a|R_W|+|O_W|$ implies $a=\frac{3-\sqrt{5}}{2}$. Since $t_{\ind}(K_{1,3},W)=0$ and $t(K_2,W)=\gamma$, there exists a sequence $\{W_m\}_{m\in\N}$ of graphons satisfying the conclusions of \Cref{lemma:UFHatGraphonSeq}. Let $\eta_m$, $\eta'_m$, and $R_m$ be the sequences from \Cref{lemma:UFHatGraphonSeq}. Let $\{W'_m\}_{m\in\N}$ be a subsequence of $\{W_m\}_{m\in\N}$ such that $W'_m\to W$ almost everywhere. By Egorov's theorem, there exists a closed subset $E\subseteq[0,1]^2$ of measure $|E|>1-\delta/16$ such that $W'_m\to W$ uniformly on $E$. Fix $m\in\N$ sufficiently large such that the following three conditions hold: the number $k$ of clusters in the type $R_m$ associated with $W'_m$ satisfies $k>16/\delta$; $\eta_m<\delta/16$; and $|W'_m-W|<\epsilon/16$ on $E$. It follows that
\begin{equation}
\norm{W'_m-W}_1\leq\frac{\delta}{16}\left(1-\frac{\delta}{16}\right)+\frac{\delta}{16}+\frac{1}{k}\leq\frac{\delta}{4}\leq\frac{\epsilon}{4}\,.\label{eqn:VarProbWmW}
\end{equation}
Let $R=([k],E_R,\sigma)$ denote the type associated with $W'_m$. For all $i,j\in[k]$ let
\begin{equation*}
S_{ij}:=\left(\frac{i-1}{k},\frac{i}{k}\right)\times\left(\frac{j-1}{k},\frac{j}{k}\right)
\end{equation*}
denote the $i,j$ block of the unit square. For all $x,y\in[0,1]^2$ define $i_{xy},\,j_{xy}\in[k]$ such that $(x,y)\in S_{i_{xy}j_{xy}}$ (we may disregard the set of measure zero where $i_{xy}$ and $j_{xy}$ are not defined). Define the graphon $g\in\cW$ by
\begin{equation}
g(x,y)=\begin{cases}0&S_{i_{xy}j_{xy}}\cap E\neq\emptyset,\,i_{xy}j_{xy}\in E_R,\,\sigma(i_{xy}j_{xy})=0\\\frac{3-\sqrt{5}}{2}&S_{i_{xy}j_{xy}}\cap E\neq\emptyset,\,i_{xy}j_{xy}\in E_R,\,\sigma(i_{xy}j_{xy})=\frac{1}{2}\\1&\text{otherwise}\end{cases}\,.\label{eqn:VarProbgDefn}
\end{equation}
Now if $(x,y)$ is a point such that $S_{i_{xy}j_{xy}}\cap E\neq\emptyset$, then the fact that $W'_m$ is constant on the squares $S_{ij}$ and $|W'_m-W|<\delta/16$ on $E$ implies that for some $z\in\{0,\frac{3-\sqrt{5}}{2},1\}$, we have $W'_m\equiv z$ identically on $S_{i_{xy}j_{xy}}$. Since at most $\delta k^2/16$ squares $S_{ij}$ do not intersect $E$, and since $R_m$ has at most $\eta_mk^2$ irregular pairs, we have
\begin{equation}
\norm{g-W'_m}_1\leq\frac{\delta}{8}+\eta_m+\frac{1}{k}\leq\frac{\delta}{4}\leq\frac{\epsilon}{4}\,.\label{eqn:VarProbgWm}
\end{equation}
Define the red--green--blue coloring $\varphi$ of $E(K_k)$ such that edge $ij$ is colored red if $g(S_{ij})=\frac{3-\sqrt{5}}{2}$, green if $g(S_{ij})=0$, and blue if $g(S_{ij})=1$. By the inclusions
$$R_W\subseteq\vast(\bigcup_{\substack{ij\in E_R\\\sigma(ij)=1/2}}S_{ij}\vast)\cup\left(\bigcup_{ij\not\in E_R}S_{ij}\right)\cup E^c\,,\hspace{6mm}\bigcup_{\substack{ij\in E_R\\\sigma(ij)=1}}S_{ij}\subseteq O_W\cup\left(\bigcup_{ij\not\in E_R}S_{ij}\right)\cup E^c\,,$$
and since $k>16/\epsilon$, we obtain the inequalities
\begin{align*}
|R_W| &\leq \frac{e_r(\varphi)}{k^2}+\frac{\delta}{16}+\frac{1}{k} \leq \frac{e_r(\varphi)}{k^2}+\frac{\delta}{8}\,, \\[5pt]
|O_W| &\geq \frac{e_b(\varphi)}{k^2}-\frac{\delta}{16}-\frac{1}{k} \geq \frac{e_b(\varphi)}{k^2}-\frac{\delta}{8} \,,
\end{align*}
where $e_r(\varphi)$ and $e_b(\varphi)$ denote the number of red and blue edges of $\varphi$, respectively. Since $|R_W|=|O_W|$ and $\delta k^2/32\geq k/2$ (using our assumption on $k$), it follows that $e_r(\varphi)\geq e_b(\varphi)+\floor{k/2}-\delta k^2$. Now \Cref{lemma:EdgeColStab} implies $d(\varphi,\cF(k))\leq\epsilon k^2/4$, where $\cF(k)$ is the set of edge colorings defined by \pref{eqn:FnColoringsSet}. Let $\psi\in\cF(k)$ such that $d(\varphi,\psi)\leq\epsilon k^2/4$, and let $h\in\cW$ be the graphon defined by
\begin{equation}
h(x,y)=\begin{cases}0&(x,y)\in S_{ij},\,\psi(ij)\text{ is green}\\\frac{3-\sqrt{5}}{2}&(x,y)\in S_{ij},\,\psi(ij)\text{ is red}\\1&(x,y)\in S_{ij}\,,\psi(ij)\text{ is blue}\end{cases}\,.\label{eqn:VarProbhDefn}
\end{equation}
It follows that $\norm{g-h}_1\leq\epsilon/4$. By definition of the set $\cF(k)$, it is clear that there exists a measure-preserving bijection $\sigma:[0,1]\to[0,1]$ such that $h\circ(\sigma\otimes\sigma)=W_{\blambda}$ for some $\blambda\in\Lambda$ (using the definition in \pref{eqn:flambdaGraphonDef}). We thus have $\delta_\square(h,\cV)=0$. Combining \pref{eqn:VarProbWmW}, \pref{eqn:VarProbgWm}, and the inequality $\norm{g-h}_1\leq\epsilon/4$, we have
\begin{align*}
\delta_\square(W,\cV) &\leq \delta_\square(W,W'_m)+\delta_\square(W'_m,g)+\delta_\square(g,h)+\delta_\square(h,\cV) \\[5pt]
&= \delta_\square(W,W'_m)+\delta_\square(W'_m,g)+\delta_\square(g,h) \\[5pt]
&\leq \norm{W-W'_m}_1+\norm{W'_m-g}_1+\norm{g-h}_1 \leq \epsilon\,.
\end{align*}
Since $\epsilon>0$ was arbitrary and $t(K_2,W)=\gamma$, we have $\delta_\square(W,\cV_\gamma)=0$, proving that $W$ is equivalent to a graphon in $\cV_\gamma$.

The proof for the case $\gamma\in\big[\frac{5-\sqrt{5}}{4},1\big)$ is very similar and we describe the parts that are different. Redefine $\delta:=\frac{1}{4}\min\left\{\epsilon,\,\delta',\,2\gamma-1\right\}$. \Cref{lemma:threevaluegraphon,lemma:kktopt} prove that $W\in\cW$ takes the two values $\{2\gamma-1,1\}$ almost everywhere. We again obtain a sequence $\{W_m\}_{m\in\N}$ satisfying the conclusions of \Cref{lemma:UFHatGraphonSeq}. Define $W'_m$, $E$, and $R$ in the same way as above. Instead of using $g$ as defined in \pref{eqn:VarProbgDefn}, we define the graphon $g'\in\cW$ by
$$g'(x,y)=\begin{cases}2\gamma-1&S_{i_{xy}j_{xy}}\cap E\neq\emptyset,\,i_{xy}j_{xy}\in E_R,\,\sigma(i_{xy}j_{xy})=\frac{1}{2}\\1&\text{otherwise}\end{cases}\,.$$
Define the edge coloring $\varphi$ of $E(K_k)$ in the same way as above, and we again obtain $e_r(\varphi)\geq e_b(\varphi)+\floor{k/2}-\delta k^2$. Let $\xi':\R^2\to[0,1]$ be the step function defined by
$$\xi'(x,y)=\begin{cases}1&(x,y)\in\big(0,\frac{1}{2}\big)^2\cup\big(\frac{1}{2},1\big)^2\\[3pt]2\gamma-1&(x,y)\in\big(0,\frac{1}{2}\big)\times\big(\frac{1}{2},1\big)\cup\big(\frac{1}{2},1\big)\times\big(0,\frac{1}{2}\big)\\[3pt]0&\text{otherwise}\end{cases}\,.$$
For all $\blambda=(\lambda_0,\lambda_1,\dots)\in\Lambda$, define the graphon $W'_{\blambda}\in\cW$ by
\begin{equation*}
W'_{\blambda}(x,y)=\sum_{0\leq i<|\blambda|}\xi'\left(\frac{x-\lambda_{i-1}}{\lambda_i-\lambda_{i-1}},\frac{y-\lambda_{i-1}}{\lambda_i-\lambda_{i-1}}\right)\,,
\end{equation*}
and let $\cV'$ denote the set of all such graphons. It follows that the graphon $h$ defined in the same way as in \pref{eqn:VarProbhDefn} (but with the new coloring $\varphi$) satisfies $\delta_\square(h,\cV')=0$. Since $\epsilon>0$ was arbitrary we have $\delta_\square(W,\cV')=0$. Since $W$ takes the two values $\{2\gamma-1,1\}$ almost everywhere and $W^\ast_\gamma\in\cV'$ (defined in \pref{eqn:deffstarlambdagraphon}) is the unique graphon in $\cV'$ with that same property, it follows that $W$ is equivalent to $W^\ast_\gamma$, proving $\widehat{\cX}^\ast_\gamma\subseteq\widehat{\cV}_\gamma$.
\end{proof}

\begin{proof}[Proof of \Cref{thm:CnmRateFunction}]
According to \Cref{prop:FsubsetLambda}, $\cV_\gamma$ is the set of graphons (up to equivalence) achieving the supremum in the variational problem \pref{eqn:variationalproblem}. It is easy to calculate that for all $\gamma\in(0,1)$ and all $W\in\cV_\gamma$, we have $H(W)=r^\ast(\gamma)$. The result now follows immediately from \Cref{lemma:FnmGraphonOptAsymps}.
\end{proof}

\begin{proof}[Proof of \Cref{thm:FixedDensityVarProblem}]
The conclusions of the theorem are immediate from the proof of \Cref{thm:CnmRateFunction}, which shows $\phi(\gamma)=r^\ast(\gamma)$, and \Cref{prop:FsubsetLambda}, which characterizes the optimal graphons achieving $\phi(\gamma)$.
\end{proof}

\begin{proof}[Proof of \Cref{thm:LDPforGNP}]
Define the set of claw-free graphons $\cX:=\{W\in\cW:t_{\ind}(K_{1,3},W)=0\}$ so that the variational problem $\psi(p)$ can be stated
\begin{equation}
\psi(p)=\sup\{I_p(W):W\in\cX\}\,.\label{eqn:variationalproblemGNP}
\end{equation}
As in the variational problem \pref{eqn:variationalproblem}, the compactness of $\cX$ and upper semicontinuity of $I_p$ imply the set of maximizers of \pref{eqn:variationalproblemGNP}, denoted $\cX^p_\ast$, is nonempty. Let $r^\ast:(0,1)\to\R$ be the function defined in \pref{eqn:ratefunction}. We compute
\begin{align}
\psi(p) &= \sup_{\gamma\in[0,\frac{1}{2}]}\,\sup_{f\in\cX_\gamma}I_p(f) \nonumber\\
&= \sup_{\gamma\in[0,\frac{1}{2}]}\,\sup_{f\in\cX_\gamma}\left\{H(f)+\log_2\left(\frac{p}{1-p}\right)\int_{[0,1]^2}f+\log_2(1-p)\right\} \nonumber\\
&= \sup_{\gamma\in[0,\frac{1}{2}]}\left\{\sup\{H(f):f\in\cX_\gamma\}+\gamma\log_2\left(\frac{p}{1-p}\right)\right\}+\log_2(1-p) \nonumber\\
&= \sup_{\gamma\in[0,\frac{1}{2}]}\left\{r^\ast(\gamma)+\gamma\log_2\left(\frac{p}{1-p}\right)\right\}+\log_2(1-p)\,.\label{eqn:gnpldpsup}
\end{align}
Notice that for all $p\in(0,1)$ such that
\begin{equation}
\log_2\left(\frac{p}{1-p}\right)+\frac{5+\sqrt{5}}{10}H\left(\frac{3-\sqrt{5}}{2}\right)<0\,,\label[inequality]{ineq:gnpldpnonpositive}
\end{equation}
the quantity $r^\ast(\gamma)+\gamma\log_2\big(\frac{p}{1-p}\big)$ is negative for all $\gamma\in[0,1]$ and the supremum in \pref{eqn:gnpldpsup} is achieved at $\gamma=0$; in this case \pref{eqn:gnpldpsup} equals $\log_2(1-p)$. \Cref{ineq:gnpldpnonpositive} holds if and only if $p\in\big(0,\frac{3-\sqrt{5}}{2}\big)$. Equality in \pref{ineq:gnpldpnonpositive} holds if and only if $p=\frac{3-\sqrt{5}}{2}$, in which case the supremum in \pref{eqn:gnpldpsup} is achieved by all $\gamma\in\big(0,\frac{5-\sqrt{5}}{4}\big]$ and again takes the value $\log_2(1-p)$. For all $p\in\big(\frac{3-\sqrt{5}}{2},1\big)$, the equation
$$\frac{\partial}{\partial\gamma}\left(r^\ast(\gamma)+\gamma\log_2\left(\frac{p}{1-p}\right)\right)=0$$
implies $\gamma=\frac{1}{2}(1+p)$, so we deduce that
$$\psi(p)=r^\ast\left(\frac{1+p}{2}\right)+\frac{1+p}{2}\log_2\left(\frac{p}{1-p}\right)+\log_2(1-p)=\frac{\log_2p}{2}\,.$$
Hence we have proven $\psi(p)=-r_\ast(p)$, where $r_\ast$ is the function defined in \pref{eqn:ratefunctionGNP}. By directly applying \Cref{lemma:LDPforGNPkl}, this completes the proof of \Cref{thm:LDPforGNP}.
\end{proof}

\begin{proof}[Proof of \Cref{thm:GNPVarProblem}]
The equality $\psi(p)=-r_\ast(p)$ was justified in the proof of \Cref{thm:LDPforGNP}, and the proof also characterized the set of optimal graphons (up to equivalence):
\begin{equation}
\cX_\ast^p=\begin{cases}\{W_0\}&p\in\big(0,\frac{3-\sqrt{5}}{2}\big)\\[5pt]\bigcup_{\gamma\in\big[0,\frac{5-\sqrt{5}}{8}\big]}\cV_\gamma&p=\frac{3-\sqrt{5}}{2}\\[5pt]\cV_{(1+p)/2}&p\in\big(\frac{3-\sqrt{5}}{2},1\big)\end{cases}\,,\label{eqn:CpStarFormalDef}
\end{equation}
where $W_0\equiv0$ is the all-zero graphon, $\cV_0:=\{W_0\}$, and $\cV_\gamma$ is as defined after \pref{eqn:deffstarlambdagraphon}. This completes the proof of \Cref{thm:GNPVarProblem}.
\end{proof}

A key consequence of the solutions to the variational problems $\phi(\gamma)$ and $\psi(p)$ is that the optimizers give a rough structural description of claw-free graphs. The following propositions are adaptations of \cite[Theorem~3.1]{chatterjee2011large}. As in the preceding proofs, let $\cX^\ast_\gamma$ and $\cX^p_\ast$ denote the sets of optimal graphons for the variational problems $\phi(\gamma)$ and $\psi(p)$.

\begin{proposition}\label{prop:CNMtypstrucconddistr}
Let $\gamma\in(0,1)$, $n,m\in\N$, and $m\sim\gamma\binom{n}{2}$. Let $G$ be the uniformly random element of $\cC(n,m)$. For all $\epsilon>0$ and large enough $n$,
$$\bbP\{\delta_\square(G,\cX^\ast_\gamma)\geq\epsilon\}\leq e^{-Cn^2}\,,$$
where $C>0$ is a constant depending only on $\epsilon$ and $\gamma$.
\end{proposition}
\begin{proof}
Let $G\sim G(n,1/2)$. Fix $\epsilon>0$ and define the set $\cC^\epsilon(n,m):=\{G\in\cC(n,m):\delta_\square(G,\cX^\ast_\gamma)\geq\epsilon\}$. Let $\widetilde{\cC}^\epsilon\subseteq\cW$ denote the set of graphons that are limits of sequences of graphs in $\cC^\epsilon(n,m)$. Using \Cref{thm:CnmRateFunction}, and applying \Cref{lemma:limsupentropyupperbound} to $\cC^\epsilon$, we compute that
\begin{align*}
\limsup_{n\to\infty}\frac{1}{\binom{n}{2}}\log_2\bbP\{\delta_\square(G,\cX^\ast_\gamma)\geq\epsilon\,|\,G\in\cC(n,m)\} &= \limsup_{n\to\infty}\frac{1}{\binom{n}{2}}\log_2\left(\frac{\bbP\{G\in\cC^\epsilon(n,m)\}}{\bbP\{G\in\cC(n,m)\}}\right) \\
&= \limsup_{n\to\infty}\frac{1}{\binom{n}{2}}\log_2|\cC^\epsilon(n,m)|-r^\ast(\gamma) \\
&\leq \sup_{W\in\widetilde{\cC}^\epsilon}H(W)-r^\ast(\gamma)\,.
\end{align*}
It suffices to show $r^\ast(\gamma)>\sup\{H(W):W\in\widetilde{\cC}_\epsilon\}=:I$. We always have $r^\ast(\gamma)\geq I$ since $\widetilde{\cC}_\epsilon\subseteq\cX_\gamma$. If equality were to hold then $H(W)=r^\ast(\gamma)$ for some $W\in\widetilde{\cC}_\epsilon$ (since $\widetilde{\cC}_\epsilon$ is a compact set), which implies $W$ is equivalent to a graphon in $\cX^\ast_\gamma$ (by \Cref{prop:FsubsetLambda}), but this contradicts $\delta_\square(W,\cX^\ast_\gamma)\geq\epsilon>0$.
\end{proof}

\begin{proposition}\label{prop:GNPtypstrucconddistr}
Let $p\in(0,1)$ be a constant and let $G\sim G(n,p)$. For all $\epsilon>0$ and large enough $n$,
$$\bbP\{\delta_\square(G,\cX_\ast^p)\geq\epsilon\,|\,G\in\cC(n)\}\leq e^{-Cn^2}\,,$$
where $C>0$ is a constant depending only on $\epsilon$ and $p$.
\end{proposition}
\begin{proof}
Fix $\epsilon>0$ and define the set $\cC^\epsilon(n):=\{G\in\cC(n):\delta_\square(G,\cX_\ast^p)\geq\epsilon\}$. Let $\widetilde{\cC}^\epsilon\subseteq\cW$ denote the set of graphons that are limits of sequences of graphs in $\cC^\epsilon(n,m)$. Using \Cref{thm:LDPforGNP}, and applying \Cref{lemma:limsupentropyupperbound} to $\cC^\epsilon$, we obtain that
\begin{align*}
\limsup_{n\to\infty}\frac{1}{\binom{n}{2}}\log_2\bbP\{\delta_\square(G,\cC_\ast^p)\geq\epsilon\,|\,G\in\cC(n)\} &= \limsup_{n\to\infty}\frac{1}{\binom{n}{2}}\log_2\left(\frac{\bbP\{G\in\cC^\epsilon(n)\}}{\bbP\{G\in\cC(n)\}}\right) \\
&= r_\ast(p)-\limsup_{n\to\infty}\frac{1}{\binom{n}{2}}\log_2\bbP\{G\in\cC^\epsilon(n)\} \\
&\leq r_\ast(p)-\sup_{W\in\widetilde{\cC}^\epsilon}I_p(W)\,.
\end{align*}
It suffices to show $r_\ast(p)<\sup\{I_p(W):W\in\widetilde{\cC}^\epsilon\}=:I$. We always have $r_\ast(p)\leq I$ since $\widetilde{\cC}^\epsilon\subseteq\cX$. If equality were to hold then $I_p(W)=-r_\ast(p)$ for some $W\in\widetilde{\cC}^\epsilon$ (since $\widetilde{\cC}^\epsilon$ is a closed set), so $W\in\cX_\ast^p$, but this contradicts $\delta_\square(W,\cX_\ast^p)\geq\epsilon>0$.
\end{proof}

\section{Exponential Penalties for Defect Edges}
\label{sec:JansonPenalties}
In this section we prove inequalities that will be used several times in \Cref{sec:supercritical,sec:subcritical}. Let $\epsilon\in(0,1/2^{64})$, 
$\alpha:=8\epsilon^{1/8}$, and $\eta:=\epsilon^{1/16}$. Recall that $V$ denotes a vertex set of size $n$. Let $\sB=\sB_n$ denote the set of all bipartitions $V=A\cup B$ such that
$$\left(\frac{1}{2}-\alpha\right)n\leq\min\{|A|\,,\,|B|\}\leq\max\{|A|\,,\,|B|\}\leq\left(\frac{1}{2}+\alpha\right)n\,.$$

\begin{lemma}\label{lemma:JansonPenaltyMedVtx}
Let $\Pi=\{A,B\}\in\sB$ and fix an integer $0\leq m\leq|A|\cdot|B|$. Assume the following hypotheses:
\begin{enumerate}[(\emph{\roman*})]
	\item Let $v\not\in V$ be an additional vertex and let $N\subseteq A$ be a subset of size $\alpha|A|\leq|N|\leq(1-\alpha)|A|$. Let $N'\subseteq B$ be a subset of size at least $\alpha|B|$. Let $F$ denote the star centered at $v$ with leaves $V\setminus(N\cup N')$.
	\item Let $T\subseteq\Pi^c$ be a graph with at most $\epsilon n^2/\eta^2$ edges.
	\item Let $H\subseteq\Pi$ be the uniformly random bipartite graph with $m$ edges. Let $G$ be the random graph on the vertex set $V\cup\{v\}$ and edge set $E(H)\cup E(F)\cup(E(\Pi^c)\setminus E(T))$.
\end{enumerate}
For sufficiently large $n$,
$$\bbP\{G\in\cC(n+1)\}\leq4n\sqrt{p}\cdot e^{-\epsilon^{3/4}pn^2}\,,$$
where $p:=m/(|A|\cdot|B|)$.
\end{lemma}
\begin{proof}
Let $H'\subseteq\Pi$ be the random bipartite graph obtained by including each edge independently with probability $p$, and let $G'$ be the random graph on the vertex set $V':=V\cup\{v\}$ and edge set $E(H')\cup E(F)\cup(E(\Pi^c)\setminus E(T))$. Using \Cref{lemma:GnmLeqGnp}, it is easily seen that
$$\bbP\{G\in\cC(n+1)\}\leq4n\sqrt{p}\cdot\bbP\{G'\in\cC(n+1)\}\,.$$
In the remainder, we will define a class of copies $K_{1,3}\subseteq K_{V'}$ and estimate the probability that $G'$ contains none of these copies as induced subgraphs. If $C\cong K_{1,3}$ then let $c(C)$ denote the vertex in $C$ of degree three.

\medskip\noindent \underline{\textbf{Definition and size of $\sK$.}}
Let $N_1\subseteq N$ and $N_2\subseteq A\setminus N$ be subsets each of size at least $\alpha|A|$. Let $T':=(\Pi\cup T)^c$ and notice that since $e(T)\leq\epsilon n^2/\eta^2$ and $|A|,|B|\geq n/4$, we have
$$e_{T'}(N_1,N_2)\geq(\alpha|A|)^2-\epsilon n^2/\eta^2\geq2\epsilon^{1/4}n^2\,.$$
Let $J\subseteq K_{N_1,N_2}$ be a bipartite graph with at least $\epsilon^{1/4}n^2$ edges, and denote the vertex sets $U_{A,1}:=V(J)\cap N_1$ and $U_{A,2}:=V(J)\cap N_2$. Notice that since $|A|\geq n/4$, we have $|N'|\geq\alpha n/4$.

Let $\sK$ denote the set of all subgraphs $C'\subseteq K_{V'}$ isomorphic to $K_{1,3}$ such that $V(C')=\{v,w,x,y\}$, $x\in U_{A,1}$, $w=c(C')\in U_{A,2}$, and $y\in N'$. We have
$$|\sK|=e(J)\cdot|N'|\geq\epsilon^{1/4}n^2\cdot\frac{\alpha n}{4}=2\epsilon^{3/8}n^3\,.$$

\medskip For all $C'\in\sK$ let $E_{C'}$ denote the event $\{G'[V(C')]=C'\}$. Let $\cI$ be the lattice of events generated by $\{E_{C'}\}_{C'\in\sK}$, i.e. by taking unions and intersections but not complements of the events $E_{C'}$. For all distinct $C_1,C_2\in\sK$, the events $E_{C_1}$ and $E_{C_2}$ are not independent if and only if $C_1$ and $C_2$ share a vertex in $U_{A,1}\cup U_{A,2}$ and in $N'$, and in this case the probability both events $E_{C_1}$ and $E_{C_2}$ occur is $p^3$. We now calculate
\begin{align*}
\mu &:= \sum_{C'\in\sK}\bbP\{E_{C'}\} \geq |\sK|\cdot p^2 \geq 2\epsilon^{3/8}p^2n^3 \,, \\
\Delta &:= \sum_{C_1\sim C_2}\bbP\{E_{C_1}\wedge E_{C_2}\} \leq \sum_{w\in U_{A,1}\cup U_{A,2}}\binom{d_J(w)}{2}\cdot|N'|\cdot p^3 \leq p^3n^4 \,,
\end{align*}
where the sum defining $\Delta$ is over unordered pairs $C_1\sim C_2$ such that $E_{C_1}$ and $E_{C_2}$ are not independent. Since the family $\cI$ and the events $\{E_{C'}\}_{C'\in\sK}\subseteq\cI$ satisfy the hypotheses of Janson's inequality (\Cref{thm:JansonsInequality}), we obtain
$$\bbP\{G'\in\cC(n+1)\} \leq \bbP\left\{\bigwedge_{C'\in\sK}\overline{E_{C'}}\right\} \leq \exp\left(-\min\left\{\frac{\mu}{2}\,,\,\frac{\mu^2}{4\Delta}\right\}\right)=e^{-\mu^2/4\Delta}\leq e^{-\epsilon^{3/4}pn^2}\,,$$
completing the proof.
\end{proof}

\begin{lemma}\label{lemma:JansonPenaltyHighDegVtx}
Let $n_1,\,n_2\in\N$, $\Pi_1=\{A_1,B_1\}\in\sB_{n_1}$, and $\Pi_2=\{A_2,B_2\}\in\sB_{n_2}$. Fix integers $0\leq m_1\leq |A_1|\cdot|B_1|$ and $0\leq m_2\leq|A_2|\cdot|B_2|$. Let $n:=n_1+n_2$ and $V=[n]$, and assume $n_1\geq\eta n$ and $n_2\geq n_1/8$. Assume the following hypotheses:
\begin{enumerate}[(\emph{\roman*})]
	\item Let $T\subseteq K_V\setminus(E(\Pi_1)\cup E(\Pi_2))$ be a graph with at most $\epsilon n^2/\eta^2$ edges.
	\item Assume there exists $v\in A_2$ such that $d_T(v,A_1)\geq\alpha|A_1|$ and $d_T(v,B_1)>(1-\alpha)|B_1|$.
	\item For $i=1,2$, let $H_i\subseteq\Pi_i$ be the uniformly random bipartite graph with $m_i$ edges (here the graphs $H_1$ and $H_2$ are independent). Let $G$ be the random graph on the vertex set $V$ and edge set $(E(H_1)\cup E(H_2)\cup E(\Pi_1^c)\cup E(\Pi_2^c))\,\Delta\,E(T)$, where $\Delta$ is the symmetric difference.
\end{enumerate}
For sufficiently large $n$,
$$\bbP\{G\in\cC(n)\}\leq8n^2\sqrt{p_1p_2}\cdot\exp\left(-\frac{\alpha^2\eta p_1p_2}{2^{32}}\cdot n_1^2\right)\,,$$
where $p_i:=m_i/(|A_i|\cdot|B_i|)$ for $i=1,2$.
\end{lemma}
\begin{proof}
For $i=1,2$, let $H'_i\subseteq\Pi_i$ be the random bipartite graph obtained by including each edge independently with probability $p_i:=m_i/(|A_i|\cdot|B_i|)$. (The graphs $H'_1$ and $H'_2$ are assumed to be independent.) Let $G'$ be the random graph on the vertex set $V$ and edge set $(E(H'_1)\cup E(H'_2)\cup E(\Pi_1^c)\cup E(\Pi_2^c))\,\Delta\,E(T)$. Using \Cref{lemma:GnmLeqGnp} and independence of the graphs $H'_i$, it is easily seen that
$$\bbP\{G\in\cC(n)\}\leq8n^2\sqrt{p_1p_2}\cdot\bbP\{G'\in\cC(n)\}\,.$$
In the remainder, we will define a class of copies $K_{1,3}\subseteq K_V$ and estimate the probability that $G'$ contains none of these copies as induced subgraphs. If $C\cong K_{1,3}$ then let $c(C)$ denote the vertex in $C$ of degree three.

\begin{enumerate}[wide, labelwidth=!, labelindent=0pt, leftmargin=0pt]
	\item[\underline{\textbf{Definition and size of $\sK$.}}] Let $N:=N_T(v,A_1)$ and $N':=N_T(v,B_1)$. Since $\Pi_1\in\sB_{n_1}$, we have $|A_1|,|B_1|\geq n_1/4$, so the hypotheses on $N$ and $N'$ imply $|N|\geq\alpha n_1/4$ and $|N'|\geq n_1/8$. \Cref{lemma:erdosgallai} implies there exists a matching $M\subseteq K_{N,N'}$ with at least $\alpha n_1/2^8$ edges. We claim there exists a subset $L\subseteq B_2$ of size at least $|B_2|/2$ such that for all $w\in L$ we have
	$$|\{e\in E(M):e\subseteq\overline{N}_T(w)\}|\geq\frac{\alpha n_1}{2^9}\,.$$
	Indeed, if this were not the case, then since $|B_2|\geq n_2/4\geq n_1/32$, the induced bipartite graph $T[A_1\cup B_1,L]$ would contain at least $\alpha n_1^2/2^{15}\geq\alpha\eta^2n^2/2^{15}$ edges, contradicting $e(T)\leq\epsilon n^2/\eta^2$. For all $w\in L$ let
	$$L_w:=\{e\in E(M):e\subseteq\overline{N}_T(w)\}\,,$$
	which we have shown contains at least $\alpha n_1/2^9$ edges. Let $\sK$ denote the set of all subgraphs $C'\subseteq K_V$ isomorphic to $K_{1,3}$ such that, letting $V(C')=\{v,w,x,y\}$, $c(C')=v$, $w\in L$, and $xy\in L_w$. (To emphasize, $xy$ is a \emph{missing} edge in such a copy $C'$.) Hence
	$$|\sK|=\sum_{w\in L}|L_w|\geq\frac{|B_2|}{2}\cdot\frac{\alpha n_1}{2^9}\geq\frac{\alpha n_1^2}{2^{15}}\,.$$
\end{enumerate}

For all $C'\in\sK$ let $E_{C'}$ be the event that $G'[V(C')]=C'$ and notice that $\bbP\{E_{C'}\}=p_1p_2$. We compute that
	\begin{align*}
	\mu &:= \sum_{C'\in\sK}\bbP\{E_{C'}\} = |\sK|\cdot p_1p_1 \geq \frac{\alpha n_1^2}{2^{15}}\cdot p_1p_2 \,, \\
	\Delta &:= \sum_{C_1\sim C_2}\bbP\{E_{C_1}\wedge E_{C_2}\} \leq e(M)\cdot|L|\cdot p_1p_2 \leq p_1p_2\cdot n_1n_2 \leq p_1p_2\cdot n_1^2/\eta \,,
	\end{align*}
	where the sum defining $\Delta$ is over unordered pairs $C_1,C_2\in\sK$ such that $E_{C_1}$ and $E_{C_2}$ are not independent, and we used that $n_2\leq n_1/\eta$. Let $\cI$ be the lattice of events generated by $\{E_{C'}\}_{C'\in\sK}$, i.e. by taking unions and intersections but not complements of the events $E_{C'}$. Since the family $\cI$ and the events $\{E_{C'}\}_{C'\in\sK}\subseteq\cI$ satisfy the hypotheses of Janson's inequality (\Cref{thm:JansonsInequality}), we obtain
	\begin{align*}
	\bbP\{G'\in\cC(n)\} &\leq \bbP\left\{\bigwedge_{C'\in\sK}\{G'[V(C')]\neq C'\}\right\} \\
	&\leq \exp\left(-\min\left\{\frac{\mu}{2}\,,\,\frac{\mu^2}{4\Delta}\right\}\right)\leq e^{-\mu^2/4\Delta} \leq \exp\left(-\frac{\alpha^2\eta p_1p_2}{2^{32}}\cdot n_1^2\right) \,,
	\end{align*}
	completing the proof.
\end{proof}

\begin{lemma}\label{lemma:JansonPenaltyMatching}
Let $\Pi=\{A,B\}\in\sB$ and fix integers $0\leq m\leq|A|\cdot|B|$, $0\leq d\leq\alpha n$, $0\leq k\leq n/8$, $0\leq l\leq\epsilon n^2$. Assume the following hypotheses:
\begin{enumerate}[(\emph{\roman*})]
	\item Let $D\subseteq A$ be a vertex subset of size $d$ and let $D'$ be a set of size $d$ that is disjoint from $V$. Let $M'$ be a perfect matching between $D$ and $D'$. Let $F'\subseteq K_{D',B}$ be a bipartite graph such that $d_{F'}(v,B)\leq\alpha|B|$ for all $v\in D'$.
	\item Let $C$ be a finite set of size $|C|\leq n/\eta$ that is disjoint from $V\cup D'$. Let $F\subseteq K_{V,C}$ be a bipartite graph with $l$ edges. Assume that for all $v\in C$, $d_F(v,A)<\alpha|A|$ and $d_F(v,B)<\alpha|B|$.
	\item Let $T\subseteq\Pi^c$ be a graph with at most $\epsilon n^2/\eta^2$ edges such that for all $v\in P\in\Pi$, we have $d_T(v)\leq\alpha|P|$. Assume $T$ has a matching $M$ covering $k$ vertices.
	\item Let $H\subseteq\Pi$ be the uniformly random bipartite graph with $m$ edges. Let $G$ be the random graph on the vertex set $V\cup C\cup D$ and edge set $E(H)\cup E(F)\cup E(M')\cup(E(\Pi^c)\setminus E(T))$.
\end{enumerate}
Letting $n':=n+|C|+|D'|$,
$$\bbP\{G\in\cC(n')\}\leq4n\cdot p^{kn/2^{10}+dn/16}\cdot\exp\left(-\frac{p^4}{2^{20}}\cdot n\cdot\min\{\eta^2 l,\,n\}\right)\,,$$
where $p:=m/(|A|\cdot|B|)$.
\end{lemma}
\begin{proof}
Let $H'\subseteq\Pi$ be the random bipartite graph obtained by including each edge independently with probability $p$, and let $G'$ be the random graph on the vertex set $V':=V\cup C\cup D'$ and edge set $E(H')\cup E(F)\cup E(M')\cup(E(\Pi^c)\setminus E(T))$. Using \Cref{lemma:GnmLeqGnp}, we see that
$$\bbP\{G\in\cC(n')\}\leq4n\cdot\bbP\{G'\in\cC(n')\}\,.$$
In the remainder, we will define three classes of copies $K_{1,3}\subseteq K_{V'}$ and estimate the probability that $G'$ contains none of these copies as induced subgraphs. If $C\cong K_{1,3}$ then let $c(C)$ denote the vertex in $C$ of degree three.

\begin{enumerate}[wide, labelwidth=!, labelindent=0pt, leftmargin=0pt]
	\item[\underline{\textbf{Definition and size of $\sK_1$.}}] Since $|V(M)|=k$, one of the two matchings $M[A]$ or $M[B]$ covers at least $k/2$ vertices, so by symmetry we may assume $e(M[A])\geq k/4$. Let $X:=V(M[A])$ denote the vertices covered by $M[A]$. Since $e(T)\leq\epsilon n^2$ and $|B|\geq n/4$, we have
	$$e(K_B\setminus T)\geq\binom{n/4}{2}-\frac{\epsilon n^2}{\eta^2}\geq\frac{n^2}{2^7}\,.$$
It follows from \Cref{lemma:erdosgallai} that $K_B\setminus T$ has a matching $N$ of size $n/2^{10}\leq e(N)\leq n/2^9$ for large enough $n$. Let $Y\subseteq V(N)$ be a set of vertices such that $|e\cap Y|=1$ for all $e\in E(N)$.

	Let $\sK_1$ denote the set of all subgraphs $C'\subseteq K_{V'}$ isomorphic to $K_{1,3}$ such that $y:=c(C')\in Y$ and, if we denote $V(C')=\{w,x,y,z\}$, then $wx\in E(M)$ and $yz\in E(N)$. (To emphasize, $wx$ is a \emph{missing} edge in such $C'$.) We have $|\sK_1|=e(M[A])\cdot e(N)\geq kn/2^{12}$.
	\item[\underline{\textbf{Definition and size of $\sK_2$.}}] Since $v(N)\leq n/2^8$, there exists a subset $Z\subseteq B\setminus V(N)$ of size $n/16\leq|Z|\leq n/8$. Let
	$$Q\in\{D,X,A\setminus(X\cup D),V(N),Z,B\setminus(V(N)\cup Z)\}$$
	be the vertex set such that $e_F(Q,C)$ is maximum, and notice that $e_F(Q,C)\geq l/8$. Let $P\in\Pi$ denote the part such that $Q\subseteq P$, and let $P'\in\Pi\setminus\{P\}$ denote the other part. Define the set
	$$R:=\begin{cases}X\cup D&\text{if }Q=A\setminus(X\cup D)\\A\setminus(X\cup D)&\text{if }Q\in\{X,D\}\\V(N)\cup Z&\text{if }Q=B\setminus(V(N)\cup Z)\\B\setminus(V(N)\cup Z)&\text{if }Q\in\{V(N),Z\}\end{cases}$$
	and the set $S:=P'\setminus(X\cup V(N)\cup D\cup Z)$. For all $v\in C$ define the set $U_v:=\overline{N}_F(v,P')$, which has size at least $(1-\alpha)|P'|$ by hypothesis. For all $v\in C$ define the set $K_v:=U_v\cap S$. For all $wx\in E_F(Q,C)$ with $w\in C$ and $x\in Q$, let $L_{wx}:=\overline{N}_F(w,R)\cap\overline{N}_T(x,R)$.
	
	Let $\sK_2$ denote the set of all subgraphs $C'\subseteq K_{V'}$ isomorphic to $K_{1,3}$ such that, if we denote $V(C')=\{w,x,y,z\}$, then $wx\in E_F(Q,C)$, $w\in C$, $x=c(C')\in Q$, $y\in L_{wx}$, and $z\in K_w$. We will now estimate the size of $\sK_2$.
	
	Since $|X|\leq n/8$, $|D|\leq\alpha n$, and $|V(N)|\leq n/2^9$, and $|Z|\leq n/8$, we have
\begin{align*}
|S| &= |P'\setminus(X\cup V(N)\cup D\cup Z)| \\
&\geq |P'|-|P'\cap(X\cup V(N)\cup D\cup Z)| \\
&\geq\left(\frac{1}{2}-\alpha\right)n-\max\{|X|+|D|,|V(N)|+|Z|\} \geq \frac{7n}{16}-\frac{3n}{16}=\frac{n}{4}\,,
\end{align*}
and the same calculation proves $|R|\geq n/4$. Notice that for all $v\in C$, $K_v=S\setminus(P'\setminus U_v)$, hence the inequalities $|P'\setminus U_v|\leq\alpha|P'|$ and $\alpha<1/16$ imply
\begin{equation}
|K_v| \geq |S|-|P'\setminus U_v|\geq\frac{n}{4}-\alpha|P'|\geq\frac{n}{8}\,.\label{eqn:Kvineq}
\end{equation}
	Using the fact that $d_F(w,P)\leq\alpha|P|$ and $d_T(x,P)\leq\alpha|P|$ and $\alpha<1/32$, we compute
	$$|L_{wx}|\geq|R|-d_F(w,P)-d_T(x,P)\geq\frac{n}{4}-2\alpha|P|\geq\frac{n}{8}\,.$$
	We compute that
	$$|\sK_2|\geq\sum_{\substack{wx\in E_F(Q,C)\\w\in C,\,x\in Q}}|K_w|\cdot|L_{wx}|\geq\frac{n^2}{64}\cdot e_F(Q,C)\geq\frac{ln^2}{2^{9}}\,,$$
	where we used that $e_F(Q,C)\geq l/8$.
	
	\item[\underline{\textbf{Definition and size of $\sK_3$.}}] Since $d=|D|\leq\alpha n$ and $d_T(v)\leq\alpha|P|$ for all $v\in V$ and $P\in\Pi$, there exists a subset $D''\subseteq P\setminus D$ such that the induced bipartite graph $(\Pi\cup T)^c[D,D'']$ contains a perfect matching $J$ between $D$ and $D''$. For all $y\in D'$ let $Z_y:=Z\cap\overline{N}_{F'}(y,B)$. Since $d_{F'}(v,B)\leq\alpha|B|$, we have
	$$|Z_y|\geq|Z|-d_{F'}(v,B)\geq\frac{n}{16}-\frac{\alpha n}{4}\geq\frac{n}{32}\,.$$
	Let $\sK_3$ be the set of all subgraphs $C'\subseteq K_{V'}$ isomorphic to $K_{1,3}$ such that, if we denote $V(C')=\{w,x,y,z\}$, then $w=c(C')\in D$, $wx\in E(J)$, $y\in D'$, $wy\in E(M')$, and $z\in Z_y$. It follows directly from the definitions that
	$$|\sK_3|=\sum_{w\in D}|Z_y|\geq \frac{dn}{32}\,.$$
\end{enumerate}

With the definitions of $\sK_1$, $\sK_2$, and $\sK_3$, we now analyze the probability $G'$ contains no induced copy $C'\in\sK_1\cup\sK_2\cup\sK_3$. For all $C'\in\sK_1\cup\sK_2\cup\sK_3$ let $E_{C'}$ denote the event $\{G'[V(C')]=C'\}$. First consider copies in $\sK_2$. Let $\cI$ be the lattice of events generated by $\{E_{C'}\}_{C'\in\sK_2}$, i.e. by taking unions and intersections but not complements of the events $E_{C'}$. The events $\{E_{C'}\}_{C'\in\sK_2}$ are not independent, but the family $\cI$ and the events $\{E_{C'}\}_{C'\in\sK_2}$ satisfy the hypotheses of Janson's inequality (\Cref{thm:JansonsInequality}). Since $\bbP\{E_{C'}\}=p^2$ for all $C'\in\sK_2$, we compute
\begin{align*}
\mu &:= \sum_{C'\in\sK_2}\bbP\{E_{C'}\} = |\sK_2|\cdot p^2 \geq \frac{p^2ln^2}{2^9} \,, \\
\Delta &:= \sum_{\substack{C_1,C_2\in\sK_2\\C_1\sim C_2}}\bbP\{E_{C_1}\wedge E_{C_2}\} \\
&\leq \sum_{\substack{wx\in E_F(Q,C)\\w\in C,x\in Q}}\binom{d_F(x,C)}{2}\left(|L_{wx}|+|K_w|\right)+|A|\cdot|B|\cdot\binom{e(F)}{2}\leq \frac{l\cdot|C|^2}{2}\cdot n+\frac{l^2n^2}{4}\,,
\end{align*}
where the sum defining $\Delta$ is over unordered pairs $C_1,C_2$ such that $E_{C_1}$ and $E_{C_2}$ are not independent. To calculate $\Delta$, we observed that if $E_{C_1}$ and $E_{C_2}$ are not independent then at least one of the following cases holds (letting $V(C_i)=\{w_i,x_i,y_i,z_i\}$ for $i=1,2$):
\begin{enumerate}[(\emph{\roman*})]
	\item $x_1=x_2\in Q$ and $z_1=z_2\in S$. In this case $E_{C_1}$ and $E_{C_2}$ are not independent since they share the edge $x_1z_1=x_2z_2$. There are at most $l\cdot|C|^2n/2$ such pairs of events.
	\item $y_1=y_2\in R$ and $z_1=z_2\in S$. In this case $E_{C_1}$ and $E_{C_2}$ are not independent since they share the edge $y_1z_1=y_2z_2$. There are at most $l^2n^2/4$ such pairs of events.
\end{enumerate}
Since $|C|\leq n/\eta$, we compute that
$$\frac{\mu^2}{4\Delta}\geq\frac{p^4l^2n^4/2^{18}}{2l\cdot n^3/\eta^2+l^2n^2}\geq\frac{p^4ln^2}{2^{20}\cdot\max\{n/\eta^2,\,l\}}=\frac{p^4}{2^{20}}\cdot n\cdot\min\{\eta^2 l,\,n\}\,.$$
Now consider the copies in $\sK_1\cup\sK_3$. By definition of $\sK_1$, the events $\{E_{C'}\}_{C'\in\sK_1}$ are independent, and we have $\bbP\{E_{C'}\}=p^4$ for all $C'\in\sK_1$. Similarly, the events $\{E_{C'}\}_{C'\in\sK_3}$ are independent, and we have $\bbP\{E_{C'}\}=p^2$ for all $C'\in\sK_3$. Additionally, the three collections of events $\{E_{C'}\}_{C'\in\sK_i}$ for $i=1,2,3$ are pairwise independent as collections. We thus compute that
\begin{align*}
\bbP\{G'\in\cC(n')\} &\leq \bbP\left\{\bigwedge_{C'\in\sK_1\cup\sK_2\cup\sK_3}\overline{E_{C'}}\right\} \\
&= \bbP\left\{\bigwedge_{C'\in\sK_1}\overline{E_{C'}}\right\}\bbP\left\{\bigwedge_{C'\in\sK_2}\overline{E_{C'}}\right\}\bbP\left\{\bigwedge_{C'\in\sK_3}\overline{E_{C'}}\right\} \\[5pt]
&\leq p^{4|\sK_1|+2|\sK_3|}\exp\left(-\min\left\{\frac{\mu}{2}\,,\,\frac{\mu^2}{4\Delta}\right\}\right) \\[5pt]
&= p^{4|\sK_1|+2|\sK_3|}\cdot e^{-\mu^2/4\Delta} \\
&\leq p^{kn/2^{10}+dn/16}\cdot\exp\left(-\frac{p^4}{2^{20}}\cdot n\cdot\min\{\eta^2 l,\,n\}\right)\,,
\end{align*}
which completes the proof of the lemma.
\end{proof}

\section{The Supercritical Regime}
\label{sec:supercritical}
In this section we prove the first and second assertions of \Cref{thm:almostallCB}. Throughout the section, let $\gamma\in\big(\frac{5-\sqrt{5}}{4},1\big)$, $m\sim\gamma\binom{n}{2}$, and $\mu:=\big(\frac{5+\sqrt{5}}{10}\gamma\big)^{1/2}$. There is a global parameter $\epsilon>0$ used throughout the section; in general, all statements hold for sufficiently small $\epsilon$, but in some cases it will be useful to have established that $0<\epsilon<\min\{\gamma-\mu,1-\gamma\}/2^{64}$. Likewise, inequalities $f(n)\leq g(n)$ are always meant for large enough $n$. We examine typical structure at the critical edge density at the end of the section.

\subsection{Parameters and Definitions}
Define the parameters and constant
$$\alpha:=8\epsilon^{1/8}\,,\hspace{8mm}\delta:=8\sqrt{\epsilon}\,,\hspace{8mm}\tau:=\frac{1}{2}\left(\frac{\epsilon}{2^8}\right)^{67}\,,\hspace{8mm}\rho:=2\gamma-1\,.$$ 
Let $W^\ast=f^\ast_\gamma$ be the graphon defined in \pref{eqn:deffstarlambdagraphon}, so $W^\ast$ is the unique optimal graphon (up to equivalence) in the supercritical regime. Define the set of graphs
$$\cC_\far:=\{G\in\cC(n,m):\delta_\square(G,W^\ast)\geq\tau\}\,.$$

For disjoint sets $A,B,C\subseteq V$ with $A$ and $B$ nonempty such that $V=A\cup B\cup C$, call $\Pi=\{\{A,B\},C\}$ a \emph{division} of $V$. For a division $\Pi$, let $\Pi_\cb:=\{A,B\}$ and $\Pi_\sparse:=C$ (thought of as the ``co-bipartite'' and ``sparse'' parts of $\Pi$, respectively). Let $\sD$ denote the set of all divisions of $V$, define the set of \emph{balanced} divisions
$$\sP:=\left\{\Pi\in\sD:\left(\frac{1}{2}-\delta\right)n\leq\min\{|P|:P\in\Pi_\cb\}\leq\max\{|P|:P\in\Pi_\cb\}\leq\left(\frac{1}{2}+\delta\right)n\right\}\,.$$
For all $s\geq0$ define the set of divisions
$$\sP_s:=\{\Pi\in\sP:|\Pi_\sparse|=s\}\,.$$
and define the set of bipartitions $\sB:=\{\Pi_\cb:\Pi\in\sP_0\}$. For all $G\in\cC(n,m)$ and $\Pi=\{\{A,B\},C\}\in\sD$ let
$$b(G,\Pi):=e(G^c[A])+e(G^c[B])+e(G[V,C])\,,$$
let $\Pi(G)\in\sD$ denote a canonically chosen division minimizing $b(G,\Pi)$, and define $b(G):=b(G,\Pi(G))$. For all $G\in\cC(n,m)$ let $\Pi_\cb(G):=(\Pi(G))_\cb$ and $\Pi_\sparse(G):=(\Pi(G))_\sparse$. Throughout this section, we often write $\Pi$ to mean the complete bipartite graph $K_\Pi$ when there is no ambiguity. For all $G\in\cC(n,m)$ with $\Pi(G)=\{\{A,B\},C\}$, define
$$D(G):=E(G^c[A])\cup E(G^c[B])\cup E(G[A\cup B,C])\,,$$
which we view as ``defect'' edges. Let $U(G):=V(D(G))$ denote the set of vertices incident to $D(G)$ and define the graph $T(G):=(U(G),D(G))$. For all $\Pi=\{\{A,B\},C\}\in\sP$ define
\[\arraycolsep=1.4pt
\begin{array}{ll}
\cC_\Pi &:= \{G\in\cC(n,m)\setminus\cC_\far:\Pi(G)=\Pi\,,\,D(G)\neq\emptyset\} \,, \\[5pt]
\cC^\ast_\Pi &:= \{G\in\cC(n,m)\setminus\cC_\far:\Pi(G)=\Pi\,,\,D(G)=\emptyset\} \,,
\end{array}\]
and define the associated sets of graphs
\[\arraycolsep=1.4pt
\begin{array}{ll}
\cT_\Pi &:= \{T(G)\cup G[C]:G\in\cC_\Pi\} \,, \\[5pt]
\cT^\ast_\Pi &:= \{G[C]:G\in\cC^\ast_\Pi\} \,.
\end{array}\]
For all $T\in\cT_\Pi$ and $P\in\Pi_\cb=\{A,B\}$, say that a vertex $v$ has
\begin{enumerate}[(\emph{\roman*})]
	\item \emph{low degree in $P$}  if $d_T(v,P)<\alpha|P|$,
	\item \emph{medium degree in $P$} if $\alpha|P|\leq d_T(v,P)\leq(1-\alpha)|P|$, or 
	\item \emph{high degree in $P$} if $d_T(v,P)>(1-\alpha)|P|$.
\end{enumerate}
If a vertex $v$ has medium degree in some $P\in\{A,B\}$ with respect to $T\in\cT_\Pi$, we say that $v$ has $(\Pi,T)$-medium degree, or simply medium degree. For all $\Pi\in\sP$ and $T\in\cT_\Pi$ let
\[\arraycolsep=1.4pt
\begin{array}{ll}
\cC'_\Pi &:= \{G\in\cC_\Pi:\text{there exists a vertex of $(\Pi,T(G))$-medium degree}\} \,, \\[5pt]
\cC_{\Pi,T} &:= \{G\in\cC_\Pi\setminus\cC'_\Pi:T(G)\cup G[C]=T\} \,.
\end{array}\]

Our first lemma toward proving \Cref{thm:almostallCB} \ref{thm:almostallCBsupercritical} describes some useful properties of graphs that are close to $W^\ast$ in cut metric.

\begin{lemma}\label{lemma:SuperCloseStructure}
For large enough $n$, every graph $G\in\cC(n,m)\setminus\cC_\far$ satisfies the following conditions:
\begin{enumerate}[(\emph{\roman*})]
	\item\label{item:lemmaSuperCloseStructurei} By editing (i.e. adding or removing) at most $\epsilon n^2$ edges of $G$, we can obtain a co-bipartite graph.
	\item\label{item:lem61balancedbipsmallindep} $\Pi(G)\in\sP$ and $|\Pi_\sparse(G)|\leq\delta n/2$. 
	\item\label{item:lem61densityclose} $|\rho-d_G(A,B)|\leq\delta$.
	\item\label{item:lem61novtxhighdeg} For all $P\in\Pi_\cb(G)$ and $v\in P$ it holds that $d_{T(G)}\leq(1-\alpha)|P|$.
\end{enumerate}
\end{lemma}

It follows immediately from \Cref{lemma:SuperCloseStructure} and the above definitions that
\begin{equation}
|\cC(n,m)| \leq |\cB_c(n,m)|+\sum_{s=1}^{\delta n}\,\sum_{\Pi\in\sP_s}|\cC^\ast_\Pi|+\sum_{\Pi\in\sP}\Bigg(\sum_{T\in\cT_\Pi}|\cC_{\Pi,T}|+|\cC'_\Pi|\Bigg)+|\cC_\far|\,.\label{eqn:SupercritKRDecomp}
\end{equation}
Thus to prove \Cref{thm:almostallCB} \ref{thm:almostallCBsupercritical}, it suffices to show that the three rightmost terms in \pref{eqn:SupercritKRDecomp} are much smaller than $|\cC(n,m)|$, which is the focus of the remainder of this section. The proof of \Cref{lemma:SuperCloseStructure} uses the following simple object.

\begin{definition}[Discretization of a graphon]\label{def:DiscreteGraphon}
For all $W\in\cW$ and $n\in\N$, let $H_n$ be the weighted graph on $V$ whose $ij$ edge weight is $W(\frac{i}{n},\frac{j}{n})$ and whose node-weights all equal 1. Define the $n$th discretization of $W$ to be the graphon $W_n:=W_{H_n}$ (recall the definition of $W_H$ for a weighted graph $H$ given in \Cref{subsec:graphlimits}).
\end{definition}

\begin{proof}[Proof of \Cref{lemma:SuperCloseStructure}]
Denote $W:=W^\ast$. For all $n\in\N$, let $H_n$ be the weighted graph on $V$ whose $ij$ edge weight is $W(\frac{i}{n},\frac{j}{n})$ and whose node-weights all equal 1. In the remainder, all statements hold for large enough $n$. The sequence $W_n$ clearly converges to $W$ pointwise almost everywhere and hence in cut metric, so the condition $\delta_\square(G,W)<\tau$ implies $\delta_\square(G,W_n)<2\tau$. By \cite[Lemma~8.9]{lovasz2012large} we have $\delta_\square(G,H_n)=\delta_\square(W_G,W_n)$, and by \cite[Theorem~2.3]{borgs2008convergent} we have
$$\widehat{\delta}(G_1,G_2)\leq32\cdot(\delta_\square(G_1,G_2))^{1/67}$$
for edge-weighted graphs with weights in $[-1,1]$; from these two results and our definition of $\tau$ we deduce that $\widehat{\delta}_\square(G,H_n)<\epsilon/8$. Since the vertex set of $H_n$ can be partitioned into two parts within each of which all edges have weight 1 (we may disregard the one or two isolated vertices of $H_n$), it follows that there is a bipartition $\Pi'=\{A',B'\}$ of $V(G)$ such that each of the graphs $G^c[A']$ and $G^c[B']$ has at most $\epsilon n^2/8$ non-edges. Hence by adding at most $\epsilon n^2/4$ edges to $G[A']$ and $G[B']$, we can obtain a co-bipartite graph (and we made a total of less than $\epsilon n^2$ edits), completing the proof of \ref{item:lemmaSuperCloseStructurei}.

To prove \ref{item:lem61balancedbipsmallindep}, first let $\Pi(G)=\{\{A,B\},C\}$ and note that since $G$ is within $\epsilon n^2$ in edit distance of $\cB_c(n,m)$ by part (\emph{i}), we have $b(G)\leq\epsilon n^2$. Let $H$ be the graph on $n$ vertices with edge set $E(K_A)\cup E(K_B)\cup E(G[A,B])$. If $|C|>\delta n/2$ then $\widehat{\delta}_\square(H,H_n)\geq\rho\frac{\delta}{2}(1-\frac{\delta}{2})$, which implies
$$\widehat{\delta}_\square(G,H_n)\geq\widehat{\delta}_\square(H,H_n)-\widehat{\delta}_\square(G,H)\geq{\textstyle\rho\frac{\delta}{2}(1-\frac{\delta}{2})}-\epsilon\geq\frac{\epsilon}{8}\,,$$
but this contradicts the opposite inequality in the previous paragraph, so we must have $|C|\leq\delta n/2$. If $|A|<(\frac{1}{2}-\delta)n$ then since $|C|\leq\delta n/2$, we have $|B|\geq(\frac{1}{2}+\delta/2)n$. It follows that
$$\widehat{\delta}_\square(H,H_n)\geq\frac{\delta}{4}(1-\rho)>\sqrt{\epsilon}\,,$$
which contradicts the fact that
$$\widehat{\delta}_\square(H,H_n)\leq\widehat{\delta}_\square(H,G)+\widehat{\delta}_\square(G,H_n)\leq\epsilon+\frac{\epsilon}{8}\,,$$
proving $|A|,|B|\geq(\frac{1}{2}-\delta)n$. The same argument clearly shows $|A|,|B|\leq(\frac{1}{2}+\delta)n$.

The third assertion of the lemma follows from a case analysis that is proven in greater generality in \Cref{lemma:SubCloseStructure} \ref{item:lem72crossdensity}, so we refer the reader to that proof.

To prove \ref{item:lem61novtxhighdeg}, suppose a vertex $v\in A$ satisfies $d_{T(G)}(v,A)>(1-\alpha)|A|$. By optimality of $\Pi(G)$, we know $\overline{d}_G(v,B)\geq\overline{d}_G(v,A)$, hence
\begin{align*}
\overline{d}_G(v) &= \overline{d}_G(v,A)+\overline{d}_G(v,B)+\overline{d}_G(v,C) \geq 2(1-\alpha)\big(\textstyle{\frac{1}{2}}-\delta\big)n \geq (1-4\alpha)n
\end{align*}
and $d_G(v)\leq4\alpha n$. If we define $\Pi':=\{\{A\setminus\{v\},B\},C\cup\{v\}\}$, then since $(1-\alpha)|A|\geq n/4$,
\begin{align*}
b(G,\Pi') &\leq b(G,\Pi(G)) + d_G(v) - d_{T(G)}(v,A) \\
&\leq b(G,\Pi(G)) + 4\alpha n - (1-\alpha)|A| \leq b(G,\Pi) - ({\textstyle\frac{1}{4}}-4\alpha)n < b(G,\Pi(G)) \,,
\end{align*}
which contradicts optimality of $\Pi(G)$, completing the proof.
\end{proof}

\begin{lemma}\label{lemma:AlmostAllNoMedVtx}
For all $\Pi\in\sP$ we have
$$|\cC'_\Pi|\leq\binom{e(\Pi_\cb)}{m-e(\Pi_\cb^c)}e^{-\beta n^2}\,,$$
where $\beta=\beta(\epsilon)>0$ is a constant.
\end{lemma}
\begin{proof}
Denote $\Pi=\{\{A,B\},C\}$. For all $T\in\cT_\Pi$ and all $v\in V$, define the sets
\[\arraycolsep=1.4pt
\begin{array}{ll}
\cC_{\Pi,T}' &:= \{G\in\cC_\Pi':T(G)\cup G[C]=T\}\,, \\[5pt]
\cC_{\Pi,T,v}' &:= \{G\in\cC_{\Pi,T}':v\text{ has $(\Pi,T(G))$-medium degree}\}\,.
\end{array}
\]
Fix $T\in\cT_\Pi$ and $v\in V$ such that $\cC_{\Pi,T,v}'$ is nonempty. For convenience we may assume $v\in A\cup C$. By optimality of $\Pi(G)$, every graph $G\in\cC'_{\Pi,T,v}$ satisfies $\overline{d}_B(v)\geq\alpha|A|$ (since otherwise we could obtain a more optimal partition by moving $v$ to $B$). Let $N\subseteq A$ be a subset of size $\alpha|A|\leq|N|\leq(1-\alpha)|A|$, and let $N'\subseteq B$ be a subset of size at least $\alpha|A|$. The sets $N$ and $N'$ represent the vertices in $A$ and $B$, respectively, to which $v$ is non-adjacent. Let $A':=A\setminus\{v\}$ and let $H$ be the uniformly random subgraph of $K_{A',B}$ with exactly
$$m':=m-e(\Pi_\cb^c)+e(T[A\cup B])-e(T[V,C])$$
edges. Let $G$ be the random graph $H\cup(\Pi_\cb^c\setminus T)$. Now by applying \Cref{lemma:JansonPenaltyMedVtx} with the bipartition $\{A',B\}$, defect graph $T[A'\cup B]$, the sets $N$ and $N'$, and the random graph $H$, we find that
$$\bbP\{G\in\cC(n)\}\leq e^{-\frac{1}{4}\epsilon^{3/4}pn^2}$$
where $p:=m'/(|A'|\cdot|B|)$. Since every graph in $\cC_{\Pi,T,v}'$ is of the form $H'\cup(\Pi_\cb^c\setminus T)\cup T[V,C]$ for some bipartite graph $H'\subseteq K_{A',B}$, since there are at most $2^{2n}$ choices for the sets $N$ and $N'$, and since we assumed $v\in A\cup C$, we have
\begin{align}
|\cC_{\Pi,T,v}'| &\leq 2^{2n+1}\cdot\binom{e(\Pi_\cb)}{m'}\cdot\bbP\{G\in\cC(n)\} \nonumber \\
&\leq 2^{2n+1}\cdot\binom{e(\Pi_\cb)}{m'}\cdot e^{-\frac{1}{4}\epsilon^{3/4}pn^2} \nonumber \\
&\leq \binom{e(\Pi_\cb)}{m-e(\Pi_\cb^c)}\exp\left(\left|1-\frac{e(\Pi_\cb)-(m-e(\Pi_\cb^c))}{m-e(\Pi_\cb^c)}\right|\epsilon n^2\right)\cdot2^{2n+1}\cdot e^{-\frac{1}{4}\epsilon^{3/4}pn^2} \nonumber \\
&\leq \binom{e(\Pi_\cb)}{m-e(\Pi_\cb^c)}e^{\epsilon n^2}\cdot2^{2n+1}\cdot e^{-\frac{1}{4}\epsilon^{3/4}pn^2} \,,\label{eqn:CPrimePiTvUB}
\end{align}
where the third inequality is an application of \Cref{lemma:BinomialCoeffIneqs} using $e(T)\leq\epsilon n^2$, and the last inequality holds since we can apply \Cref{lemma:SuperCloseStructure} (\emph{ii}) to compute
$$0<\frac{\left(\frac{1}{2}-\delta\right)^2}{(\gamma+\epsilon)\binom{n}{2}-\left(\frac{1}{2}-\delta\right)^2}-1\leq\frac{e(\Pi_\cb)-(m-e(\Pi_\cb^c))}{m-e(\Pi_\cb^c)}\leq\frac{\left(\frac{1}{2}+\delta\right)^2}{(\gamma-\epsilon)\binom{n}{2}-\left(\frac{1}{2}+\delta\right)^2}-1<2\,,$$
which holds provided $n$ is large enough that $(\gamma-\epsilon)\binom{n}{2}\leq m\leq(\gamma+\epsilon)\binom{n}{2}$. Using \pref{eqn:CPrimePiTvUB}, we compute
\begin{align*}
|\cC'_\Pi| &\leq \sum_{t=1}^{\epsilon n^2}\,\sum_{\substack{T\in\cT_{\Pi}\\e(T)=t}}\,\sum_{v\in V}|\cC'_{\Pi,T,v}| \\
&\leq \epsilon n^2\cdot\binom{\binom{n}{2}}{\epsilon n^2}\cdot n\cdot\binom{e(\Pi_\cb)}{m-e(\Pi_\cb^c)}e^{\epsilon n^2}\cdot2^{2n+1}\cdot e^{-\frac{1}{4}\epsilon^{3/4}pn^2} \\
&\leq \binom{e(\Pi_\cb)}{m-e(\Pi_\cb^c)}e^{-\frac{1}{4}\epsilon^{3/4}pn^2}\cdot\epsilon n^3\cdot\left(\frac{e}{\epsilon}\right)^{\epsilon n^2}\cdot2^{2n+1} \,.
\end{align*}
Since $0<\epsilon<1/2^{64}$ and $p>1/4$, it is easy to see that $\frac{1}{4}\epsilon^{3/4}p>\epsilon\left(1+\log\left(\frac{1}{\epsilon}\right)\right)$, which implies the claimed inequality for a constant $\beta=\beta(\epsilon)>0$.
\end{proof}

Define the set of graphs
$$\cT'_\Pi:=\{T[G]\cup G[\Pi_\sparse]:G\in\cC_\Pi\setminus\cC'_\Pi\}\,.$$
For all $\Pi=\{\{A,B\},C\}\in\sP$ and $T\in\cT'_\Pi$, let $M(T)$ be a canonically chosen maximum matching of the graph $T[A\cup B]$. Let $X(T):=V(M(T))$ denote the set of vertices covered by the matching $M(T)$. For all $k,l\geq0$ define the sets of graphs
\[\arraycolsep=1.4pt
\begin{array}{ll}
\cT^\cb_{\Pi,k} &:= \left\{T[A\cup B]:T\in\cT'_\Pi\,,\,|X(T)|=k\right\} \,, \\[10pt]
\cT^\bt_{\Pi,l} &:= \left\{T[A\cup B,C]:T\in\cT'_\Pi\,,\,e(T[A\cup B,C])=l\right\} \,, \\[10pt]
\cT^\sparse_{\Pi} &:= \{T[\Pi_\sparse]:T\in\cT'_\Pi\} \,, \\[10pt]
\cT_{\Pi,k,l} &:= \left\{T\in\cT'_{\Pi}:T[A\cup B]\in\cT^\cb_{\Pi,k}\,,\,T[A\cup B,C]\in\cT^\bt_{\Pi,l}\right\} \,.
\end{array}\]

\begin{lemma}\label{lemma:CPITMatchingBound}
For all $\Pi\in\sP$, $k\geq2$, $l\geq1$, and $T\in\cT_{\Pi,k,l}$\,, we have
$$|\cC_{\Pi,T}|\leq\binom{e(\Pi_\cb)}{m-e(\Pi_\cb^c)-e(T[\Pi_\sparse])}e^{-\xi kn-\zeta n\cdot\min\{l,n\}}$$
where $\xi,\zeta>0$ are constants depending only on $\epsilon$.
\end{lemma}
\begin{proof}
Denote $\Pi=\{\{A,B\},C\}$. Let $H$ be the uniformly random subgraph of $\Pi_\cb$ with exactly
$$m':=m-e(\Pi_\cb^c)+e(T[A\cup B])-e(T[A\cup B,C])-e(T[C])$$
edges, and define $p:=m'/(|A|\cdot|B|)$. Let $G$ be the random graph $H\cup(\Pi_\cb^c\setminus T)\cup T[V,C]$. We now apply \Cref{lemma:JansonPenaltyMatching} with the following variable assignments
\[\arraycolsep=5mm
\begin{array}{lllll}
\Pi\gets\Pi_i & m\gets m' & D\gets \emptyset & D'\gets\emptyset & F'\gets\emptyset \\[10pt]
k\gets k & d\gets 0 & l\gets l & C\gets C & F\gets T[A\cup B,V\setminus(A\cup B)]\,.
\end{array}\]
The hypotheses of \Cref{lemma:JansonPenaltyMatching} are indeed met because $T\in\cT'_\Pi$ implies there is no medium degree vertex, and \Cref{lemma:SuperCloseStructure} \ref{item:lem61novtxhighdeg} implies every vertex $v\in P\in\Pi$ satisfies $d_T(v,P)<\alpha|P|$. Hence
$$\bbP\{G\in\cC(n)\}\leq p^{kn/2^{11}}\cdot\exp\left(-\frac{p^4}{2^{25}}\min\left\{ln,n^2\right\}\right)$$
(notice we have divided by two in the exponents of the right-hand side to account for the fact that $|A\cup B|\geq(1-\delta)n$). Since every graph in $\cC_{\Pi,T}$ is of the form $H'\cup(\Pi_\cb^c\setminus T)\cup T[V,C]$ for some bipartite graph $H'\subseteq\Pi_\cb$, we have
\begin{align*}
|\cC_{\Pi,T}| &\leq \binom{e(\Pi_\cb)}{m'}\cdot\bbP\{G\in\cC(n)\} \\
&\leq \binom{e(\Pi_\cb)}{m-e(\Pi_\cb^c)-e(T[C])} \\
&\hspace{2cm}\cdot\exp\left(-\left(1-\frac{e(\Pi_\cb)-(m-e(\Pi^c_\cb))}{m-e(\Pi^c_\cb)}\right)(2\alpha kn+l)\right)\cdot\bbP\{G\in\cC(n)\} \\
&\leq \binom{e(\Pi_\cb)}{m-e(\Pi_\cb^c)-e(T[C])}e^{2\alpha kn+l}\cdot p^{kn/2^{11}}\cdot\exp\left(-\frac{p^4}{2^{25}}\min\left\{ln,n^2\right\}\right) \,,
\end{align*}
for all sufficiently large $n$, where the second inequality follows by using \Cref{lemma:BinomialCoeffIneqs} and noticing that $e(T[A\cup B])\leq2\alpha kn$ (since by definition of $\cC_{\Pi,T}$, the set $A\cup B$ only has vertices that have low degree in their respective parts, and every edge of $T[A\cup B]$ is incident to $X(T)$ since the matching $M(T)$ is maximum). The third inequality holds for the same reasons \pref{eqn:CPrimePiTvUB} was shown to hold. Since $p\geq1/4$ is of constant order, from the last inequality above, we see that the claimed inequality holds for small enough $\epsilon>0$ and constants $\xi,\zeta>0$.
\end{proof}

For all $s\geq1$, let $\sP_s$ denote the set of all $\Pi\in\sP$ such that $|\Pi_\sparse|=s$. For all $\Pi\in\sP_s$, let $\Pi_\ast$ denote either $\{A\cup\Pi_\sparse,B\}$ or $\{A,B\cup\Pi_\sparse\}$, canonically chosen.

\begin{lemma}\label{lemma:CPiStarBound}
If $\gamma\in\big(\frac{5-\sqrt{5}}{4},1\big)$ then for all $s\geq1$ and all $\Pi\in\sP_s$ we have
$$|\cC_\Pi^\ast|\leq\binom{e(\Pi_\ast)}{m-e(\Pi_\ast^c)}e^{-\nu sn}\,,$$
where $\nu>0$ is a constant depending only on $\gamma.$
\end{lemma}
\begin{proof}
Denote $\Pi=\{\{A,B\},C\}$, and we may assume $\Pi_\ast=\{A\cup C,B\}$. For all $T\in\cT^\ast_\Pi$ let $\cC^\ast_{\Pi,T}$ denote the set of all $G\in\cC^\ast_\Pi$ such that $T(G)=T$. Define the quantities $a:=|A|$, $b:=|B|$, $N:=e(\Pi_\ast)$, $M:=m-e(\Pi_\ast^c)$, and $t:=e(T)$. Since $\gamma>\frac{5-\sqrt{5}}{4}$ we have $(1-\rho)^2/\rho<1$, so by \Cref{lemma:SuperCloseStructure} \ref{item:lem61densityclose}, there is a constant $\nu'>0$ depending only on $\gamma$ such that $(N-M)^2/(NM)<e^{-\nu'}$ for small enough $\epsilon$ and large enough $n$. Similarly, we have $(N-M)/M<e$ for large enough $n$. Hence for all $T\in\cT^\ast_\Pi$, we use \Cref{lemma:BinomialCoeffIneqs} to bound
\begin{align*}
|\cC^\ast_{\Pi,T}| &\leq \binom{e(\Pi_\cb)}{m-e(\Pi_\cb^c)-e(T)} \\
&= \binom{e(\Pi_\ast)-sb}{m-e(\Pi_\ast^c)+sa+\binom{s}{2}-e(T)} \\
&\leq \binom{e(\Pi_\ast)}{m-e(\Pi_\ast^c)}\left(\frac{N-M}{N}\right)^{sb}\left(\frac{N-M}{M}\right)^{sa+\binom{s}{2}-e(T)} \\
&= \binom{e(\Pi_\ast)}{m-e(\Pi_\ast^c)}\left(\frac{(N-M)^2}{NM}\right)^{sb}\left(\frac{N-M}{M}\right)^{s(a-b)+\binom{s}{2}-e(T)} \\
&\leq \binom{e(\Pi_\ast)}{m-e(\Pi_\ast^c)}e^{-\nu'sb+s(a-b)+s^2} \leq \binom{e(\Pi_\ast)}{m-e(\Pi_\ast^c)}e^{-\nu''sn} \,,
\end{align*}
which holds for $\nu'':=\nu'/8$ and sufficiently large $n$; we have also used the fact that $\nu'b>|a-b|$ for small enough $\epsilon>0$ since $\Pi\in\sP$ and $|a-b|\leq2\delta$. Since $|\cT^\ast_\Pi|\leq2^{\binom{s}{2}}$,
$$|\cC^\ast_\Pi|=\sum_{T\in\cT^\ast_\Pi}|\cC^\ast_{\Pi,T}|\leq2^{\binom{s}{2}}\binom{e(\Pi_\ast)}{m-e(\Pi_\ast^c)}e^{-\nu''sn}\leq\binom{e(\Pi_\ast)}{m-e(\Pi_\ast^c)}e^{-\nu sn}\,,$$
where $\nu:=\nu''/2$, completing the proof.
\end{proof}

\begin{proof}[Proof of \Cref{thm:almostallCB} \rm{\ref{thm:almostallCBsupercritical}}]
First, for all $\Pi\in\sP$, \Cref{lemma:AlmostAllNoMedVtx} shows
\begin{equation}
|\cC'_\Pi|\leq\binom{e(\Pi_\cb)}{m-e(\Pi_\cb^c)}e^{-\beta n^2}\leq e^{-\beta n^2/2}|\cC^\ast_\Pi| \,, \label{eqn:CPrimePiStarBound}
\end{equation}
where the second inequality is immediate from \Cref{cor:cbuniquepartition}, which shows almost all co-bipartite graphs with edge density $\gamma\in(\frac{1}{2},1)$ admit a unique 2-clique-cover.

We now obtain a bound on $\sum_{T\in\cT_\Pi}|\cC_{\Pi,T}|$ for $\Pi\in\sP$ and $k\geq2$. First notice that for all $T\in\cT^\cb_{\Pi,k}$, every vertex $v\in X(T)$ satisfies $d_T(v,P)<\alpha|P|$ for $P\in\{A,B\}$, which follows from \Cref{lemma:SuperCloseStructure} \ref{item:lem61novtxhighdeg} and the definition of $\cT^\cb_{\Pi,k}$. Since every edge in $T[A\cup B]$ is incident to $X(T)$, we have
\begin{align*}
|\cT^\cb_{\Pi,k}| &\leq \sum_{X\in\binom{V}{k}}\,\prod_{P\in\Pi}\,\prod_{v\in P\cap X}\binom{|P|}{\leq\alpha|P|} \leq \sum_{X\in\binom{V}{k}}\,\prod_{P\in\Pi}\,\prod_{v\in P\cap X}\alpha n\binom{n}{\alpha n} \\
&\leq \binom{n}{k}\left(\alpha n\binom{n}{\alpha n}\right)^k \leq 2^{H(\alpha)kn+2k\log_2n+k\log_2\alpha} \leq 2^{2H(\alpha)kn} \,,
\end{align*}
where we used that $\binom{a}{b}\leq\binom{xa}{xb}$ for $x\geq1$. We also trivially have
$$|\cT^\bt_{\Pi,l}|\leq\binom{\binom{n}{2}}{l}\leq\left(\frac{en^2}{2l}\right)^l\,.$$
Let $\xi,\zeta>0$ be the constants from \Cref{lemma:CPITMatchingBound}, and let $\epsilon>0$ sufficiently small such that $H(\alpha)<\xi/6$. Then for sufficiently large $n$,
\begin{align}
\sum_{T\in\cT_\Pi}|\cC_{\Pi,T}| &= \sum_{k=2}^{n}~\sum_{l=1}^{\epsilon n^2}~\sum_{T\in\cT_{\Pi,k,l}}|\cC_{\Pi,T}| \nonumber\\
&\leq \sum_{k=2}^{n}~\sum_{l=1}^{\epsilon n^2}~\sum_{T\in\cT_{\Pi,k,l}}\binom{e(\Pi_\cb)}{m-e(\Pi_\cb^c)-e(T[\Pi_\sparse])}e^{-\xi kn-\zeta n\cdot\min\{l,n\}} \nonumber\\
&\leq \sum_{k=2}^{n}~\sum_{l=1}^{\epsilon n^2}|\cT^\cb_{\Pi,k}|\cdot|\cT^\bt_{\Pi,l}|\cdot e^{-\xi kn-\zeta n\cdot\min\{l,n\}}\sum_{T\in\cT^\sparse_\Pi}\binom{e(\Pi_\cb)}{m-e(\Pi_\cb^c)-e(T[\Pi_\sparse])} \nonumber\\
&\leq 2|\cC^\ast_\Pi|\cdot e^{-\xi n}\left(\sum_{k=2}^\infty2^{2H(\alpha)kn}e^{-\xi kn/2}\right)\left(\sum_{l=1}^\infty\exp\left(l\log\left(\frac{en^2}{2l}\right)-\zeta n\cdot\min\{l,\,n\}\right)\right) \nonumber\\[5pt]
&\leq e^{-\xi n/2}|\cC^\ast_\Pi| \,, \label{eqn:CPiTSumUP}
\end{align}
where we again used \Cref{cor:cbuniquepartition} to deduce the third inequality, and the last inequality holds since $H(\alpha)\to0$ as $\epsilon\to0$ and the two rightmost factors on the left-hand side of \pref{eqn:CPiTSumUP} are geometric series.

Let $\nu>0$ be the constant from \Cref{lemma:CPiStarBound}. Since for all $\Pi'\in\sB$, there are at most $\binom{n}{s}2^s$ divisions $\Pi\in\sP_s$ such that $\Pi_\ast=\Pi'$, we use \Cref{lemma:CPiStarBound} to deduce
\begin{align}
\sum_{s=1}^{\delta n}\sum_{\Pi\in\sP_s}|\cC^\ast_\Pi| &\leq \sum_{s=1}^{\delta n}\binom{e(\Pi_\ast)}{m-e(\Pi_\ast^c)}e^{-\nu sn} \nonumber\\
&\leq \sum_{\Pi\in\sB}\binom{e(\Pi)}{m-e(\Pi^c)}\sum_{s=1}^{\delta n}\binom{n}{s}2^se^{-\nu sn} \nonumber\\
&\leq e^{-\nu n/2}\left(\sum_{s=1}^\infty e^{-\nu sn/2+s\log n+(\log2)s}\right)\cdot2|\cB_c(n,m)| \leq e^{-\nu n/4}|\cB_c(n,m)|\,, \label{eqn:sumCStarPiBound}
\end{align}
where the third inequality follows from \Cref{cor:cbuniquepartition}, and the last inequality follows by evaluating a geometric series.

Using \pref{eqn:CPrimePiStarBound}, \pref{eqn:CPiTSumUP}, and \pref{eqn:sumCStarPiBound}, we have
\begin{equation}
\begin{split}
\sum_{s=1}^{\delta n}\,\sum_{\Pi\in\sP_s}\Bigg(|\cC^\ast_\Pi|+\sum_{T\in\cT_\Pi}|\cC_{\Pi,T}|+|\cC'_\Pi|\Bigg) &\leq \sum_{s=1}^{\delta n}\sum_{\Pi\in\sP_s}\big(1+e^{-\xi n/2}+e^{-\beta n^2/2}\big)|\cC^\ast_\Pi| \\
&\leq 2e^{-\nu n/4}|\cB_c(n,m)| \,,
\end{split}\label{eqn:sumsPicase1}
\end{equation}
and similarly
\begin{equation}
\begin{split}
\sum_{\Pi\in\sP_0}\Bigg(\sum_{T\in\cT_\Pi}|\cC_{\Pi,T}|+|\cC'_\Pi|\Bigg)\leq\sum_{\Pi\in\sP_0}(e^{-\xi n/2}+e^{-\beta n^2/2})|\cC^\ast_\Pi|\leq2e^{-\xi n/2}|\cB_c(n,m)| \,.
\end{split}\label{eqn:sumsPicase2}
\end{equation}

Let $C>0$ be the constant from \Cref{prop:CNMtypstrucconddistr}. By combining \pref{eqn:SupercritKRDecomp}, \pref{eqn:sumsPicase1}, and \pref{eqn:sumsPicase2}, we obtain
\begin{align*}
|\cC(n,m)| &\leq |\cB_c(n,m)|+\sum_{s=1}^{\delta n}\,\sum_{\Pi\in\sP_s}|\cC^\ast_\Pi|+\sum_{\Pi\in\sP}\Bigg(\sum_{T\in\cT_\Pi}|\cC_{\Pi,T}|+|\cC'_\Pi|\Bigg)+|\cC_\far| \\[10pt]
&\leq (1+2e^{-\nu n/4}+2e^{-\xi n/2})|\cB_c(n,m)|+e^{-Cn^2}|\cC(n,m)| \,,
\end{align*}
completing the proof.
\end{proof}

\begin{proof}[Proof of \Cref{thm:almostallCB} \rm{\ref{thm:almostallCBcritical}}]
Let $\cC^\ast:=\bigcup_{\Pi\in\sP}\cC^\ast_\Pi$. For a fixed set $S\subseteq V$ of size $|S|\leq\delta n/2$, \Cref{lemma:almostallcovnmbalanced} proves that almost all $G\in\cC^\ast$ with $\Pi_\sparse=S$ have the property that $\Pi_\cb(G)=\{A,B\}$ is almost equitable (that is, $||A|-|B||\leq\sqrt{\log n}$). Hence in the remainder, we assume $\Pi\in\sP$ has the property that $\Pi_\cb$ is almost equitable.

Let $\sP^{(1)}$ denote the set of all $\Pi\in\sP$ such that $|\Pi_\sparse|\leq\log n$, and let $\sP^{(2)}:=\sP\setminus\sP^{(1)}$. Let $s\geq1$, $\Pi\in\sP_s$, and $T\in\cT_\Pi$. Denote $\Pi_\cb=\{A,B\}$ and let $a:=|A|$ and $b:=|B|$. Let $s_a,s_b\geq0$ be integers such that $s_a+s_b=s$ and $|a+s_a-b-s_b|$ is minimum. (Notice $s_a$ and $s_b$ are uniquely determined by $\Pi$.) Fix a bipartition $\Pi_\ast=\{A_\ast,B_\ast\}$ of $V$ such that $A\subseteq A_\ast$, $B\subseteq B_\ast$, $|A_\ast|=a+s_a$, and $|B_\ast|=b+s_b$. Define the quantities
\[
\arraycolsep=5mm
\begin{array}{ll}
N:=(a+s_a)(b+s_b)\,, & \displaystyle M:=m-\binom{a+s_a}{2}-\binom{b+s_b}{2}\,, \\[15pt]
K:=s_ba+s_ab\,, & \displaystyle L:=s_aa+s_bb+\binom{s_a}{2}+\binom{s_b}{2}-e(T)\,.
\end{array}\]
Since $\Pi\in\sP_s$, \Cref{lemma:SuperCloseStructure} proves $|a-n/2|\leq\delta n$ and $|b-n/2|\leq\delta n$. Using \Cref{lemma:universalvtxlittleoh1} and a second-order series expansion, we compute
\begin{align}
|\cC^\ast_{\Pi,T}| &\sim \binom{ab}{m-\binom{a}{2}-\binom{b}{2}-e(T)} = \binom{N-K}{M+L} \nonumber\\[5pt]
&= \binom{N}{M}\frac{\prod_{i=1}^{K+L}(N-M-i+1)}{\big(\prod_{i=1}^L(M+i)\big)\big(\prod_{i=1}^K(N-i+1)\big)} \nonumber\\[5pt]
&= \binom{N}{M}\frac{\prod_{i=1}^{K+L}(1-\frac{i-1}{N-M})}{\big(\prod_{i=1}^K(\frac{N}{N-M}-\frac{i-1}{N-M})\big)\big(\prod_{i=1}^L(\frac{M}{N-M}+\frac{i}{N-M})\big)} \nonumber\\[5pt]
&= \binom{N}{M}\left(\frac{N-M}{N}\right)^K\left(\frac{N-M}{M}\right)^L \nonumber\\
&\hspace{4cm}\cdot\exp\left(-\frac{(K+L)^2}{2(N-M)}+\frac{K^2}{2N}-\frac{L^2}{2M}+O\left(\frac{(K+L)^3}{(N-M)^2}\right)\right) \nonumber\\[5pt]
\begin{split}
&= \binom{N}{M}\left(\frac{(N-M)^2}{NM}\right)^L\left(\frac{N-M}{N}\right)^{K-L} \\
&\hspace{4cm}\cdot\exp\left(-\frac{(K+L)^2}{2(N-M)}+\frac{K^2}{2N}-\frac{L^2}{2M}+O\left(\frac{1}{n}\right)\right)\,.
\end{split}\label{eqn:CstarPiTAsymptotics}
\end{align}
Notice that for sufficiently large $n$, we have $L\geq sn/2$, $M\leq\frac{1}{2}(\gamma-1/2)n^2$, and $a+b\geq(1-4\delta)n$. Hence we compute that
\begin{align*}
\frac{(K+L)^2}{2(N-M)}-\frac{K^2}{2N}+\frac{L^2}{2M} &\geq L\left(\frac{K+L}{N-M}+\frac{L}{2M}\right) \\
&\geq \frac{sn}{2}\left(\frac{s(a+b)+\binom{s_a}{2}+\binom{s_b}{2}-e(T)}{\binom{n}{2}-m}+\frac{s}{(2\gamma-1)n}\right) \\
&\geq \frac{(1-4\delta)s^2}{1-\gamma}+\frac{s^2}{4\gamma-2} \geq 6s^2 \,.
\end{align*}

We now claim $((N-M)^2/(NM))^L|=O(n)$. Indeed, since $|a-b|\leq\sqrt{\log n}$, we have
$$\left|(1-\rho)-\frac{N-M}{N}\right|=O\left(\frac{\log n}{n}\right)\hspace{6mm}\text{and}\hspace{6mm}\left|\frac{1-\rho}{\rho}-\frac{N-M}{M}\right|=O\left(\frac{\log n}{n}\right)\,,$$
which implies (using that $(1-\rho)^2/\rho=1$)
$$\left|1-\frac{(N-M)^2}{NM}\right|=O\left(\frac{\log n}{n}\right)\,.$$
The fact that $L=O(n)$ now completes the proof of our claim.

Now if $\Pi$ satisfies $|\Pi_\sparse|=s\geq\log n$ then from \pref{eqn:CstarPiTAsymptotics},
\begin{align*}
|\cC^\ast_{\Pi,T}| &\leq \binom{N}{M}\left(\frac{(N-M)^2}{NM}\right)^L\left(\frac{N-M}{N}\right)^{(s_b-s_a)(a-b)-\binom{s_a}{2}-\binom{s_b}{2}+e(T)}e^{-5s^2} \\
&\leq \binom{N}{M}\left(\frac{N-M}{N}\right)^{-\binom{s_a}{2}-\binom{s_b}{2}}e^{-4s^2} \leq \binom{N}{M}e^{-2s^2}\,,
\end{align*}
where we used that $(s_b-s_a)(a-b)\geq0$ and $1<N/(N-M)<e$. Since for all $s\geq0$ and $\Pi'\in\sP$, there are at most $\binom{n}{s}2^s$ divisions $\Pi\in\sP_s$ such that $\Pi_\ast=\Pi'$, we have 
\begin{align*}
\sum_{\Pi\in\sP^{(2)}}\,\sum_{T\in\cT^\ast_\Pi}|\cC^\ast_{\Pi,T}| & \leq\sum_{s\geq\log n}\,\sum_{\Pi\in\sP_s}\binom{N}{M}e^{-\frac{3}{2}s^2} \\
&\leq \sum_{s\geq\log n}\binom{n}{s}2^se^{-\frac{3}{2}s^2}\sum_{\Pi\in\sP}\binom{e(\Pi)}{m-e(\Pi^c)} \\
&\leq 2\cdot|\cB_c(n,m)|\sum_{s\geq\log n}e^{-\frac{3}{2}s^2+s\log n+(\log 2)n} \leq \frac{n^{-\log n}}{\log n}\cdot|\cB(n,m)|\,,
\end{align*}
where we used that $|\cT^\ast_\Pi|\leq e^{s^2/2}$ for $\Pi\in\sP_s$, and we used a standard Gaussian tail bound to obtain the last inequality. Combining the previous bound with the rest of the proof of \Cref{thm:almostallCB} immediately implies the result.
\end{proof}

\section{The Subcritical Regime}
\label{sec:subcritical}
In this section we prove the third assertion of \Cref{thm:almostallCB}. Throughout the section, let $\gamma\in\big(0,\frac{5-\sqrt{5}}{4}\big)$ and $m\sim\gamma\binom{n}{2}$. There is a global parameter $\epsilon>0$ used throughout the section; in general, all statements hold for sufficiently small $\epsilon$, but in some cases it will be useful to have established that  $0<\epsilon<\min\{\gamma,\frac{5-\sqrt{5}}{4}-\gamma\}/2^{65}$. Likewise, inequalities $f(n)\leq g(n)$ are always meant for large enough $n$.

\subsection{Parameters and Definitions}
Define the parameters and constants
\[\arraycolsep=5mm
\begin{array}{llll}
\displaystyle\tau:=\frac{1}{2}\left(\frac{\epsilon^2}{2^8}\right)^{67}\,, & \delta:=8\sqrt{\epsilon}\,, & \alpha:=8\epsilon^{1/8}\,, & \eta:=\epsilon^{1/16}\,, \\ [10pt]
\omega:=\epsilon^{1/32}\,, & \displaystyle\mu:=\bigg(\frac{5+\sqrt{5}}{20}\gamma\bigg)^{1/2}\,, & \displaystyle\rho:=\frac{3-\sqrt{5}}{2} \,.
\end{array}\]
Let $\blambda_\ast:=(0,2\mu)$ and $W^\ast:=W_{\blambda_\ast}$ (using the definition in \pref{eqn:flambdaGraphonDef}), so $W^\ast$ is the unique optimal graphon $W_{\blambda}\in\cX^\ast_\gamma$ (up to equivalence) of the variational problem \pref{eqn:variationalproblem} satisfying $|\blambda|=2$.

Let $A_1,B_1,\dots,A_\ell,B_\ell\subseteq V$ be disjoint nonempty vertex subsets, let $\Pi_i:=\{A_i,B_i\}$ for all $i\in[\ell]$, and let $\Pi:=\{\Pi_1,\dots,\Pi_\ell\}$; we call such a set $\Pi$ an \emph{division} of $V$. Let $\sD$ denote the set of all divisions of $V$. We will always assume the $\Pi_i$ are ordered such that the numbers $|A_i|+|B_i|$ are nonincreasing in $i$. When there is no ambiguity, we use the notation $\Pi_i$ to refer to the complete bipartite graph $K_{\Pi_i}$, and $\Pi$ to refer to the vertex-disjoint union of bipartite graphs $\Pi_i$. For all divisions $\Pi\in\sD$, let $V(\Pi):=\bigcup\bigcup\Pi\subseteq V$ be the vertex set covered by $\Pi$, let $v(\Pi):=|V(\Pi)|$, and let $\Pi_\sparse:=V\setminus V(\Pi)$.

Let $\Pi=\{\Pi_1,\dots,\Pi_\ell\}\in\sD$ be a division and denote $\Pi_i=\{A_i,B_i\}$ for all $i\in[\ell]$. Also let $v_i:=|A_i\cup B_i|$ for all $i\in[\ell]$. Let $\ell'\in[\ell]$ be the greatest index such that $v_{\ell'}\geq8\delta n$. Let $\blambda=(\lambda_0,\lambda_1,\dots)\in\Lambda$ (see \Cref{sec:GraphonOpt} for the definition of $\Lambda$) and let $W:=W_{\blambda}$. Denote $\mu_i:=\lambda_i-\lambda_{i-1}$ for all $i\geq1$. Let $L$ be the set $\N$ if $|\blambda|=\infty$ and $\{1,\dots,|\blambda|\}$ otherwise. Define the set of divisions
$$\sD_W:=\left\{\Pi\in\sD:\splitfrac{\vphantom{\displaystyle\bigcup}\text{there exists an injection $\varphi:[\ell']\to L$ such that}}{\text{$|v_i-\mu_{\varphi(i)}n|\leq\delta n$ for all $i\in[\ell']$, and $\mu_i\leq\delta$ for all $i\in L\setminus\varphi([\ell'])$}}\right\}\,.$$
Since for every graphon $U\in\cX^\ast_\gamma$ there is a unique sequence $\blambda\in\Lambda$ such that $U$ is equivalent to $W_{\blambda}$\,, the notation $\sD_U:=\sD_{W_{\blambda}}$ is well-defined.

For all graphs $G$ and divisions $\Pi=\{\Pi_1,\dots,\Pi_\ell\}$ with $\Pi_i=\{A_i,B_i\}$, let
\begin{align*}
b(G,\Pi) &:= \sum_{i=1}^\ell\sum_{P\in\Pi_i}e(G^c[P])+\sum_{1\leq i<j\leq\ell}e(G[A_i\cup B_i,A_j\cup B_j])+e(G[V,\Pi_\sparse]) \,.
\end{align*}
Let $\Pi(G)\in\sD$ denote a canonically chosen division minimizing $b(G,\Pi)$, and define $b(G):=b(G,\Pi(G))$. For all $G\in\cC(n,m)$ with $\Pi(G)=\{\Pi_1,\dots,\Pi_\ell\}$, $\Pi_i=\{A_i,B_i\}$, define
$$D(G):=\bigcup_{i=1}^\ell\bigcup_{P\in\Pi_i}E(G^c[P])\cup\bigcup_{1\leq i<j\leq\ell}E(G[A_i\cup B_i,A_j\cup B_j])\cup E(G[V(\Pi),\Pi_\sparse])\,.$$
Let $U(G):=V(D(G))$ denote the set of vertices incident to $D(G)$ and define the graph $T(G):=(U(G),D(G))$. For all $W\in\cX^\ast_\gamma$ and $\Pi\in\sD_W$, define
\[\arraycolsep=1.4pt
\begin{array}{ll}
\cC_{W,\Pi} &:= \left\{G\in\cC(n,m):\splitfrac{\text{$\delta_\square(G,W)<\tau$, $\Pi(G)=\Pi$, and $\exists~\Pi'\in\Pi$}\vphantom{\displaystyle\bigcup}}{\text{such that $|V(\Pi')|\geq\eta n$ and $e_{T(G)}(V,V(\Pi'))\geq1$}}\right\} \,, \\[20pt]
\cC^\ast_{W,\Pi} &:= \left\{G\in\cC(n,m):\splitfrac{\text{$\delta_\square(G,W)<\tau$, $\Pi(G)=\Pi$, and $\forall~\Pi'\in\Pi$,}\vphantom{\displaystyle\bigcup}}{\text{$|V(\Pi')|\geq\eta n$ implies $e_{T(G)}(V,V(\Pi'))=0$}}\right\} \,,
\end{array}\]
and define the associated sets of graphs
\[\arraycolsep=1.4pt
\begin{array}{ll}
\cT_{W,\Pi} &:= \{T(G)\cup G[\Pi_\sparse]:G\in\cC_{W,\Pi}\} \,, \\[5pt]
\cT^\ast_{W,\Pi} &:= \{G[\Pi_\sparse]:G\in\cC^\ast_{W,\Pi}\} \,.
\end{array}\]
For all $T\in\cT_{W,\Pi}$ and $P\in\Pi'\in\Pi$, where $|V(\Pi')|\geq\eta n$, say that a vertex $v\in V$ has
\begin{enumerate}[(\emph{\roman*})]
	\item \emph{low degree in $P$}  if $d_T(v,P)<\alpha|P|$,
	\item \emph{medium degree in $P$} if $\alpha|P|\leq d_T(v,P)\leq(1-\alpha)|P|$, or 
	\item \emph{high degree in $P$} if $d_T(v,P)>(1-\alpha)|P|$.
\end{enumerate}
If a vertex $v$ has medium degree in some $P\in\Pi'\in\Pi$ with respect to $T\in\cT_{W,\Pi}$, we say that $v$ has $(P,T)$-medium degree. Define the sets
\[\arraycolsep=1.4pt
\begin{array}{ll}
\cC'_{W,\Pi} &:= \left\{G\in\cC_{W,\Pi}:\splitfrac{\text{for some $\Pi'=\{A,B\}\in\Pi$ and $v\in V\setminus V(\Pi')$,}\vphantom{\displaystyle\bigcup}}{\text{$|V(\Pi')|\geq\eta n$, $d_{T(G)}(v,A)>(1-\alpha)|A|$, and $d_{T(G)}(v,B)\geq\alpha|B|$}}\right\} \,, \\[20pt]
\cC''_{W,\Pi} &:= \left\{G\in\cC_{W,\Pi}\setminus\cC'_{W,\Pi}:\splitfrac{\text{for some $\Pi'\in\Pi$, $P\in\Pi'$, and $v\in V$,}\vphantom{\displaystyle\bigcup}}{\text{$|V(\Pi')|\geq\eta n$ and $v$ has $(P,T(G))$-medium degree}}\right\} \,,
\end{array}\]
and for all $T\in\cT_{W,\Pi}$ define the set
$$\cC_{W,\Pi,T} := \{G\in\cC_{W,\Pi}\setminus(\cC'_{W,\Pi}\cup\cC''_{W,\Pi}):T(G)\cup G[\Pi_\sparse]=T\}\,.$$
For all $n,m\in\N$, $\tau>0$, and graphons $W\in\cW$, define the neighborhood
$$N_\tau(n,m,W):=\{G\in\cC(n,m):\delta_\square(G,W)<\tau\}\,.$$
In the following lemma, recall that the sets $\Lambda$ and $\cV$ were defined in \Cref{sec:GraphonOpt}, and recall the discretization $W_n$ of a graphon $W$ from \Cref{def:DiscreteGraphon}.

\begin{lemma}\label{lemma:WtoWtildeMetrics}
If $W\in\cV$ and $G\in N_\tau(n,m,W)$ then $\widehat{\delta}_\square(G,W_n)<\epsilon^2/8$ for large enough $n$. Additionally, by editing (i.e. adding or removing) at most $\epsilon n^2$ edges of $G$, we can obtain a graph that is the vertex-disjoint union of co-bipartite graphs and an independent set.
\end{lemma}
\begin{proof}
We may assume $W=W_{\blambda}$, where $\blambda=(\lambda_0,\lambda_1,\dots)\in\Lambda$. Let $\mu_i:=\lambda_i-\lambda_{i-1}$ for $i\geq1$. For all $n\in\N$, let $H_n$ be the weighted graph on $V$ whose $ij$ edge weight is $W(\frac{i}{n},\frac{j}{n})$ with unweighted nodes. In the remainder, all statements hold for large enough $n$. The sequence $W_n=W_{H_n}$ clearly converges to $W$ pointwise almost everywhere and hence in cut metric, so the condition $\delta_\square(G,W)<\tau$ implies $\delta_\square(G,W_n)<2\tau$. By \cite[Lemma~8.9]{lovasz2012large} we have $\delta_\square(G,H_n)=\delta_\square(W_G,W_n)$, and by \cite[Theorem~2.3]{borgs2008convergent} we have
$$\widehat{\delta}(G_1,G_2)\leq32\cdot(\delta_\square(G_1,G_2))^{1/67}$$
for edge-weighted graphs with the same number of unweighted nodes; from these two results we deduce that $\widehat{\delta}_\square(G,H_n)<\epsilon^2/8$.

Let $0\leq k\leq1/\epsilon$ be the greatest index such that $\mu_k\geq\epsilon$. Let $G_0$ be an isomorphic copy of $G$ such that $d_\square(G_0,H_n)=\widehat{\delta}_\square(G,H_n)$, and let $\phi:V(G)\to V(G_0)$ be the isomorphism. For all $i\in[k]$, define the sets
\[\arraycolsep=1.4pt
\begin{array}{ll}
C_{i,1} &:= \displaystyle\left\{j\in[n]:j/n\in\left(\lambda_{i-1}\,,\,\frac{\lambda_{i-1}+\lambda_i}{2}\right)\right\} \,, \\[12pt]
C_{i,2} &:= \displaystyle\left\{j\in[n]:j/n\in\left(\frac{\lambda_{i-1}+\lambda_i}{2}\,,\,\lambda_i\right)\right\} \,, \\[12pt]
C_i &:= C_{i,1}\cup C_{i,2} \,,
\end{array}\]
and let $J:=V\setminus\bigcup_{i=1}^kC_i$. Since $d_\square(G_0,H_n)<\epsilon^2/8$, we have $e(G_0^c[C_{i,j}])\leq\epsilon^2n^2/8$ for all $i\in[k]$ and $j\in[2]$. Similarly, for all distinct $i,j\in[k]$, we have $e(G_0[C_i,C_j])\leq\epsilon^2n^2/8$. Notice also that there are at most $(k+1)n$ edges in $G$ incident to vertices $i\in[n]$ that satisfy $i/n=\lambda_j$ for some $j\in[k]$ ; let $U$ denote the set of these vertices.

Let $\blambda':=(\lambda_0,\dots,\lambda_k)$ and let $H'_n$ be the weighted graph on $V$ whose $ij$ edge weight is $W_{\blambda'}(\frac{i}{n},\frac{j}{n})$ with unweighted nodes. Since the sequence $(W_{\blambda'})_n=W_{H'_n}$ of graphons converges to $W_{\blambda'}$ pointwise almost everywhere, we have
\begin{align*}
e(G_0[J]) &\leq d_\square(G_0,H'_n) \\
&\leq d_\square(G_0,H_n)+d_\square(H_n,H'_n) \\
&\leq \frac{\epsilon^2}{8}+2\int_{[\lambda_k,1]^2}W(x,y)\,dx\,dy \\
&= \frac{\epsilon^2}{8}+2\sum_{k\leq i<|\blambda|}\frac{\mu_i^2}{2}\Bigg(1+\frac{3-\sqrt{5}}{2}\Bigg)\leq \frac{\epsilon}{4}\,,
\end{align*}
where the last inequality holds since $\mu_i<\epsilon$ for all $i>k$. From the above considerations, it follows immediately that by adding at most $\epsilon n^2/4$ edges to $G_0$ in the vertex sets $C_{i,j}$, removing at most $\epsilon n^2/8$ edges from $G_0$ in the edge sets $E(G_0[C_i,C_j])$, removing the edges incident to $U$, and removing at most $\epsilon n^2/4$ edges from $G_0$ in the vertex set $J$, we can obtain a graph that is the vertex-disjoint union of co-bipartite graphs (on the vertex sets $C_i$) and an independent set (on the vertex set $J\cup U$). Since we made at most $\epsilon n^2$ edits (provided $n$ is large enough that $(k+1)n\leq\epsilon n^2/4$), this completes the proof.
\end{proof}

\begin{lemma}\label{lemma:SubCloseStructure}
Let $W\in\cX^\ast_\gamma$. For sufficiently large $n$, all graphs $G\in N_\tau(n,m,W)$ satisfy the following conditions:
\begin{enumerate}[(\emph{\roman*})]
	\item\label{item:lem72PiGBalanced} $\Pi(G)\in\sD_W$\,.
	\item\label{item:lem72Piprimebalanced} For all $\Pi'=\{A,B\}\in\Pi(G)$ such that $|V(\Pi')|\geq\alpha n$, it holds that
	$$\left(\frac{1}{2}-\alpha\right)|V(\Pi')|\leq\min\{|A|,\,|B|\}\leq\max\{|A|,\,|B|\}\leq\left(\frac{1}{2}+\alpha\right)|V(\Pi')|\,.$$
	\item\label{item:lem72crossdensity} For all $\Pi'=\{A,B\}\in\Pi(G)$ such that $|V(\Pi')|\geq\delta n$, we have $|\rho-d_G(A,B)|\leq\delta$.
	\item\label{item:lem72nohighdegvtx} For all $\Pi'=\{A,B\}\in\Pi(G)$ such that $|V(\Pi')|\geq8\delta n$ and all vertices $v\in A$, it holds that $d_{T(G)}(v,A)\leq(1-\alpha)|A|$.
\end{enumerate}
\end{lemma}
\begin{proof}
Denote $\Pi:=\Pi(G)=\{\Pi_1,\dots,\Pi_r\}$, $\Pi_j=\{A_j,B_j\}$, and $V_j=A_j\cup B_j$. Let $r'\in[r]$ be the greatest index such that $|V_{r'}|\geq8\delta n$. We may assume $W=W_{\blambda}$, where $\blambda=(\lambda_0,\lambda_1,\dots)\in\Lambda$. Let $\mu_i:=\lambda_i-\lambda_{i-1}$ for $i\geq1$, and let $1\leq k\leq1/\delta$ be the greatest index such that $\mu_k\geq\delta$. For $i\geq1$, let $C_i$, $C_{i,1}$, and $C_{i,2}$ be the vertex sets defined in the proof of \Cref{lemma:WtoWtildeMetrics}. Notice that $|C_i-\mu_in|\leq2$ and $|C_{i,j}-\mu_in/2|\leq2$. To prove the first assertion of the lemma, it suffices to establish an injection $\varphi:[r']\to[k]$ satisfying the conditions in the definition of $\sD_W$.

Let $H_n$ be the same weighted graph as in the proof of \Cref{lemma:WtoWtildeMetrics} and assume the vertices of $G$ are labeled such that $d_\square(G,H_n)=\widehat{\delta}_\square(G,H_n)$. From the proof of \Cref{lemma:WtoWtildeMetrics}, we know $d_\square(G,H_n)\leq\epsilon^2/8$. Let $H$ be the graph on the vertex set $V$ and edge set $(E(G)\,\Delta\,D(G))\setminus E(G[\Pi_\sparse])$, where $\Delta$ is the symmetric difference (so $H$ is the graph that witnesses the fact that $\Pi$ is optimal for $G$). By \Cref{lemma:WtoWtildeMetrics}, we know $d_\square(G,H)n^2\leq d_1(G,H)n^2=b(G,\Pi)\leq\epsilon n^2$ (where $d_1$ denotes the normalized edit distance), hence
\begin{equation}
d_\square(H,H_n)\leq d_\square(H,G)+d_\square(G,H_n)\leq2\epsilon\,.\label{eqn:lem72dsquareHHnepsDist}
\end{equation}

By definition, for all $i\geq1$ and $j=1,2$, the weighted graph $H_n[C_{i,j}]$ is a clique (i.e. has all edge-weights equal to 1). Since $H$ is a disjoint union of co-bipartite graphs and an independent set, the following holds for all $1\leq i\leq k$ and $j=1,2$: there is a subset $C'_{i,j}\subseteq C_{i,j}$ of size at least $|C_{i,j}|-\ell$ such that $H[C'_{i,j}]$ is a clique, where $\ell$ is the greatest integer satisfying the inequality $\ell\cdot\delta n+\binom{\ell}{2}\leq2\epsilon n^2$\,; it is easy to check that this $\ell$ satisfies $\ell\leq\delta n/32$.

For the remainder of the proof fix $1\leq i\leq k$. Note that since $H[C'_{i,1}]$ is a clique and $H$ is a vertex-disjoint union of co-bipartite graphs and an independent set, we must have $C'_{i,1}\subseteq V_s$ for some $s\in[r]$, and we must then also have $C'_{i,2}\subseteq V_s$. We claim that
\begin{equation}
|A_s\cap C'_{i,1}|\leq\delta n/16\hspace{7mm}\text{or}\hspace{7mm}|A_s\cap C'_{i,2}|\leq\delta n/16\,.\label{eqn:ascapcprimei1contra}
\end{equation}
Assume to the contrary that neither inequality in \pref{eqn:ascapcprimei1contra} held. Now $H[A_s\cap C'_{i,1},A_s\cap C'_{i,2}]$ is a complete bipartite graph since it is a subgraph of $H[A_s]$, but the edge-weights of $H_n[A_s\cap C'_{i,1},A_s\cap C'_{i,2}]$ all equal $\rho$, which implies $d_\square(H,H_n)\geq16\epsilon(1-\rho)n^2>2\epsilon n^2$. This contradicts \pref{eqn:lem72dsquareHHnepsDist}, hence \pref{eqn:ascapcprimei1contra} must hold. Since there is symmetry in $A_s$ and $B_s$, \pref{eqn:ascapcprimei1contra} also holds with $A_s$ replaced by $B_s$.

Define the sets $U_1:=B_s\cap C'_{i,1}$ and $U_2:=A_s\cap C'_{i,2}$. We claim $|U_1|\leq\delta n/16$ and $|U_2|\leq\delta n/16$ (possibly after swapping the roles of $A_s$ and $B_s$). Due to \pref{eqn:ascapcprimei1contra} and symmetry in $A_s$ and $B_s$, we may assume at least one of the two claimed inequalities holds. Assume the second claimed inequality holds but not the first. Now $H[U_1,B_s\cap C'_{i,2}]$ is a complete bipartite graph since it is a subgraph of $H[B_s]$, but the edge-weights of the induced bipartite graph $H_n[U_1,B_s\cap C'_{i,2}]$ all equal $\rho$, which implies $d_\square(H,H_n)>2\epsilon$, contradicting \pref{eqn:lem72dsquareHHnepsDist}. Hence both claimed inequalities hold.

Define the sets $U'_1:=A_s\setminus C_i$ and $U'_2:=B_s\setminus C_i$. We now claim $|U'_1|\leq\delta n/16$ and $|U'_2|\leq\delta n/16$. Assume the first claimed inequality does not hold. Now $H[U'_1,C'_{i,1}]$ is a complete bipartite graph, but the edge-weights of the induced bipartite graph $H_n[U'_1,C'_{i,1}]$ all equal zero, which implies $d_\square(H,H_n)>2\epsilon$, contradicting \pref{eqn:lem72dsquareHHnepsDist}. By symmetry in $A_s$ and $B_s$, both claimed inequalities hold.

Combining the conclusions of the previous three paragraphs, we have
\begin{align*}
|A_s\,\Delta\,C_{i,1}| &= |A_s\setminus C_{i,1}| + |C_{i,1}\setminus A_s| \\
&= |A_s\setminus C_i|+|A_s\cap C'_{i,2}|+|A_s\cap(C_{i,2}\setminus C'_{i,2})|+|B_s\cap C'_{i,1}|+|C_{i,1}\setminus C'_{i,1}| \\
&\leq \delta n/16 + \delta n/16 + \delta n/32 + \delta n/16 + \delta n/32 \leq \delta n/4 \,,
\end{align*}
and by symmetry in $A_s$ and $B_s$ we have $|B_s\,\Delta\,C_{i,2}|\leq\delta n/4$. It follows that
\begin{equation}
|(A_s\cup B_s)\,\Delta\,C_i|\leq|A_s\,\Delta\,C_{i,1}|+|B_s\,\Delta\,C_{i,2}|\leq\delta n/2\,.\label{eqn:AsBsDiSmallDeltaineq}
\end{equation}
We have thus established an injection $f:[k]\to[r]$ such that $|(A_{f(i)}\cup B_{f(i)})\,\Delta\, C_i|\leq\delta n/2$ for all $i\in[k]$. Notice also that $f$ maps onto $[r']$ since if there existed $s\in[r']\setminus f([k])$, then we would have $d_\square(H,H_n)\geq2\epsilon$ (since $\Pi_s$ cannot satisfy a condition similar to \pref{eqn:AsBsDiSmallDeltaineq} for any $C_i$). It follows that $\varphi:[r']\to[k]$ defined by $\varphi(s)=f^{-1}(s)$ is the claimed injection.

To prove \ref{item:lem72Piprimebalanced}, notice that $||C_{i,1}|-|C_{i,2}||\leq2$, so the inequalities $|A_s\,\Delta\,C_{i,1}|\leq\delta n/4$ and $|B_s\,\Delta\,C_{i,2}|\leq\delta n/4$ imply $||A_s|-|B_s||\leq\delta n/2+2$. Using this last inequality, the conclusion of \ref{item:lem72Piprimebalanced} follows easily whenever $|V_s|\geq\alpha n$.

To prove \ref{item:lem72crossdensity}, we use all the same definitions and conclusions as in the first part of the lemma. If we define the sets $D_1:=A_s\cap C_{i,1}$ and $D_2:=B_s\cap C_{i,2}$, since the edge-weights of the induced bipartite graph $H_n[D_1,D_2]$ all equal $\rho$, we know $|e_G(D_1,D_2)-\rho|D_1|\cdot|D_2||\leq\epsilon n^2$. We compute that
\begin{align*}
d_G(A_s,B_s) &= \frac{e_G(A_s\cap C_{i,1},B_s\cap C_{i,2})+e_G(A_s\setminus C_{i,1},B_s)+e_G(A_s,B_s\setminus C_{i,2})}{|A_s|\cdot|B_s|} \\
&\leq \frac{\rho|D_1|\cdot|D_2|+\epsilon n^2+|A_s\,\Delta\,C_{i,1}|\cdot n+|B_s\,\Delta\,C_{i,2}|\cdot n}{|A_s|\cdot|B_s|} \\
&\leq \rho + \epsilon + \delta n^2/2 \leq \rho + \delta\,,
\end{align*}
and a similar calculation shows $d_G(A_s,B_s)\geq\rho-\delta$.

To prove \ref{item:lem72nohighdegvtx}, let $\Pi'=\{A,B\}\in\Pi$ and suppose $d_{T(G)}(v,A)>(1-\alpha)|A|$. By the optimality of $\Pi(G)$, we know $\overline{d}_G(v,B)\geq\overline{d}_G(v,A)$. If $\Pi''\in\sD$ is the division obtained by moving $v$ from $\Pi'$ to $\Pi_\sparse$, then using the fact that $||A|-|B||\leq2\delta n$ (proved in \ref{item:lem72PiGBalanced}),
\begin{align*}
b(G,\Pi'') &= b(G,\Pi)+d_G(v,A)+d_G(v,B)-d_{T(G)}(v,A) \\
&\leq b(G,\Pi) + \alpha|A| + (\alpha|A|+|B|-|A|) - (1-\alpha)|A| \\
&\leq b(G,\Pi) - (1-3\alpha)\cdot8\delta n + 2\delta n < b(G,\Pi) \,,
\end{align*}
which contradicts the optimality of $\Pi=\Pi(G)$, completing the proof.
\end{proof}

For all $\Pi=\{\Pi_1,\dots,\Pi_\ell\}\in\sD$ with $\Pi_i=\{A_i,B_i\}$ and all $0\leq t\leq\binom{n}{2}$, define the set
$$\cM_{\Pi,t}:=\left\{(m_1,\dots,m_\ell)\in\N^\ell:\splitfrac{\sum_{i=1}^\ell(e(\Pi_i^c)+m_i)=m-t\text{ and}\vphantom{\displaystyle\bigcup}}{\vphantom{\displaystyle\bigcup}\text{for all $i\in[\ell]$ s.t. $|V(\Pi_i)|\geq\delta n$, $\rho-\delta\leq\frac{m_i}{|A_i|\cdot|B_i|}\leq\rho+\delta$}}\right\}\,.$$
\Cref{lemma:SubCloseStructure} proves that $\cM_{\Pi,t}$ contains the sequences of numbers describing the possible number of edges between parts in optimal divisions $\Pi(G)$ of graphs $G\in N_\tau(n,m,W)$. On several occasions we will use the fact that for all $W\in\cX^\ast_\gamma$ and $\Pi\in\sD_W$,
\begin{equation}
\max_{0\leq t\leq\epsilon n^2}\sum_{\cM_{\Pi,t}}\prod_{i=1}^\ell\binom{e(\Pi_i)}{m_i} \leq 2\cdot|\cC^\ast_{W,\Pi}| \,,\label{eqn:MaxSumProdBinomCStar}
\end{equation}
which follows from the fact that for all $\bsm\in\cM_{\Pi,t}$, almost all graphs $G[V(\Pi_i)]$ such that $G\in\cC^\ast_{W,\Pi}$ and $e_G(\Pi_i)=m_i$ admit a unique 2-clique-cover, as proven in \Cref{cor:cbuniquepartition}.

\begin{definition}[The random graph $G_{\Pi,T,\bsm}$]\label{def:RandomGraphGPiTm}
Let $\Pi=\{\Pi_1,\dots,\Pi_\ell\}\in\sD$ and $T\in\cT_{W,\Pi}$. Let $0\leq t\leq m$ and let $\boldsymbol{m}\in\cM_{\Pi,t}$. For all $i\in[\ell]$, let $H_i\subseteq\Pi_i$ be the uniformly random bipartite graph with $m_i$ edges (where the graphs $H_1,\dots,H_\ell$ are independent). Let $G_{\Pi,T,\bsm}$ denote the random graph on the vertex set $V$ and edge set
\begin{equation}
E(G)=\left(\bigcup_{i=1}^\ell(E(\Pi_i^c)\cup E(H_i))\right)\Delta\, E(T)\,,\label{eqn:CWPiTRandomGraphEdges}
\end{equation}
where $\Delta$ denotes the symmetric difference.
\end{definition}

\begin{lemma}\label{lemma:SubAlmostAllNoHighHighMedVtx}
For all $W\in\cX^\ast_\gamma$ and $\Pi\in\sD_W$,
$$|\cC'_{W,\Pi}|\leq e^{-\beta n^2}|\cC^\ast_{W,\Pi}| \,,$$
where $\beta=\beta(\epsilon)>0$ is a constant.
\end{lemma}
\begin{proof}
Denote $\Pi=\{\Pi_1,\dots,\Pi_\ell\}$, $\Pi_i=\{A_i,B_i\}$, and $V_i=A_i\cup B_i$ for all $i\in[\ell]$. For all $T\in\cT_{W,\Pi}$, $\Pi'=\{A,B\}\in\Pi$ such that $|V(\Pi')|\geq\eta n$, and $v\in V\setminus V(\Pi')$, define
\[\arraycolsep=1.4pt
\begin{array}{ll}
\cC'_{W,\Pi,T} &:= \{G\in\cC''_{W,\Pi}:T(G)\cup G[\Pi_\sparse]=T\}\,, \\[5pt]
\cC'_{W,\Pi,A,T,v} &:= \{G\in\cC''_{W,\Pi,T}:\text{$d_T(v,A)>(1-\alpha)|A|$ and $d_T(v,B)\geq\alpha|B|$}\}\,.
\end{array}\]
Fix $T\in\cT_{W,\Pi}$, $\Pi_i\in\Pi$, and $v\in V\setminus V_i$ such that $|V_i|\geq\eta n$ and $\cC'_{W,\Pi,A_i,T,v}$ is nonempty. Define the quantity
$$t:=\sum_{1\leq i<j\leq\ell}e_T(V_i,V_j)+e_T(V,\Pi_\sparse)-\sum_{i=1}^\ell e_T(V_i)\,,$$
let $\bsm\in\cM_{\Pi,t}$, and let $G_{\Pi,T,\bsm}$ be the random graph from \Cref{def:RandomGraphGPiTm}. Every graph $G\in\cC'_{W,\Pi,P,T,v}$ is of the form \pref{eqn:CWPiTRandomGraphEdges} for some $\bsm\in\cM_{\Pi,t}$ and $H\subseteq\Pi$, hence
$$|\cC'_{W,\Pi,P,T,v}|\leq\sum_{\cM_{\Pi,t}}\bbP\{G_{\Pi,T,\bsm}\in\cC(n)\}\prod_{i=1}^\ell\binom{e(\Pi_i)}{m_i}\,.$$
In the remainder we bound the probability that $G_{\Pi,T,\bsm}$ is claw-free. The optimality of the division $\Pi(G)$ implies that $v\in V_j$ and $G\in\cC'_{W,\Pi,A_i,T,v}$ then $v\not\in\Pi_\sparse$ and $|V_j|\geq\eta n/4$ (since otherwise one could obtain a more optimal division by moving $v$ to $V_i$). Hence $v\in A_j$ for some $\Pi_j=\{A_j,B_j\}$ with $|V(\Pi_j)|\geq\eta n/4$. Now by applying \Cref{lemma:JansonPenaltyHighDegVtx} with variable assignments
\[\arraycolsep=5mm
\begin{array}{llll}
\Pi_1\gets\Pi_i & \Pi_2\gets\Pi_j & n_1\gets|V(\Pi_i)| & n_2\gets|V(\Pi_j)| \\[10pt]
m_1\gets m_i & m_2\gets m_j & T\gets T[V_i\cup V_j] & v\gets v \,,
\end{array}\]
we directly obtain
$$\bbP\{G_{\Pi,T,\bsm}\in\cC(n)\}\leq\exp\left(-\frac{\alpha^2\eta p_1p_2}{2^{33}}\cdot (\eta n)^2\right)\leq\exp\left(-\frac{\sqrt{\epsilon}}{2^{30}}\cdot n^2\right)\,,$$
where we used that $n_1\geq\eta n$ and $p_1p_2\geq1/8$. We thus compute
\begin{align*}
|\cC'_{W,\Pi}| &\leq \sum_{t=1}^{\epsilon n^2}\,\sum_{\substack{T\in\cT_{W,\Pi}\\e(T)=t}}\,\sum_{P\in\bigcup\Pi}\,\sum_{v\in V}|\cC'_{W,\Pi,P,T,v}| \\
&\leq \epsilon n^2\cdot\binom{\binom{n}{2}}{\epsilon n^2}\cdot n^2\cdot \exp\left(-\frac{\sqrt{\epsilon}}{2^{30}}\cdot n^2\right)\max_{0\leq t\leq\epsilon n^2}\sum_{\cM_{\Pi,t}}\prod_{i=1}^\ell\binom{e(\Pi_i)}{m_i} \\
&\leq \epsilon n^4\left(\frac{e}{2\epsilon}\right)^{\epsilon n^2}\cdot \exp\left(-\frac{\sqrt{\epsilon}}{2^{30}}\cdot n^2\right)\max_{0\leq t\leq\epsilon n^2}\sum_{\cM_{\Pi,t}}\prod_{i=1}^\ell\binom{e(\Pi_i)}{m_i} \leq e^{-\beta n^2}|\cC^\ast_{W,\Pi}| \,,
\end{align*}
where we used \pref{eqn:MaxSumProdBinomCStar} and the fact that $\sqrt{\epsilon}\gg\epsilon$ as $\epsilon\to0$, completing the proof.
\end{proof}

\begin{lemma}\label{lemma:SubAlmostAllNoMedVtx}
For all $W\in\cX^\ast_\gamma$ and $\Pi\in\sD_W$,
$$|\cC''_{W,\Pi}|\leq e^{-\beta n^2}|\cC^\ast_{W,\Pi}| \,,$$
where $\beta=\beta(\epsilon)>0$ is a constant.
\end{lemma}
\begin{proof}
Denote $\Pi=\{\Pi_1,\dots,\Pi_\ell\}$, $\Pi_i=\{A_i,B_i\}$, and $V_i=A_i\cup B_i$ for all $i\in[\ell]$. For all $T\in\cT_{W,\Pi}$, $\Pi'\in\Pi$ such that $|V(\Pi')|\geq\eta n$, $P\in\Pi'$, and $v\in V$, define the sets
\[\arraycolsep=1.4pt
\begin{array}{ll}
\cC''_{W,\Pi,T} &:= \{G\in\cC''_{W,\Pi}:T(G)\cup G[\Pi_\sparse]=T\}\,, \\[5pt]
\cC''_{W,\Pi,P,T,v} &:= \{G\in\cC''_{W,\Pi,T}:v\text{ has $(P,T(G))$-medium degree}\}\,.
\end{array}\]
Fix $T\in\cT_{W,\Pi}$, $\Pi_i\in\Pi$, $A\in\Pi_i$, and $v\in V$ such that $|V_i|\geq\eta n$ and $\cC''_{W,\Pi,A,T,v}$ is nonempty. Let $B\in\Pi_i\setminus\{A\}$ be the other part in $\Pi_i$. Define
$$t:=\sum_{1\leq i<j\leq\ell}e(T[V_i,V_j])+e(T[V,\Pi_\sparse])-\sum_{i=1}^\ell e(T[V_i])\,,$$
let $\bsm\in\cM_{\Pi,t}$, and let $G_{\Pi,T,\bsm}$ be the random graph from \Cref{def:RandomGraphGPiTm}. Every graph $G\in\cC''_{W,\Pi,P,T,v}$ is of the form \pref{eqn:CWPiTRandomGraphEdges} for some $\bsm\in\cM_{\Pi,t}$ and $H\subseteq\Pi$, hence
$$|\cC''_{W,\Pi,P,T,v}|\leq\sum_{\cM_{\Pi,t}}\bbP\{G_{\Pi,T,\bsm}\in\cC(n)\}\prod_{i=1}^\ell\binom{e(\Pi_i)}{m_i}\,.$$
In the remainder we bound the probability that $G_{\Pi,T,\bsm}$ is claw-free. We apply \Cref{lemma:JansonPenaltyMedVtx} with the following variable assignments:
\[\arraycolsep=5mm
\begin{array}{llll}
\Pi\gets\Pi_i & V\gets V_i & m\gets m_i & v\gets v \\[10pt]
N\gets\overline{N}_T(v,A) & N'\gets\overline{N}_T(v,B) & T\gets T[V_i] \,.
\end{array}\]
Note that $\overline{N}_T(v,B)\geq\alpha|B|$ by definition of $\cC''_{W,\Pi}$, so the hypotheses of \Cref{lemma:JansonPenaltyMedVtx} are indeed met, implying
$$\bbP\{G_{\Pi,T,\bsm}\in\cC(n)\}\leq\exp\bigg(-\frac{\epsilon^{3/4}p}{2}(\eta n)^2\bigg)\leq\exp\bigg(-\frac{\epsilon^{7/8}}{8}n^2\bigg)\,,$$
where we used that $|V_i|\geq\eta n$ and $p\geq1/4$. We thus compute
\begin{align*}
|\cC''_{W,\Pi}| &\leq \sum_{t=1}^{\epsilon n^2}\,\sum_{\substack{T\in\cT_{W,\Pi}\\e(T)=t}}\,\sum_{P\in\bigcup\Pi}\,\sum_{v\in V}|\cC''_{W,\Pi,P,T,v}| \\
&\leq \epsilon n^2\cdot\binom{\binom{n}{2}}{\epsilon n^2}\cdot n^2\cdot\exp\bigg(-\frac{\epsilon^{7/8}}{8}n^2\bigg)\max_{0\leq t\leq\epsilon n^2}\sum_{\cM_{\Pi,t}}\prod_{i=1}^\ell\binom{e(\Pi_i)}{m_i} \\
&\leq \epsilon n^4\left(\frac{e}{2\epsilon}\right)^{\epsilon n^2}\cdot \exp\bigg(-\frac{\epsilon^{7/8}}{8}n^2\bigg)\max_{0\leq t\leq\epsilon n^2}\sum_{\cM_{\Pi,t}}\prod_{i=1}^\ell\binom{e(\Pi_i)}{m_i} \leq e^{-\beta n^2}|\cC^\ast_{W,\Pi}| \,,
\end{align*}
where we used \pref{eqn:MaxSumProdBinomCStar} and the fact that $\epsilon^{7/8}\gg\epsilon$ as $\epsilon\to0$, completing the proof.
\end{proof}

\begin{lemma}\label{lemma:NtaunmWUpperBd}
For all $W\in\cX^\ast_\gamma$ and $\Pi\in\sD_W$,
$$|\cC_{W,\Pi}|\leq e^{-\xi n}|\cC^\ast_{W,\Pi}|\,,$$
where $\xi>0$ is a constant, hence
$$N_\tau(n,m,W)\leq(1+e^{-\xi n})\sum_{\Pi\in\sD_W}|\cC^\ast_{W,\Pi}|\,.$$
\end{lemma}
\begin{proof}
We first introduce definitions that will be used throughout the proof. Denote $\Pi=\{\Pi_1,\dots,\Pi_{\ell'}\}$, $\Pi_i=\{A_i,B_i\}$, $V_i:=A_i\cup B_i$, and $v_i:=|V_i|$ for all $i\in[\ell']$. Let $\ell\in[\ell']$ be the greatest index such that $v_\ell\geq\eta n$. Define the set of graphs
$$\cT'_{W,\Pi}:=\{T(G)\cup G[\Pi_\sparse]:G\in\cC_{W,\Pi}\setminus(\cC'_{W,\Pi}\cup\cC''_{W,\Pi})\}\,.$$
For all $T\in\cT'_{W,\Pi}$ and $i\in[\ell]$, let $M_i(T)$ denote a canonically chosen maximum matching of the graph $T[V_i]$. Let $X_i(T):=V(M_i(T))$ denote the set of vertices covered by $M_i(T)$.

For all $v\in V$, define the set
$$\High(v):=\{i\in[\ell]:\text{$d_T(v,A_i)>(1-\alpha)|A_i|$ or $d_T(v,B_i)>(1-\alpha)|B_i|$}\}\,.$$
Notice that for all $G\in C_{W,\Pi}\,$, if $v\in V_i$ and $j\in\High(v)$ then the optimality of $\Pi=\Pi(G)$ implies $v_i\geq v_j/4$ (since otherwise one could obtain a more optimal division for $G$ by moving $v$ to $V_j$). Let $r\in[\ell']$ be the greatest index such that $v_r\geq\eta n/4$. By definition of $r$, the set $\High(v)$ is empty whenever $i>r$ and $v\in V_i$.

For all $T\in\cT'_{W,\Pi}$ and $i\in[r]$, define the set
$$V_i^\high(T):=\left\{v\in V_i:\High(v)\neq\emptyset\right\}\,,$$
and let $D_i(T)\subseteq V_i^\high(T)$ be a subset of either $A_i$ or $B_i$ of size exactly $|V_i^\high(T)|/2$. For all $T\in\cT'_{W,\Pi}$, let $T_\high$ be the graph on the edge set
$$E(T_\high)=\bigcup_{i=1}^r\,\bigcup_{v\in V_i^\high(T)}\,\bigcup_{j\in\High(v)}E_T(v,V_j)\,,$$
where the vertex set of $T_\high$ is induced by its edges. Also let $T_\low$ be the graph on the edge set $E(T)\setminus E(T_\high)$, where the vertex set is again induced by the edge set.

For all $i\in[r]$, $v\in D_i(T)$ (we may assume $v\in A_i$), and $j\in\High(v)$, there exists a vertex $h_v\in N_T(v,V_j)$ such that $d_T(h_v,B_i)\leq\alpha|B_i|$; indeed, if there were no such vertex $h_v$ then we would have
$$e(T)\geq e_T(N_T(v,V_j),B_i)\geq|N_T(v,V_j)|\cdot\alpha|B_i|\geq\frac{\eta n}{4}\cdot\frac{\alpha\eta n}{16}>\epsilon n^2\,,$$
which contradicts the fact that $e(T)\leq\epsilon n^2$. 

Define the vertex set
$$S:=\Pi_\sparse\cup\bigcup_{\ell<i\leq\ell'}V_i\,.$$
For all vectors $\boldsymbol{d}\in\{0,1,2\dots\}^{\binom{\ell}{2}+2\ell+r}$ we will index $\boldsymbol{d}$ using the components $d_{ij}$ for $1\leq i<j\leq\ell$, the components $d_i$ and $d_{\sparse,i}$ for $i\in[\ell]$, and the components $d_{\high,i}$ for $i\in[r]$.
For all such $\bsd$, define the set of graphs
$$\cT_{W,\Pi,\boldsymbol{d}}:=\left\{T\in\cT'_{W,\Pi}:\splitfrac{\text{for all $i\in[\ell]$, $j\in[\ell]\setminus\{i\}$, and $k\in[r]$, $e_{T_\low}(V_i,V_j)=d_{ij}$}\vphantom{\displaystyle\bigcup}}{\text{$|X_i(T)|=d_i$, $|V_k^\high(T)|=d_{\high,k}$, and $e_{T_\low}(V_i,S)=d_{\sparse,i}$}}\right\}\,.$$

Fix $T\in\cT'_{W,\Pi}$ and define the quantity
$$t:=\sum_{1\leq i<j\leq\ell'}e(T[V_i,V_j])+e(T[V,\Pi_\sparse])-\sum_{i=1}^{\ell'}e(T[V_i])\,.$$
For all $\bsm\in\cM_{\Pi,t}$ let $G_{\Pi,T,\bsm}$ be the random graph from \Cref{def:RandomGraphGPiTm}. \Cref{lemma:SubCloseStructure} proves that every graph $G\in\cC_{W,\Pi,T}$ is of the form \pref{eqn:CWPiTRandomGraphEdges} for some $\bsm\in\cM_{\Pi,t}$ and $H\subseteq\Pi$, hence
\begin{align*}
|\cC_{W,\Pi,T}| \leq \sum_{\cM_{\Pi,t}}\bbP\{G_{\Pi,T,\bsm}\in\cC(n)\}\prod_{i=1}^\ell\binom{e(\Pi_i)}{m_i} \,.
\end{align*}

For all $T\in\cT_{W,\Pi,\boldsymbol{d}}$ and $i\in[r]$, let $J_i:=\{h_v:v\in D_i(T)\}$ and assume without loss of generality $D_i(T)\subseteq A_i$. We can then apply \Cref{lemma:JansonPenaltyMatching} with the following variable assignments:
\[\arraycolsep=5mm
\begin{array}{lllll}
\Pi\gets\Pi_i & m\gets m_i & D\gets D_i(T) & D'\gets J_i & F'\gets T[J_i,B_i] \\[10pt]
k\gets d_i & d\gets d_{\high,i}/2 & l\gets d_{\sparse,i} & C\gets V\setminus V_i & F\gets T_\low[V_i,V\setminus V_i]\,.
\end{array}\]
Similarly for $\ell<i\leq r$, we apply \Cref{lemma:JansonPenaltyMatching} with the variable assignments
\[\arraycolsep=5mm
\begin{array}{lllll}
\Pi\gets\Pi_i & m\gets m_i & D\gets D_i(T) & D'\gets J_i & F'\gets T[J_i,B_i] \\[10pt]
k\gets 0 & d\gets d_{\high,i}/2 & l\gets 0 & C\gets\emptyset & F\gets\emptyset \,.
\end{array}\]
For all $i\in[\ell]$ define $q_i:=\eta^2\big(\sum_{j\neq i}d_{ij}+d_{\sparse,i}\big)$. Since the random bipartite graphs $H_1,\dots,H_r$ from the definition of $G_{\Pi,T,\bsm}$ are all independent, the aforementioned applications of \Cref{lemma:JansonPenaltyMatching} imply
\begin{equation}
\bbP\{G_{\Pi,T,\bsm}\in\cC(n)\} \leq \left(\prod_{i=1}^\ell p_i^{d_iv_i/2^{11}}\exp\left(-\frac{p_i^4}{2^{21}}\cdot v_i\cdot\min\left\{q_i\,,\,v_i\right\}\right)\right)\left(\prod_{i=1}^rp^{d_{\high,i}v_i/64}\right) \,,\label{eqn:PGPiTmleqprodexp}
\end{equation}
where $p_i:=m_i/(|A_i|\cdot|B_i|)$ for all $i\in[r]$.

For all $i\in[\ell]$, $j\in[\ell]\setminus\{i\}$, and $k\in[r]$, define the sets of graphs
\[\arraycolsep=1.4pt
\begin{array}{ll}
\cT_{W,\Pi,\boldsymbol{d}}^i &:= \{T[V_i]:T\in\cT_{W,\Pi,\bsd}\} \,, \\[10pt]
\cT_{W,\Pi,\boldsymbol{d}}^{\sparse,i} &:= \{T_\low[V_i,\,S]:T\in\cT_{W,\Pi,\bsd}\} \,, \\[10pt]
\cT_{W,\Pi,\boldsymbol{d}}^{ij} &:= \{T_\low[V_i,\,V_j]:T\in\cT_{W,\Pi,\bsd}\} \,, \\[10pt]
\cT_{W,\Pi,\bsd}^{\high,k} &:= \{T_\high[V_k,\,V\setminus V_k]:T\in\cT_{W,\Pi,\bsd}\} \,, \\[10pt]
\cT_{W,\Pi}^\sparse &:= \{T[S]:T\in\cT'_{W,\Pi}\} \,.
\end{array}\]
We now bound the sizes of these sets. By combining \Cref{lemma:SubCloseStructure} \ref{item:lem72nohighdegvtx} and the definition of $\cT'_{W,\Pi}$, we know that for all $T\in\cT'_{W,\Pi}$, $G\in\cC_{W,\Pi,T}$, and $i\in[\ell]$, every vertex $v\in A_i$ satisfies $d_T(v,A_i)<\alpha|A_i|$. Additionally, since $M_i(T)$ is a maximum matching of the graph $T[V_i]$, every edge $e\in E_T(V_i)$ has an endpoint in $X_i(T)$. These facts directly imply
\begin{align*}
\begin{split}
|\cT_{W,\Pi,\bsd}^i| &\leq \sum_{X\in\binom{V_i}{d_i}}\,\prod_{P\in\Pi_i}\,\prod_{v\in P\cap X}\binom{|P|}{\leq\alpha|P|} \leq \sum_{X\in\binom{V_i}{d_i}}\,\prod_{P\in\Pi_i}\,\prod_{v\in P\cap X}\alpha v_i\binom{v_i}{\alpha v_i} \\
&\leq \binom{v_i}{d_i}\left(\alpha v_i\binom{v_i}{\alpha v_i}\right)^{d_i} \leq 2^{H(\alpha)d_iv_i+d_i\log_2v_i+d_i\log_2\alpha} \leq 2^{2H(\alpha)d_iv_i} \,.
\end{split}
\end{align*}
We also have the trivial bounds
\begin{align*}
\begin{split}
|\cT_{W,\Pi,\bsd}^{\sparse,i}| &\leq \binom{\binom{n}{2}}{d_{\sparse,i}}\leq\left(\frac{en^2}{2d_{\sparse,i}}\right)^{d_{\sparse,i}} \,, \\
|\cT_{W,\Pi,\bsd}^{ij}| &\leq \binom{v_iv_j}{d_{ij}} \leq\left(\frac{ev_iv_j}{d_{ij}}\right)^{d_{ij}} \,, \\
|\cT_{W,\Pi,\bsd}^{\high,i}| &\leq \binom{n}{(1-\alpha)n}^{d_{\high,i}}\binom{n}{\alpha n}^{d_{\high,i}} \leq 2^{H(\alpha)\cdot d_{\high,i}n} \,,
\end{split}
\end{align*}
which we combine to deduce
\begin{align}
|\cT_{W,\Pi,\bsd}| &\leq \Bigg(\prod_{i=1}^\ell|\cT_{W,\Pi,\boldsymbol{d}}^i|\cdot|\cT_{W,\Pi,\boldsymbol{d}}^{\sparse,i}|\Bigg)\Bigg(\prod_{1\leq i<j\leq\ell}|\cT_{W,\Pi,\boldsymbol{d}}^{ij}|\Bigg)\Bigg(\prod_{i=1}^r|\cT_{W,\Pi,\bsd}^{\high,i}|\Bigg)\cdot|\cT_{W,\Pi}^\sparse| \nonumber\\
\begin{split}
&\leq \exp\Bigg(\sum_{i=1}^\ell\left(2(\log2)H(\alpha)d_iv_i+d_{\sparse,i}\log\left(\frac{en^2}{2d_{\sparse,i}}\right)\right) \\
&\hspace{3cm}+\sum_{1\leq i<j\leq\ell}d_{ij}\log\left(\frac{ev_iv_j}{d_{ij}}\right)+\sum_{i=1}^rH(\alpha)(\log2)d_{\high,i}n\Bigg)\cdot|\cT_{W,\Pi}^\sparse| \,.
\end{split}\label{eqn:TWPidcombinedupperbdd}
\end{align}

For all $T\in\cT_{W,\Pi}$ and $\bsm\in\cM_{\Pi,t}$ let $G_{\Pi,T,\bsm}$ be the random graph from \Cref{def:RandomGraphGPiTm}. Define the set of graphs
$$\cC_{W,\Pi,\bsd}:=\bigcup_{T\in\cT_{W,\Pi,\bsd}}\cC_{W,\Pi,T}\,.$$
Now by comparing terms pairwise between expressions \pref{eqn:PGPiTmleqprodexp} and \pref{eqn:TWPidcombinedupperbdd}, we obtain
\begin{align*}
|\cC_{W,\Pi,\bsd}| &\leq \sum_{t=1}^{\epsilon n^2}\,\sum_{\cM_{\Pi,t}}\,\sum_{\substack{T\in\cT_{W,\Pi,\boldsymbol{d}}\\e(T)=t}}\bbP\{G_{\Pi,T,\bsm}\in\cC(n)\}\prod_{i=1}^\ell\binom{e(\Pi_i)}{m_i} \\
&\leq \epsilon n^2 \cdot |\cT_{W,\Pi,\bsd}| \cdot\left(\prod_{i=1}^\ell p_i^{d_iv_i/2^{11}}\exp\left(-\frac{p_i^4}{2^{21}}\cdot v_i\cdot\min\left\{q_i\,,\,v_i\right\}\right)\right)\left(\prod_{i=1}^rp^{d_{\high,i}v_i/64}\right) \\
&\hspace{4cm} \cdot\max_{0\leq t\leq\epsilon n^2}\sum_{\cM_{\Pi,t}}\prod_{i=1}^\ell\binom{e(\Pi_i)}{m_i} \\
&\leq e^{-\zeta'n}\cdot|\cT_{W,\Pi}^\sparse|\cdot\max_{0\leq t\leq\epsilon n^2}\sum_{\cM_{\Pi,t}}\prod_{i=1}^\ell\binom{e(\Pi_i)}{m_i} \leq e^{-\zeta n}|\cC^\ast_{W,\Pi}|
\end{align*}
for constants $\zeta,\zeta'>0$. In the above calculation we used the fact that $\eta^2d_{\sparse,i}$ and $\eta^2d_{ij}$ each appear in the penalty term involving $q_i$\,; namely, if $v_i\leq q_i$ then since $v_i\geq\eta n$, the penalty from this term is of order $\exp(-\Omega(\alpha n^2))$; otherwise if $q_i<v_i$ then the penalty is of order $\exp(-\omega(\eta^3(d_{ij}+d_{\sparse,i})n))$; in both cases, the terms from $|\cT^{\sparse,i}_{W,\Pi,\bsd}|$ and $|\cT^{ij}_{W,\Pi,\bsd}|$ are compensated for. We also used that $H(\alpha)\to0$ as $\epsilon\to0$, so that for small enough $\epsilon>0$,
$$2(\log2)H(\alpha)d_iv_i<\log\left(\frac{1}{p}\right)d_iv_i\,,$$
and a similar comparison applies to the terms involving $d_{\high,i}$ using the fact that $v_i\geq\eta n$ and $\eta\gg H(\alpha)$ as $\epsilon\to0$. The last inequality bounding $|\cC_{W,\Pi,\bsd}|$ used \pref{eqn:MaxSumProdBinomCStar} (note we used a slightly stronger inequality than \pref{eqn:MaxSumProdBinomCStar}, which holds by definition of $\cC^\ast_{W,\Pi}$ and $\cT_{W,\Pi}^\sparse$).

Let $\beta>0$ be the minimum of the constants from \Cref{lemma:SubAlmostAllNoHighHighMedVtx,lemma:SubAlmostAllNoMedVtx}. Using the above bound on $|\cC_{W,\Pi,\bsd}|$ we compute that
\begin{align*}
|\cC_{W,\Pi}| &\leq \sum_{\bsd}|\cC_{W,\Pi,\bsd}|+|\cC'_{W,\Pi}|+|\cC''_{W,\Pi}| \leq (n^{8r^2}e^{-\zeta n}+2e^{-\beta n^2})|\cC^\ast_{W,\Pi}| \leq e^{-\xi n}|\cC^\ast_{W,\Pi}|
\end{align*}
for a constant $\xi>0$, where we used the fact that $\cC_{W,\Pi,\bsd}$ is nonempty for at most $n^{8r^2}$ of the vectors $\bsd\in\{0,1,2,\dots\}^{\binom{\ell}{2}+2\ell+r}$, completing the proof of the first statement of the lemma.

The second assertion of the lemma follows immediately by observing
$$|N_\tau(n,m,W)|\leq\sum_{\Pi\in\sD_W}(|\cC_{W,\Pi}|+|\cC^\ast_{W,\Pi}|)\,,$$
which uses the statement in \Cref{lemma:SubCloseStructure} that $\Pi(G)\in\sD_W$ for all $G\in N_\tau(n,m,W)$.
\end{proof}

The next lemma is one of the key steps of this section and the proof of \Cref{thm:almostallCB} \rm{\ref{thm:almostallCBsubcritical}}. It asserts that for any optimal graphon $W\in\cX^\ast_\gamma$ that is not equivalent to $W^\ast$, if we make $\epsilon>0$ sufficiently small then there are at least $n^{\Theta(n)}$ more claw-free graphs $G\in\cC(n,m)$ that are close to $W^\ast$ in cut metric than that are close to $W$.

\begin{lemma}\label{lemma:CompareBallsInCutMetricNlogN}
If $W\in\cX^\ast_\gamma$ and $\delta_\square(W,W^\ast)\geq 8\omega$ then
$$|N_\tau(n,m,W)|\leq n^{-\nu n}|N_{\omega}(n,m,W^\ast)|$$
for a constant $\nu>0$ depending only on $\epsilon$.
\end{lemma}

Before beginning the proof of \Cref{lemma:CompareBallsInCutMetricNlogN} we introduce some more notation and prove one more lemma. Let $W\in\cX^\ast_\gamma$, $\Pi=\{\Pi_1,\dots,\Pi_\ell\}\in\sD_W$, $\Pi_i=\{A_i,B_i\}$, and $V_i:=A_i\cup B_i$. Let $\ell'\in[\ell]$ be the greatest index such that $|V_{\ell'}|\geq\eta n$ (note this is well-defined by definition of $\eta$), and let $\Pi':=\{\Pi_1,\dots,\Pi_{\ell'}\}$. For all such $\Pi$ and for all $0\leq t\leq\binom{n}{2}$, define the set
$$\cR_{\Pi,t} := \left\{F\subseteq\Pi':e_F(V(\Pi_i))=m_i\,,\,i\in[\ell']\,,\,\bsm\in\cM_{\Pi,t}\right\}\,,$$
where $F\subseteq\Pi'$ means $F$ is a subgraph of the disjoint union $\bigcup_{i=1}^{\ell'}K_{\Pi_i}$ of bipartite graphs. Define the vertex set
$$S_\Pi:=\Pi_\sparse\cup\bigcup_{\ell'<i\leq\ell}V_i$$
and the integer $s(\Pi):=|S_\Pi|$. For all $0\leq t\leq\binom{n}{2}$, define the set of claw-free graphs
$$\cD_{\Pi,t}:=\left\{\left(\bigcup_{i=1}^{\ell'}G_i\right)\sqcup S:G_i^c\subseteq\Pi_i\,,\,e_{G_i}(A_i,B_i)=m_i\,,\bsm\in\cM_{\Pi,t}\,,\,S\in\cC(s(\Pi),t)\right\}\,,$$
where $\sqcup$ denotes the vertex-disjoint union.

We claim that for all $\Pi\in\sD_{W^\ast}$ and all $0\leq t\leq4\eta n^2$,
\begin{equation}
|\cR_{\Pi,t}|\cdot|\cC(s(\Pi),t)|=|\cD_{\Pi,t}|\leq 2\cdot|N_{\omega}(n,m,W^\ast)|\,.\label{eqn:RPitCSPitDPit}
\end{equation}
The equality in \pref{eqn:RPitCSPitDPit} holds automatically by definition of $\cR_{\Pi,t}$, $\cD_{\Pi,t}$, and $s(\Pi)$. To see the inequality, first note that by \Cref{lemma:SubCloseStructure} \ref{item:lem72PiGBalanced} every $\Pi=\{\Pi_1,\dots,\Pi_\ell\}\in\sD_{W^\ast}$ satisfies $||V(\Pi_1)|-2\mu n|\leq\delta n$ and $|V(\Pi_i)|\leq8\delta n$ for all $i\geq2$. Also, a uniformly random element $G$ of $\cD_{\Pi,t}$ has regularly distributed edges $E(G[\Pi_i])$ for all $i\in[\ell]$ with high probability. It follows that $\widehat{\delta}_\square(G,W^\ast_n)\leq\omega/2$ with high probability, and since $\delta_\square(G_1,G_2)\leq\widehat{\delta}_\square(G_1,G_2)$ always, inequality \pref{eqn:RPitCSPitDPit} follows.

For positive numbers $X(n,\epsilon)$ and $Y(n,\epsilon)$, define the relation $X(n,\epsilon)\preceq Y(n,\epsilon)$ to mean 
\begin{equation}
\lim_{\epsilon\to0}\liminf_{n\to\infty}\frac{\log Y(n,\epsilon)-\log X(n,\epsilon)}{n\log n}\geq0\,.\label{eqn:nlognNotionOfIneq}
\end{equation}
For all $U\in\cX^\ast_\gamma$, let $\sD'_U$ denote the set of all $\Pi\in\sD_U$ such that $\min\{|V(\Pi')|:\Pi'\in\Pi\}\geq\epsilon n$. Notice the cardinality of $\sD'_U$ is small with respect to the relation $\preceq$ in the following sense: if we let $N:=\ceil{1/\epsilon}+1$ then
\begin{equation}
\begin{aligned}
|\sD'_U| &\leq \sum_{\substack{x_1+\cdots+x_N=n\\x_i\geq0}}\binom{n}{x_1,\dots,x_N}2^{x_1+\cdots+x_N} \\
&\leq (n+N)^N\cdot2^{H(x_1/n\,,\dots,\,x_N/n)n}\cdot 2^{Nn}\preceq1\,,
\end{aligned}\label{eqn:sDprimeUmutlinomialupperbd}
\end{equation}
where we used a standard bound on the multinomial coefficient (see e.g. \cite[Lemma~2.2]{csiszar2004information}) and $H(y_1,\dots,y_n)=-\sum_iy_i\log_2 y_i$ is the entropy.

\begin{lemma}\label{lemma:RPitSameOrderUpToNlogN}
For all $W,W'\in\cX^\ast_\gamma$ and all $0\leq t\leq4\eta n^2$,
\begin{equation}
\sum_{\Pi\in\sD'_{W}}|\cR_{\Pi,t}|\succeq\sum_{\Pi\in\sD'_{W'}}|\cR_{\Pi',t}|\,,\label{eqn:limEpslimNsumDPitrStargamma}
\end{equation}
where the relation $\succeq$ is meant in the sense of \pref{eqn:nlognNotionOfIneq}.
\end{lemma}
\begin{proof}
Let $\blambda=(\lambda_0,\lambda_1,\dots)\in\Lambda$ be the unique sequence such that $W$ is equivalent to $W_{\blambda}$ (using the definition of $W_{\blambda}$ in \Cref{eqn:flambdaGraphonDef}). Denote $\Pi=\{\Pi_1,\dots,\Pi_q\}\in\sD_W$ and let $\ell\in[q]$ be the greatest index such that $|V(\Pi_\ell)|\geq\eta n$. Denote $\Pi'=\{\Pi'_1,\dots,\Pi'_r\}\in\sD_{W'}$ and let $\ell'\in[r]$ be the greatest index such that $|V(\Pi_{\ell'})|\geq\eta n$. By definition of $\cR_{\Pi,t}$ we have
$$|\cR_{\Pi,t}|=\sum_{\cM_{\Pi,t}}\prod_{i=1}^{\ell}\binom{e(\Pi_i)}{m_i}\hspace{7mm}\text{and}\hspace{7mm}|\cR_{\Pi',t}|=\sum_{\cM_{\Pi',t}}\prod_{i=1}^{\ell'}\binom{e(\Pi'_i)}{m_i}\,.$$
Using \pref{eqn:sDprimeUmutlinomialupperbd} and $|\cM_{\Pi',t}|\leq(2\delta n)^\ell$, it follows that \pref{eqn:limEpslimNsumDPitrStargamma} is implied by the assertion
\begin{equation}
\lim_{\epsilon\to0}\liminf_{n\to\infty}\frac{1}{n\log_2n}\left(\min\sum_{i=1}^\ell\log_2\binom{e(\Pi_i)}{m_i}-\max\sum_{i=1}^{\ell'}\log_2\binom{e(\Pi'_i)}{m'_i}\right)\geq0\,,\label{eqn:nlognequivalencesuffcond}
\end{equation}
where the minimum is over all $\Pi\in\sD'_W$ and $\bsm\in\cM_{\Pi,t}$, and the maximum is over all $\Pi'\in\sD'_{W'}$ and $\bsm'\in\cM_{\Pi,t}$. Notice that
\begin{equation}
\sum_{i=1}^\ell\log_2\binom{e(\Pi_i)}{m_i}=\sum_{i=1}^\ell \left(H\left(\frac{m_i}{e(\Pi_i)}\right)e(\Pi_i)+O(n)\right)\,,\label{eqn:sumlogentropyePii}
\end{equation}
which follows easily from Stirling's formula. As $\epsilon\to0$ and $n\to\infty$, we have $m_i/e(\Pi_i)\to\frac{3-\sqrt{5}}{2}$ and $e(\Pi_i)\sim(\lambda_j-\lambda_{j-1})^2n^2/4$ for some $j\geq1$ by \Cref{lemma:SubCloseStructure} \ref{item:lem72PiGBalanced}, which implies \pref{eqn:sumlogentropyePii} is asymptotic to $H(W)\binom{n}{2}+O(n)$. Inequality \pref{eqn:nlognequivalencesuffcond} now follows since $H(W)=H(W')$.
\end{proof}

The main result of \cite{mckay2003asymptotic} is an asymptotic formula for the number $\kappa(n)$ of cubic claw-free graphs on $2n$ vertices:
$$\kappa(n)\sim\frac{(2n)!}{e\sqrt{6\pi n}}\left(\frac{n}{2e}\right)^{n/3}e^{(n/2)^{1/3}}=e^{\frac{7}{3}n\log n+O(n)}\,.$$
The asymptotics of $\kappa(n)$ play a key role in the proof of \Cref{lemma:CompareBallsInCutMetricNlogN}.

\begin{proof}[Proof of \Cref{lemma:CompareBallsInCutMetricNlogN}]
Let $\blambda=(\lambda_0,\lambda_1,\dots)\in\Lambda$ be the unique sequence such that $W$ is equivalent to $W_{\blambda}$ (using the definition of $W_{\blambda}$ in \Cref{eqn:flambdaGraphonDef}). Let $1\leq j\leq1/\eta$ be the greatest index such that $\lambda_j-\lambda_{j-1}\geq\eta$ (this is well-defined by definition of $\eta$), and let $\lambda:=\lambda_j$. We claim that $\lambda\geq2\mu+\omega$. Suppose to the contrary $\lambda<2\mu+\omega$. Then we must have $\lambda_1\geq2\mu-\omega$, since otherwise
\begin{align*}
t(K_2,W) &\leq (2\mu-\omega)^2\cdot\frac{1}{2}\bigg(1+\frac{3-\sqrt{5}}{2}\bigg)+4\omega^2+4\eta = \gamma+\frac{37-\sqrt{5}}{4}\cdot\eta-(5-\sqrt{5})\mu\omega<\gamma\,.
\end{align*}
But the inequalities $2\mu-\omega\leq\lambda_1\leq\lambda<2\mu+\omega$ imply $\delta_\square(W,W^\ast)\leq\norm{W-W^\ast}_1<8\omega$, a contradiction. We thus have $\lambda\geq2\mu+\omega$.

Let $k:=\floor{\omega n/4}$ and $s:=\ceil{(1-2\mu-\omega+\alpha)n}$. It is straightforward from \Cref{lemma:SubCloseStructure} \ref{item:lem72PiGBalanced} and the inequality $\lambda\geq2\mu+\omega$ that $s(\Pi)\leq s$ for all $\Pi\in\sD_W$. Note also that every division $\Pi\in\sD_W$ satisfies $V(\Pi)\geq(2\mu-\delta)n+2k$, and every division $\Pi\in\sD_{W^\ast}$ satisfies $s(\Pi)\geq s+2k$. By definition of $S_\Pi$, every graph $G\in\cC^\ast_{W,\Pi}$ satisfies $e_G(S_\Pi)\leq4\eta n^2$, which directly implies
$$|\cC^\ast_{W,\Pi}|\leq\sum_{t=0}^{4\eta n^2}|\cR_{\Pi,t}|\cdot|\cC(s(\Pi),t)|\leq\sum_{t=0}^{4\eta n^2}|\cR_{\Pi,t}|\cdot|\cC(s,t)|\,.$$

We will use that for all $U\in\cX^\ast_\gamma$ and $\Pi\in\sD_U$, $k^{2k}|\cR_{\Pi,t}|\leq|\cR_{\Pi,t+3k}|\cdot\kappa(k)$, which follows from \Cref{lemma:BinomialCoeffIneqs} since for any $1\leq i\leq\ell'$,
$$k^{2k}\binom{e(\Pi_i)}{m_i}\binom{e(\Pi_i)}{m_i-3k}^{-1}\leq k^{2k}\exp\left(-\left(1-\frac{e(\Pi_i)-m_i+3k}{m_i-3k}\right)3k\right)=k^{2k}e^{O(n)}\leq\kappa(k)\,.$$

Let $N_1(n):=N_{\omega}(n,m,W^\ast)$ and $N_2(n):=N_\tau(n,m,W)$. Using \Cref{lemma:NtaunmWUpperBd}, \Cref{lemma:RPitSameOrderUpToNlogN}, \pref{eqn:RPitCSPitDPit}, and \pref{eqn:sDprimeUmutlinomialupperbd}, and using the relation $\preceq$ in the sense of \pref{eqn:nlognNotionOfIneq}, we compute
\begin{align*}
k^{2k}\cdot|N_2(n)| &\preceq k^{2k}\sum_{\Pi\in\sD_W}|\cC^\ast_{W,\Pi}| \\
&\leq k^{2k}\sum_{\Pi\in\sD'_W}\sum_{t=0}^{4\eta n^2}|\cR_{\Pi,t}|\cdot|\cC(s,t)| \\
&\preceq k^{2k}\max_{0\leq t\leq4\eta n^2}\left\{|\cC(s,t)|\sum_{\Pi\in\sD'_W}|\cR_{\Pi,t}|\right\} \\
&\preceq k^{2k}\cdot\sum_{t=0}^{4\eta n^2}|\cC(s,t)|\sum_{\Pi\in\sD'_{W^\ast}}|\cR_{\Pi,t}| \\
&\preceq \sum_{\Pi\in\sD'_{W^\ast}}\sum_{t=0}^{4\eta n^2-3k}|\cR_{\Pi,t+3k}|\cdot|\cC(s,t)|\cdot\kappa(k) \\
&\preceq \sum_{\Pi\in\sD'_{W^\ast}}\sum_{t=0}^{4\eta n^2}|\cR_{\Pi,t}|\cdot|\cC(s(\Pi),t)| \preceq |N_1(n)| \,.
\end{align*}
Since $k=\floor{\omega n/4}$, it directly follows that $|N_2(n)|\leq n^{-\nu n}|N_1(n)|$ for a constant $\nu>0$ depending only on $\epsilon$, completing the proof.
\end{proof}

\begin{proof}[Proof of \Cref{thm:almostallCB} \rm{\ref{thm:almostallCBsubcritical}}]
Let $\cN=\{W^\ast,W_1,\dots,W_k\}\subseteq\cX^\ast_\gamma$ be a $\tau$-net of $\cX^\ast_\gamma$ in cut metric. Define the sets
\[\arraycolsep=1.4pt
\begin{array}{ll}
\cC_\close &:= \displaystyle N_\tau(n,m,W^\ast)\cup\bigcup_{i=1}^kN_\tau(n,m,W_i) \,, \\[15pt]
\cC_\far &:= \cC(n,m)\setminus\cC_\close \,.
\end{array}\]
Let $I:=\{i\in[k]:\delta_\square(W_i,W^\ast)\geq8\omega\}$ and $\tau':=\tau+8\omega$, so that
$$|\cC_\close|\leq|N_{\tau'}(n,m,W^\ast)|+\sum_{i\in I}|N_\tau(n,m,W_i)|\,.$$
For $\epsilon>0$ small enough, the parameter $\tau'$ is small enough that \Cref{lemma:NtaunmWUpperBd} implies
$$|N_{\tau'}(n,m,W^\ast)|\leq(1+e^{-\beta n})\sum_{\Pi\in\sD_{W^\ast}}|\cC^\ast_{W^\ast,\Pi}|\,,$$
where the sets $\cC^\ast_{W^\ast,\Pi}$ are understood to be redefined with $\tau'$.
Let $C>0$ be the constant from \Cref{prop:CNMtypstrucconddistr} and let $\nu>0$ be the minimum of the constants when \Cref{lemma:CompareBallsInCutMetricNlogN} is applied to the graphons $W_i$, $i\in I$. We compute that
\begin{align*}
|\cC(n,m)| &\leq |\cC_\close| + |\cC_\far| \\
&\leq |N_{\tau'}(n,m,W^\ast)|+\sum_{i\in I}|N_\tau(n,m,W_i)|+|\cC_\far| \\
&\leq \big(1+kn^{-\nu n}+e^{-Cn^2}\big)|N_{\tau'}(n,m,W^\ast)| \\
&\leq \big(1+kn^{-\nu n}+e^{-Cn^2}\big)(1+e^{-\beta n})\sum_{\Pi\in\sD_{W^\ast}}|\cC^\ast_{W^\ast,\Pi}|\,.
\end{align*}
From \Cref{lemma:SubCloseStructure} \ref{item:lem72PiGBalanced} we know that every $\Pi=\{\Pi_1,\dots,\Pi_\ell\}\in\sD_{W^\ast}$ satisfies $||V(\Pi_1)|-2\mu n|\leq\delta n$ and $|V(\Pi_i)|\leq8\delta n$ for all $i\geq2$, proving the first statement of \Cref{thm:almostallCB} \rm{\ref{thm:almostallCBsubcritical}}. The statement that almost every $G\in\cC^\ast_{W^\ast,\Pi}$ has $e_G(S_\Pi)=\Omega(n)$ follows from the argument already carried out in the proof of \Cref{lemma:CompareBallsInCutMetricNlogN}, namely the use of cubic claw-free graphs to establish a lower-bound on $|N_\omega(n,m,W^\ast)|$.
\end{proof}

\section{The Conditional \erdosrenyi{} Random Graph}
\label{sec:GnpTypicalStructure}
In this section we prove \Cref{thm:GnpTypicalStructure}. Several definitions and results from \Cref{sec:subcritical} are used throughout. Define the parameter $\rho:=\frac{3-\sqrt{5}}{2}$. If $\cF\subseteq\cG(n)$ is a set of graphs on $n$ vertices and $p\in(0,1)$ then define the quantity
$$Z_p(\cF):=(1-p)^{\binom{n}{2}}\sum_{G\in\cF}\left(\frac{p}{1-p}\right)^{e(G)}\,.$$
For all $\tau>0$, $n\in\N$, and graphons $W\in\cW$, define the neighborhood
$$N_\tau(n,W):=\{G\in\cC(n):\delta_\square(G,W)<\tau\}\,.$$

\begin{proof}[Proof of \Cref{thm:GnpTypicalStructure}]
Throughout the proof, let $p\in(0,1)$ and let $G$ be the \erdosrenyi{} random graph $G(n,p)$ conditioned on being claw-free.

In the range $p\in(0,\rho)$, \Cref{prop:GNPtypstrucconddistr} proves $\delta_\square(G,W_0)<\epsilon$ with probability at least $e^{-Cn^2}$ for a constant $C=C(\epsilon,p)>0$, which proves $G$ has less than $\epsilon n^2$ edges with high probability.

Next assume $p\in\big(\rho,1\big)$. Let $\epsilon\in(0,(p-\rho)/2^{64})$ and $\tau:=\frac{1}{16}(\epsilon/32)^{67}$. Let $Z_p:=Z_p(\cC(n))$, $\gamma:=(1+p)/2$, and let $U_p:=W^\ast_\gamma$ (where $W^\ast_\gamma$ is the function defined in \pref{eqn:deffstarlambdagraphon}), so that $U_p$ is the unique optimal graphon (up to equivalence) in the variational problem \pref{eqn:variationalproblemGNP}. Note also that $U_p$ has edge density $\gamma$. Define the quantities
\[\arraycolsep=1.4pt
\begin{array}{ll}
Z_\close &:= Z_p(N_\tau(n,U_p)) \,, \\[8pt]
Z_\far &:= Z_p(\cC(n)\setminus N_\tau(n,U_p)) \,.
\end{array}\]
By the counting lemma for graphons \cite[Lemma~10.23]{lovasz2012large}, we know that for all $G\in N_\tau(n,U_p)$, $\big|e(G)-\gamma\binom{n}{2}\big|\leq\tau$. Also, letting $m:=e(G)$, we have $G\in N_{8\tau}(n,m,W_{2m/n^2})$. From the proof of \Cref{thm:almostallCB} \rm{\ref{thm:almostallCBsupercritical}}, we see that
$$|N_{8\tau}(n,m,W_{2m/n^2})|\leq(1+e^{-\xi n})|\cB_c(n,m)|$$
for a constant $\xi>0$. If $C=C(\epsilon,p)>0$ is the constant from \Cref{prop:GNPtypstrucconddistr} then
\begin{align*}
Z_p &= Z_\close+Z_\far \\
&\leq (1-p)^{\binom{n}{2}}\sum_{m=(\gamma-{\tau})\binom{n}{2}}^{(\gamma+{\tau})\binom{n}{2}}|N_{8\tau}(n,m,W_{2m/n^2})|\cdot\left(\frac{p}{1-p}\right)^m+e^{-Cn^2} \\
&\leq (1+e^{-\xi n})(1-p)^{\binom{n}{2}}\sum_{m=(\gamma-{\tau})\binom{n}{2}}^{(\gamma+{\tau})\binom{n}{2}}|\cB_c(n,m)|\cdot\left(\frac{p}{1-p}\right)^m+e^{-Cn^2} \\
&\leq (1+e^{-\xi n})\cdot Z_p(\cB_c(n))+e^{-Cn^2} \,,
\end{align*}
which proves $Z_p=(1+o(1))\cdot Z_p(\cB_c(n))$, hence $G$ is co-bipartite with high probability.

Finally we examine the case $p=\rho$. Let $W_0$ denote the all-zero graphon and let $W\in\cX_\ast^p$ be a graphon such that $\delta_\square(W,W_0)\geq8\omega$, which implies $W$ has edge density at least $8\omega$. We may assume $W$ is of the form $W_{\blambda}$ for some $\blambda=(\lambda_0,\lambda_1,\dots)\in\Lambda$ as defined in \pref{eqn:flambdaGraphonDef}. Let $\gamma:=\int W$ be the edge density of $W$. Notice that $\lambda_1\geq\omega$ since if the opposite were to hold, then $\lambda_{i+1}-\lambda_i<\omega$ for all $0\leq i<|\blambda|$, which would imply $t(K_2,W)\leq4\omega$, a contradiction.

Below, we write the sets $\cC^\ast_{W,\Pi}$ and $\cR_{\Pi,t}$ as they were defined in \Cref{sec:subcritical}, i.e. with respect to the number of edges $m$, since these sets appear in a summation over $m$.

For all $(\gamma-{\tau})\binom{n}{2}\leq m\leq(\gamma+{\tau})\binom{n}{2}$, let $W^m\in\cX^\ast_{2m/n^2}$ denote an optimal graphon with edge density $2m/n^2$ such that $\delta_\square(W^m,W)\leq6\tau$. If $G\in N_\tau(n,W)$ then $\big|e(G)-\gamma\binom{n}{2}|\leq{\tau}$ and $G\in N_{8\tau}(n,m,W^m)$, hence by \Cref{lemma:NtaunmWUpperBd} we have
\begin{align*}
|N_\tau(n,W)| &\leq \sum_{m=(\gamma-{\tau})\binom{n}{2}}^{(\gamma+{\tau})\binom{n}{2}}|N_{8\tau}(n,m,W^m)| \leq (1+e^{-\xi n})\sum_{m=(\gamma-{\tau})\binom{n}{2}}^{(\gamma+{\tau})\binom{n}{2}}\,\sum_{\Pi\in\sD_{W^m}}|\cC^\ast_{W^m,\Pi}|\,.
\end{align*}

Let $k:=\floor{\omega n/4}$ and $s:=\ceil{(1-\omega+\alpha)n}$. It is straightforward from \Cref{lemma:SubCloseStructure} \ref{item:lem72PiGBalanced} and the inequality $\lambda_1\geq\omega$ that $s(\Pi)\leq s$ for all $\Pi\in\sD_W$, and that $s+2k\leq n$. Note that every division $\Pi\in\sD_W$ satisfies $V(\Pi)\geq2k$. By definition of $S_\Pi$, every graph $G\in\cC^\ast_{W^m,\Pi}$ satisfies $e_G(s(\Pi))\leq4\eta n^2$. Let $N_1:=N_\omega(n,W_0)$ and $N_2:=N_\tau(n,W)$. For all $U\in\cX^\ast_\gamma$, let $\sD'_U$ be the set of all $\Pi\in\sD_U$ such that $\min\{|V(\Pi')|:\Pi'\in\Pi\}\geq\epsilon n$. By an argument very similar to \Cref{lemma:RPitSameOrderUpToNlogN}, and using the fact that $I_p(W_0)=I_p(W)=I_p(W^m)=0$, we have that for all $0\leq t\leq4\eta n^2$,
\begin{equation}
1\preceq\sum_{m=(\gamma-{\tau})\binom{n}{2}}^{(\gamma+{\tau})\binom{n}{2}}\,\sum_{\Pi\in\sD'_W}|\cR_{\Pi,t}|\left(\frac{p}{1-p}\right)^m\preceq\sum_{m=0}^{\eta n^2}\,\sum_{\Pi\in\sD'_{W_0}}|\cR_{\Pi,t}|\left(\frac{p}{1-p}\right)^m\preceq1\,,\label{eqn:limEpslimNsumDPitrStargammaZp}
\end{equation}
where the relation $\succeq$ is meant in the sense of \pref{eqn:nlognNotionOfIneq}.

Using \pref{eqn:sDprimeUmutlinomialupperbd} and \pref{eqn:limEpslimNsumDPitrStargammaZp}, and using the relation $\preceq$ in the sense of \pref{eqn:nlognNotionOfIneq}, we compute
\begin{align*}
k^{2k}\cdot Z(N_2) &\preceq (1-p)^{\binom{n}{2}}\sum_{m=(\gamma-{\tau})\binom{n}{2}}^{(\gamma+{\tau})\binom{n}{2}}\,\sum_{\Pi\in\sD_{W^m}}|\cC^\ast_{W^m,\Pi}|\cdot k^{2k}\cdot\left(\frac{p}{1-p}\right)^m \\
&\preceq (1-p)^{\binom{n}{2}}\sum_{m=(\gamma-{\tau})\binom{n}{2}}^{(\gamma+{\tau})\binom{n}{2}}\,\sum_{\Pi\in\sD'_{W^m}}\sum_{t=0}^{4\eta n^2}|\cR_{\Pi,t}|\cdot|\cC(s,t)|\cdot k^{2k}\cdot\left(\frac{p}{1-p}\right)^m \\
&\preceq (1-p)^{\binom{n}{2}}\max_{0\leq t\leq4\eta n^2}\left\{|\cC(s,t)|\cdot k^{2k}\sum_{m=(\gamma-{\tau})\binom{n}{2}}^{(\gamma+{\tau})\binom{n}{2}}\,\sum_{\Pi\in\sD'_{W^m}}|\cR_{\Pi,t}|\cdot\left(\frac{p}{1-p}\right)^m\right\} \\
&\preceq (1-p)^{\binom{n}{2}}\sum_{t=0}^{4\eta n^2}|\cC(s,t)|\cdot\kappa(k) \preceq (1-p)^{\binom{n}{2}}\sum_{t=0}^{4\eta n^2}|\cC(n,t)| \preceq Z(N_1) \,.
\end{align*}
Since $k=\floor{\omega n/4}$, it follows that $Z(N_2)\leq n^{-\zeta n}\cdot Z(N_1)$ for a constant $\zeta>0$.

To complete the proof in the case $p=\rho$, let $\cN=\{W_0,W_1,\dots,W_k\}\subseteq\cX^p_\ast$ be a $\tau$-net of $\cX^p_\ast$ in cut metric. Define the quantities
\[\arraycolsep=1.4pt
\begin{array}{ll}
Z'_\close &:= \displaystyle Z_p\left(N_\tau(n,W_0)\cup\bigcup_{i=1}^kN_\tau(n,W_i)\right) \\[18pt]
Z'_\far &:= Z_p(\cC(n))-Z'_\close \,.
\end{array}\]
Let $I:=\{i\in[k]:\delta_\square(W_i\,,\,W^\ast)\geq8\omega\}$ and $\tau':=\tau+8\omega$, so that
$$Z'_\close \leq Z_p(N_{\tau'}(n,W_0))+\sum_{i\in I}Z_p(N_\tau(n,W_i))\,.$$
If $C>0$ is the constant from \Cref{prop:GNPtypstrucconddistr} then
\begin{align*}
Z_p(\cC(n)) &\leq Z'_\close+Z'_\far \\
&\leq Z_p(N_{\tau'}(n,W_0))+\sum_{i\in I}Z_p(N_\tau(n,W_i))+Z'_\far \\
&\leq \big(1+kn^{-\zeta n}+e^{-Cn^2}\big)\cdot Z_p(N_{\tau'}(n,W_0)) \,,
\end{align*}
which directly implies $\delta_\square(G,W_0)<\epsilon$ with high probability, completing the proof.
\end{proof}

\section{Extremal Structure and Stability of Edge Colorings}
\label{sec:edgecolorings}
Let $\varphi$ be a red--green--blue coloring of $E(K_n)$. For all $u\in V(K_n)$ and $U\subseteq V(K_n)$ define
\begin{align*}
N_r(u) &:= \{v\in V(K_n):\varphi(uv)\text{ is red}\}\,, \\
d_r(u) &:= |N_r(u)|\,, \\
e_r(U) &:= |\{e\in\textstyle\binom{U}{2}:\varphi(e)\text{ is red}\}|\,.
\end{align*}
We will also use the notation $N_r(u,U):=N_r(u)\cap U$ and $d_r(u,U):=|N_r(u,U)|$. Define the green and blue counterparts $N_g(u)$, $N_b(u)$, etc. analogously. Let $E_r(\varphi)$, $E_g(\varphi)$, and $E_b(\varphi)$ denote the sets of red, green, and blue edges, respectively, of the coloring $\varphi$, and let $e_r(\varphi)$, $e_g(\varphi)$, and $e_b(\varphi)$ denote the respective cardinalities of those sets. We now prove \Cref{lemma:c4ATMOSTc3}, which generalizes Mantel's theorem to red--green--blue edge colorings: Mantel's theorem is obtained by considering only those $\varphi$ with $e_g(\varphi)=0$.

\begin{proof}[Proof of \Cref{lemma:c4ATMOSTc3}]
Define the sets $L_1:=\{u:d_r(u)<d_b(u)+1\}$, $L_2:=\{u:d_r(u)=d_b(u)+1\}$, and $L_3:=\{u:d_r(u)>d_b(u)+1\}$. The number of red edges is
\begin{equation}
e_r(\varphi)=\frac{1}{2}\left(\sum_{u\in L_1}d_r(u)+\sum_{u\in L_2}d_r(u)+\sum_{u\in L_3}d_r(u)\right)\,.\label{eqn:rededgeenum}
\end{equation}
If $L_3=\emptyset$ then the assertion is immediate since
$$e_r(\varphi)\leq\frac{1}{2}\sum_{u\in L_1\cup L_2}(d_b(u)+1)=e_b(\varphi)+\frac{n}{2}\,,$$
so we may assume $L_3\neq\emptyset$. Fix $u\in L_3$ and fix $v\in N_r(u)$. Since the red neighborhood of every vertex is a blue clique, we have $d_r(v)\leq d_b(u)+1<d_r(u)\leq d_b(v)+1$, implying $v\in L_1$ and
\begin{equation}
d_r(u)+d_r(v)\leq d_b(u)+d_b(v)+2\,.\label[inequality]{ineq:redblueineq}
\end{equation}
We claim there exists an injection $f:L_3\to L_1$ such that $f(u)\in N_r(u)$ for all $u\in L_3$. Notice that this is sufficient to prove the inequality $e_r(\varphi)\leq e_b(\varphi)+\frac{n}{2}$ for the following reasons. For all $u\in L_3$, since $d_r(u)>d_b(u)+1$ and $f(u)\in N_r(u)$, \Cref{ineq:redblueineq} implies
\begin{equation}
d_r(u)+d_r(f(u))\leq d_b(u)+d_b(f(u))+2\,.\label[inequality]{ineq:redbluefuncineq}
\end{equation}
The definition of $f$, \Cref{eqn:rededgeenum}, and \Cref{ineq:redbluefuncineq} then imply
\begin{align}
e_r(\varphi) &= \frac{1}{2}\left(\sum_{u\in L_1\setminus f(L_3)}d_r(u)+\sum_{u\in L_2}d_r(u)+\sum_{u\in L_3}d_r(u)+\sum_{u\in f(L_3)}d_r(u)\right) \nonumber\\
&= \frac{1}{2}\left(\sum_{u\in L_1\setminus f(L_3)}d_r(u)+\sum_{u\in L_2}(d_b(u)+1)+\sum_{u\in L_3}(d_r(u)+d_r(f(u)))\right) \nonumber\\
&\leq \frac{1}{2}\left(\sum_{u\in L_1\setminus f(L_3)}(d_b(u)+1)+\sum_{u\in L_2}(d_b(u)+1)+\sum_{u\in L_3}(d_b(u)+d_b(f(u))+2)\right) \label[inequality]{ineq:redbluecrucialineq}\\
&= \frac{1}{2}\sum_{u\in V(K_n)}(d_b(u)+1) = e_b(\varphi)+\frac{n}{2}\,.\nonumber
\end{align}

It remains to establish that such $f$ exists. Let $w_1,\dots,w_l$ be an arbitrary ordering of the vertices in $L_3$, and let $W_i:=\{w_1,\dots,w_i\}$. We will prove by induction on $i\in\{1,\dots,l\}$ that there exists an injection $g_i:W_i\to L_1$ such that $g_i(u)\in N_r(u)$ for all $u\in W_i$. The claim is clearly true for $i=1$ since we may set $g_1(w_1)$ to be any member of $N_r(w_1)$. Assume $i\geq2$ and assume the claim holds for all smaller values of $i$. Notice that there must exist at least two distinct vertices in the set $U_i:=N_r(w_i)\setminus g_{i-1}(W_{i-1})$. Indeed, we have $N_r(w_i)\subseteq g_{i-1}(W_{i-1})\cup U_i$ and
$$B:=g_{i-1}^{-1}(N_r(w_i)\setminus U_i)\subseteq N_b(w_i)\,,$$
since if the red neighborhoods of two vertices overlap then those two vertices are joined by a blue edge. If $|U_i|\leq1$ then by injectivity of $g_{i-1}$ it holds that
$$d_b(w_i)=|N_b(w_i)|\geq|B|=|N_r(w_i)|-|U_i|\geq d_r(w_i)-1\,,$$
contradicting our assumption that $w_i\in L_3$, proving our claim that $|U_i|\geq2$. Fixing $x\in U_i$, we can now define an injection
$$g_i(w_j):=\begin{cases}g_{i-1}(w_j)&1\leq j\leq i-1\\x&j=i\end{cases}\,,$$
completing the induction step. Hence $f:=g_l$ satisfies the claimed properties, completing the proof of the inequality $e_r(\varphi)\leq e_b(\varphi)+\frac{n}{2}$.
\end{proof}

The next lemma proves that equality holds in \pref{ineq:thmredblueineq} if and only if $\varphi$ has the the extremal structure $\cE(n)$ defined in \Cref{sec:GraphonOpt}.

\begin{lemma}\label{eqn:edgecolorextremalstructure}
Let $\varphi$ be a red--green--blue coloring of $E(K_n)$ such that $\varphi$ has no triangle colored (red, red, red) nor any colored (red, red, green). Then one has $e_r(\varphi)=e_b(\varphi)+\floor{n/2}$ if and only if $\varphi\in\cE(n)$.
\end{lemma}
\begin{proof}
If $\varphi\in\cE(n)$ then clearly \pref{ineq:thmredblueineq} holds with equality. Conversely, suppose equality is attained, and first assume $n$ is even. We claim that $d_r(u)=d_b(u)+1$ for all $u\in V(K_n)$, which immediately implies $\varphi\in\cE(n)$. Let $L_1$, $L_2$, and $L_3$ be as defined in the proof of \Cref{lemma:c4ATMOSTc3}. Since $n$ is even and equality holds in \Cref{ineq:redbluecrucialineq}, the set $L_1\setminus f(L_3)$ must be empty. We claim that $L_3=\emptyset$\,: indeed, the proof of \Cref{lemma:c4ATMOSTc3} shows that at each step $i$ of the induction, the set $U_i$ contains at least two distinct elements; if it were the case that $L_3\neq\emptyset$ then
$$U_l\setminus\{f(w_l)\}=N_r(w_l)\setminus f(W_l)\subseteq L_1\setminus f(L_3)$$
would be nonempty. All vertices $u\in V(K_n)$ thus have $u\in L_2$ and $d_r(u)=d_b(u)+1$.

Suppose equality is attained in \pref{ineq:thmredblueineq} and $n$ is odd. We may assume there is no vertex $u$ for which $d_g(u)=n-1$\,: if there were one then it could be removed while maintaining equality in \pref{ineq:thmredblueineq}, leaving us in the case when $n$ is even. We claim that $L_1\neq\emptyset$\,: if $L_1$ were empty then the set $L_3$ must also be empty (by the argument in the previous paragraph), and all vertices $u\in V(G)$ satisfy $d_r(u)=d_b(u)+1$, implying $n$ is even. We may thus assume $L_1\neq\emptyset$. We claim that for all $u\in L_1$ one has $d_r(u)=d_b(u)$. Indeed, suppose $d_r(u)\leq d_b(u)-1$ for some $u\in L_1$. Invoking \Cref{ineq:redbluefuncineq}, it follows that
\begin{align*}
e_b(\varphi)+\frac{n-1}{2} &= \frac{1}{2}\left(d_b(u)+\sum_{v\in V(K_n)\setminus\{u\}}(d_b(v)+1)\right) \\
&\geq\frac{1}{2}\left((d_r(u)+1)+\sum_{v\in V(K_n)\setminus\{u\}}d_r(v)\right) = e_r(\varphi)+\frac{1}{2}\,,
\end{align*}
contradicting the assumption of equality in \pref{ineq:thmredblueineq}. Fix $u\in L_1$, which we have shown satisfies $d_r(u)=d_b(u)$. The graph $G:=K_n-u$ has an even number of vertices and, for any bijection $\sigma:V(K_{n-1})\to V(G)$, the coloring $\psi:=\varphi\circ\sigma$ of $E(K_{n-1})$ meets \pref{ineq:thmredblueineq} with equality. Hence the analysis for the case when $n$ is even implies $\psi\in\cE(n-1)$, which is to say there is a partition $V(K_{n-1})=H_1\sqcup\cdots\sqcup H_k$ in which $H_i=A_i\sqcup B_i$ and $|A_i|=|B_i|$ for all $i\in[k]$. Notice that the only manner whereby $d_r(u)=d_b(u)$ can be satisfied in $\varphi$ is if there exists $i$ such that $N_r(u)\in\{A_i,B_i\}$ and $\varphi(uv)$ is green for all $v\not\in H_i$, which is precisely the assertion that $\varphi\in\cE(n)$, completing the proof.
\end{proof}

The following lemma is the main technical assertion invoked in the proof of \Cref{lemma:EdgeColStab}.

\begin{lemma}\label{lemma:stabrecursion}
Let $\alpha\in(0,\frac{1}{8})$ and let $n\in\N$ be even. Let $\varphi$ be a red--green--blue coloring of $E(K_n)$ such that there is no triangle colored (red, red, red) nor any colored (red, red, green). Assume $e_r(\varphi)\geq e_b(\varphi)+\frac{n}{2}-\alpha n^2$. Let $\eta\geq2\sqrt{2\alpha}$ and assume there exist at least $\eta n$ vertices $v$ with $d_r(v)>\eta n$. Then there exist disjoint subsets $X,Y,Z\subseteq V(K_n)$ satisfying the following conditions:
\begin{enumerate}[(\emph{\roman*})]
	\item the edges $E(X)\cup E(Y)$ are all colored blue,
	\item $|X|=|Y|$, $|X\cup Y\cup Z|\geq\eta n$, and $|Z|\leq5(2\alpha)^{1/4}n$,
	\item letting $S:=X\cup Y\cup Z$, it holds that $e_r(S,V(K_n)\setminus S)\leq8\sqrt{2\alpha}n^2$,
	\item letting $U:=V(K_n)\setminus S$, it holds that $e_r(U)\geq e_b(U)+\floor{|U|/2}-(\alpha+8\sqrt{2\alpha})n^2$.
\end{enumerate}
\end{lemma}
\begin{proof}
Let $L_1$, $L_2$, and $L_3$ be as defined in the proof of \Cref{lemma:c4ATMOSTc3}. As shown in the proof of \Cref{eqn:edgecolorextremalstructure}, we may assume $L_1\neq\emptyset$\,: if $L_1$ were empty then in fact $\varphi\in\cE(n)$. Note it may be the case that $L_3=\emptyset$. Fix a (possibly empty) injection $f:L_3\to L_1$ such that $f(u)\in N_r(u)$ for all $u\in L_3$, as provided by the proof of \Cref{lemma:c4ATMOSTc3}. Let
\begin{align*}
D_r &:= \frac{1}{2}\left(\sum_{u\in L_1\setminus f(L_3)}d_r(u)+\sum_{u\in L_3}(d_r(u)+d_r(f(u)))\right)\,, \\
D_b &:= \frac{1}{2}\left(\sum_{u\in L_1\setminus f(L_3)}(d_b(u)+1)+\sum_{u\in L_3}(d_b(u)+d_b(f(u))+2)\right)\,,
\end{align*}
and define the sets
\begin{align*}
M_1 &:= \big\{u\in L_1\setminus f(L_3):d_r(u)\geq d_b(u)+1-\sqrt{2\alpha}n\big\}\,, \\
M_3 &:= \big\{u\in L_3: d_r(u)+d_r(f(u))\geq d_b(u)+d_b(f(u))+2-\sqrt{2\alpha}n\big\}\,.
\end{align*}
By our assumption that $e_r(\varphi)\geq e_b(\varphi)+\frac{n}{2}-\alpha n^2$ and by \pref{ineq:redbluecrucialineq} we have $D_b-\alpha n^2\leq D_r\leq D_b$, which implies
\begin{equation}
|(L_1\setminus f(L_3))\setminus M_1|+|L_3\setminus M_3|\leq\sqrt{2\alpha}n\,.\label{eqn:L1fL3M1L3M3sqrt2alphan}
\end{equation}
For all $u\in M_3$, using the fact that $f(u)\in L_1$, we have
\begin{equation}
\begin{aligned}
0\leq d_b(f(u))+1-d_r(u) &\leq \sqrt{2\alpha}n\,, \\
0\leq d_b(u)+1-d_r(f(u)) &\leq \sqrt{2\alpha}n\,,\label{ineq:M3twoineqs}
\end{aligned}
\end{equation}
since otherwise we would contradict the definition of $M_3$. Let $U:=M_1\cup f(M_3)\cup L_2\cup M_3$ and note that \pref{eqn:L1fL3M1L3M3sqrt2alphan} implies $|U|\geq(1-2\sqrt{2\alpha})n$. Make the following definitions
\begin{align*}
x &:= \argmax\{d_r(u):u\in U\}\,, \\
y &:= \argmax\{d_r(u):u\in N_r(x)\cap U\}\,,
\end{align*}
where $\argmax$ selects any vertex $u$ achieving the maximum. Define also the vertex subsets $X:=N_r(y)$, $Y':=N_r(x)$, and $S:=X\cup Y'$. Note that by the hypotheses of the lemma, $d_r(x)>2\sqrt{2\alpha}n$, so $N_r(x)\cap U$ is nonempty and $y$ is well-defined.

Letting $\beta:=8\sqrt{2\alpha}$, we claim that
\begin{equation}
e_r(S,\overline{S})\leq\beta n^2\,,\label[inequality]{ineq:rededgesStoVminusS}
\end{equation}
where $\overline{S}:=V(G)\setminus S$. Suppose that instead the opposite of \pref{ineq:rededgesStoVminusS} held. Then at least one of the two inequalities $e_r(X,\overline{S})>\beta n^2/2$ or $e_r(Y',\overline{S})>\beta n^2/2$ must hold. We will show both cases lead to a contradiction.

\begin{enumerate}[wide, labelwidth=!, labelindent=0pt, leftmargin=0pt]
	\item[\underline{\textbf{Case 1.}}] First assume $e_r(X,\overline{S})>\beta n^2/2$. Let $W:=\bigcup E_r(X,\overline{S})$ be the vertex set covered by the edges $E_r(X,\overline{S})$, and define the bipartite graph $J:=(W,E_r(X,\overline{S}))$. Since $J$ has at least $\beta n^2/2$ edges, it contains a subgraph $J'$ with minimum degree $>\beta n/2$ (see e.g. \cite[p.~6]{diestel2017graph}). In particular $|X\cap W|>\beta n/2$. Notice that $2\sqrt{2\alpha}n<\beta n/2$, hence $U\cap X\cap W$ contains at least one vertex, say $w$. Notice that $y\in N_r(w)$ and $d_r(w,W)\geq\beta n/2$, so the fact that $N_r(w)$ is a blue clique implies $d_b(y)+1\geq d_r(x)+\beta n/2$. We consider the following three mutually exclusive cases.

	\item[\underline{\textbf{Case 1a.}}] If $y\in M_1\cup L_2$ then we compute
	$$d_r(y)\geq d_b(y)+1-\sqrt{2\alpha}n\geq d_r(x)+\frac{\beta n}{2}-\sqrt{2\alpha}n>d_r(x)\,,$$
	contradicting that $x$ was chosen with maximum red degree.
	\item[\underline{\textbf{Case 1b.}}] If $y\in f(M_3)$ then there exists $u\in M_3$ such that $y=f(u)$. The inequalities in \pref{ineq:M3twoineqs} then imply $d_r(u)\geq d_b(y)+1-\sqrt{2\alpha}n$, and we deduce similarly to Case 1a that $d_r(u)>d_r(x)$, a contradiction.
	\item[\underline{\textbf{Case 1c.}}] If $y\in M_3$ then $d_r(y)>d_b(y)+1\geq d_r(x)$ (since $y\in N_r(x)$), a contradiction.
\end{enumerate}

\begin{enumerate}[wide, labelwidth=!, labelindent=0pt, leftmargin=0pt]
	\item[\underline{\textbf{Case 2.}}] Now assume $e_r(Y',\overline{S})>\beta n^2/2$. First note $y\not\in M_3$\,: if $y\in M_3$ then by the proof of \Cref{lemma:c4ATMOSTc3}, there exists a choice of the mapping $f$ such that $x=f(y)$, which implies $d_r(y)>d_r(x)$, instantiating a contradiction. As in Case 1, there exists a vertex $w\in Y'$ such that $d_r(w,\overline{S})>\beta n/2$. Since $x\in N_r(w)$ and $N_r(w)$ is a blue clique, it follows that $d_b(x)+1>d_r(y)+\beta n/2$, and we will use this fact several times in what follows. We consider the following cases.

	\item[\underline{\textbf{Case 2a.}}] Assume $x\in M_1\cup L_2$. First, if $y\in M_1\cup L_2$ then we compute
	\begin{align*}
	d_r(x) &\geq d_b(x)+1-\sqrt{2\alpha}n>d_r(y)+\frac{\beta n}{2}-\sqrt{2\alpha}n \\
	&\geq d_b(y)+1+\frac{\beta n}{2}-2\sqrt{2\alpha}n\geq d_r(x)+\frac{\beta n}{2}-2\sqrt{2\alpha}n>d_r(x)\,,
	\end{align*}
	where we used $2\sqrt{2\alpha}n<\beta n/2$, but this sequence of inequalities is impossible. Hence we may assume $y\in f(M_3)$ (since it was noted $y\not\in M_3$). Then there exists $z\in N_r(y,M_3)$ such that $y=f(z)$, hence \pref{ineq:M3twoineqs} and the fact that $x\in M_1\cup L_2$ imply
	\begin{align*}
	d_r(z) &\geq d_b(y)+1-\sqrt{2\alpha}n\geq d_r(x)-\sqrt{2\alpha}n \geq d_b(x)+1-2\sqrt{2\alpha}n \\
	&> d_r(y)+\frac{\beta n}{2}-2\sqrt{2\alpha}n=d_r(y)+2\sqrt{2\alpha}n\,,
	\end{align*}
	thus again by \pref{ineq:M3twoineqs}, we have that
	\begin{equation}
	d_b(z)+1\leq d_r(y)+\sqrt{2\alpha}n<d_r(z)-\sqrt{2\alpha}n\,.\label[inequality]{ineq:stabineqdbclesdrc}
	\end{equation}
	Since $z\in L_3$, it holds that $N_r(z)\subseteq L_1$ (as noted in the proof of \Cref{lemma:c4ATMOSTc3}). We let $N_b[z]:=N_b(z)\cup\{z\}$ and claim
	\begin{equation}
	K:=N_r(z)\setminus f(N_b[z]\cap L_3)\subseteq L_1\setminus f(L_3)\,.\label{eqn:NrzsetminusfNbzL3}
	\end{equation}
	Indeed, if there were a vertex $z'\in K$ not contained in $L_1\setminus f(L_3)$, then $z'\in f(L_3)$; hence there exists $x'\in L_3$ such that $z'=f(x')\in N_r(x')$. Now since $x',\,z\in N_r(z')$ the edge $(x',z)$ must be blue, but this contradicts the definition of $K$, namely that $f(x')\not\in K$.
	
	Using \pref{eqn:NrzsetminusfNbzL3}, we compute
	\begin{align*}
	|L_1\setminus f(L_3)| &\geq |N_r(z)\setminus f(N_b[z]\cap L_3)| \\
	&\geq |N_r(z)|-|f(N_b[z]\cap L_3)| \\
	&\geq d_r(z)-|N_b[z]| \\
	&= d_r(z)-(d_b(z)+1) \\
	&> \sqrt{2\alpha}n\,,
	\end{align*}
	where the third inequality holds by injectivity of $f$ and the last inequality uses \pref{ineq:stabineqdbclesdrc}. Note that for all $y'\in N_r(z)$ it holds that $d_r(y')\leq d_b(z)+1$ and $d_r(z)\leq d_b(y')+1$, so together with \pref{ineq:stabineqdbclesdrc} we deduce $d_b(y')+1\geq d_r(y')+\sqrt{2\alpha}n$. Using that $|K|>\sqrt{2\alpha}n$, we compute
	\begin{align*}
	D_b-D_r &\geq \frac{1}{2}\left(\sum_{u\in L_1\setminus f(L_3)}(d_b(u)+1)-\sum_{u\in L_1\setminus f(L_3)}d_r(u)\right) \\
	&\geq\frac{1}{2}\sum_{u\in K}(d_b(u)+1-d_r(u)) > \alpha n^2\,,
	\end{align*}
	contradicting our assumption that $D_r\geq D_b-\alpha n^2$.
	\item[\underline{\textbf{Case 2b.}}] If $x\in f(M_3)\cup M_3$ then by \pref{ineq:M3twoineqs} there exists $u\in N_r(x)$ such that
	$$d_r(u)\geq d_b(x)+1-\sqrt{2\alpha}n\geq d_r(y)+\frac{\beta n}{2}-\sqrt{2\alpha}n>d_r(y)\,,$$
	contradicting that $y$ has maximum degree among the vertices in $N_r(x)$.
\end{enumerate}

Hence we have proven \pref{ineq:rededgesStoVminusS}. If $|X|=|Y'|$ then set $Y:=Y'$ and $Z:=\emptyset$. If $|X|<|Y'|$ then we will remove vertices from $Y'$ to obtain a set of equal size as $X$, which is done as follows. Let $k:=|X|$ and $l:=|Y'|-|X|$. Using \pref{ineq:rededgesStoVminusS} and applying \Cref{lemma:c4ATMOSTc3} to the coloring $\varphi$ restricted to $\overline{S}$, we compute
\begin{align*}
e_r(\varphi) &= e_r(S,\overline{S}) + e_r(X,Y') + e_r(\overline{S}) \leq \beta n^2+k(k+l)+e_b(\overline{S})+\frac{n-2k-l}{2}\,, \\
e_b(\varphi) &\geq e_b(X) + e_b(Y') + e_b(\overline{S}) \\
&\geq \binom{k}{2}+\binom{k+l}{2}+e_b(\overline{S})=k(k+l)+\frac{l(l-1)}{2}-k+e_b(\overline{S})\,,
\end{align*}
which, together with the inequality $e_r(\varphi)\geq e_b(\varphi)+n/2-\alpha n^2$, quickly implies the bound $l\leq\sqrt{2(\alpha+\beta)}n\leq5(2\alpha)^{1/4}n$, as claimed. Let $Z\subseteq Y'$ be any subset with $|Z|=l$, and define $Y:=Y'\setminus Z$.

By construction, the edges $E(X)\cup E(Y)$ are all colored blue, $|X|=|Y|$, and $|S|>\eta n$ (since $d_r(x)>\eta n$). Defining $U:=V(K_n)\setminus S$ and applying \Cref{lemma:c4ATMOSTc3} to $S$, we compute
\begin{align*}
e_r(U) &\geq e_r(\varphi)-e_r(S)-\beta n^2 \geq \left(e_b(\varphi)+\frac{n}{2}-\alpha n^2\right)-\left(e_b(S)+\floor*{\frac{|S|}{2}}\right)-\beta n^2 \\ &\geq e_b(U)+\floor*{\frac{n-|S|}{2}}-(\alpha+\beta)n^2 = e_b(U)+\floor*{\frac{|U|}{2}}-(\alpha+\beta)n^2\,,
\end{align*}
where the first inequality invokes \pref{ineq:rededgesStoVminusS}. This completes the proof.
\end{proof}

\begin{proof}[Proof of \Cref{lemma:EdgeColStab}]
Define the function $f_0:[0,1]\to\R$ by $f_0(x)=x$. For all $k\in\N$ let $f_k:[0,1]\to\R$ and $F_k:[0,1]\to\R$ be defined by
\begin{align*}
f_k(x) &= \frac{64}{\epsilon^2}\left(f_{k-1}(x)+8\sqrt{2f_{k-1}(x)}\right)\,, \\
F_k(x) &= \sum_{l=0}^{k-1}\left(8\sqrt{2f_l(x)}+5(2f_l(x))^{1/4}\right)\,.
\end{align*}
Clearly $f_k(0)=0$ and $f_k$ is strictly increasing for all $k\in\N$ (since it is a sum and composition of such functions). Hence, letting $N:=\ceil{16/\epsilon}$, there exists $\delta_0\in(0,\frac{1}{8})$ such that $f_N(\delta_0)<\epsilon^2/2^{11}$. Similarly $F_k(0)=0$ and $F_k$ is strictly increasing for all $k\in\N$, so there exists $\delta_1\in(0,\frac{1}{8})$ such that $F_N(\delta_1)<\epsilon/4$.

We claim $\delta:=\min\{\delta_0,\delta_1,\epsilon^2/64\}$ suffices. So assume $e_r(\varphi)\geq e_b(\varphi)+\floor{n/2}-\delta n^2$, and we claim $d(\varphi,\cF(n))\leq\epsilon n^2$. If $n<8/\epsilon$ then the assertion is immediate since $\delta n^2<1$ and \Cref{eqn:edgecolorextremalstructure} proves $\varphi\in\cF(n)$. Henceforth we assume $n\geq8/\epsilon$. We may further assume $n$ is even: if $n$ is odd then we can select any vertex $v$ and edit each edge incident to $v$ to be green; since $n\leq\epsilon n^2/8$, this operation expends at most $\epsilon n^2/8$ Hamming distance; we may proceed with the coloring $\varphi$ restricted to $K_n-v$, and show this coloring is within $7\epsilon n^2/8$ Hamming distance of $\cF(n)$.

We apply \Cref{lemma:stabrecursion} inductively in the following manner. In each application of \Cref{lemma:stabrecursion} we set $\eta:=\epsilon/16$, noting that by our choice of $\delta$, the inequality $\eta\geq2\sqrt{2\alpha}$ will always hold in the following steps. Let $\alpha_0:=\delta$, $\varphi_0:=\varphi$, and $R_0:=V(K_n)$. For all $l\in\N$, if $|R_{l-1}|\leq\epsilon n/8$ or if $R_{l-1}$ has less than $\epsilon n/16$ vertices $u$ with $d_r(u)>\epsilon n/16$ (with respect to the coloring $\varphi_{l-1}$), then terminate the induction. Otherwise $|R_{l-1}|>\epsilon n/8$ and fix an arbitrary bijection $\psi_l:V(K_{n_l})\to R_{l-1}$, where $n_l:=|R_{l-1}|$. Define the red--green--blue coloring $\varphi_l:=\varphi\circ\psi_l$ of $E(K_{n_l})$. We then apply \Cref{lemma:stabrecursion} to the coloring $\varphi_l$ with $\alpha=f_l(\delta)$, yielding vertex subsets $X_l,Y_l,Z_l\subseteq V(K_{n_l})$ satisfying the conclusions of the lemma. Set $H_l:=\psi_l(X_l\cup Y_l)$, $S_l:=H_l\cup\psi_l(Z_l)$, and
$$R_l:=V(K_n)\setminus\bigcup_{j=1}^lS_j\,.$$
Notice that $|S_l|>\epsilon n/16$ for all $l$ (by the minimum red degree condition), so the induction terminates after some finite number $k\in\N$ of steps, where $k\leq N$.

Define the edge sets
\begin{equation*}
\begin{array}{lll}
\displaystyle A:=\bigcup_{l=1}^k(E(X_l)\cup E(Y_l))\,, & B:=E_b(\varphi)\setminus A\,, & \displaystyle C:=\bigcup_{l=1}^kE(X_l,Y_l)\,, \\[20pt]
\displaystyle D:=\bigcup_{1\leq i<j\leq k}E_r(S_i,S_j)\,, & \displaystyle E:=\bigcup_{l=1}^kE_r(Z_l,V(K_n))\,, & F:=E_r(R_k,V(K_n))\,,
\end{array}
\end{equation*}
where the notations $E_r$ and $E_b$ in those definitions refer to the coloring $\varphi$. Let $\psi$ be the coloring of $E(K_n)$ such that $\psi(e)$ is blue if $e\in A$, red if $e\in C$, and green otherwise. Notice that
\begin{equation}
\{e\in E(K_n):\varphi(e)\neq\psi(e)\}\subseteq B\cup(C\setminus E_r(\varphi))\cup D\cup E\cup F\,,\label{eqn:finalsetsinclusion}
\end{equation}
so it suffices to show the set on the right hand side of \pref{eqn:finalsetsinclusion} has size $\leq7\epsilon n^2/8$. By choice of $\delta$ and the bounds given by \Cref{lemma:stabrecursion} it holds that $|D\cup E|\leq\epsilon n^2/4$. Indeed, for all $l\in\N$ we edited the (at most $f_l(\delta)n^2$) red edges $E_r(S_l,R_l)$ of $\varphi$ to be green; and edited all edges incident to $Z_l$ to be green (there were at most $5(2f_l(\delta))^{1/4}n^2$ such edges). By definition of the condition for terminating the induction, there are at most $\epsilon n^2/8$ red edges incident to $R_k$, hence $|F|\leq\epsilon n^2/8$. To bound $|B|+|C\setminus E_r(\varphi)|$, we first compute
\begin{align*}
e_r(\varphi) &\leq |D\cup E\cup F|+\sum_{l=1}^ke_r(X_l,Y_l)\leq\frac{3\epsilon}{8}n^2+\sum_{l=1}^ke_r(X_l,Y_l) \\
e_b(\varphi) &\geq |A|+|B|\,.
\end{align*}
Noticing that $|A|+\floor{n/2}\geq\sum_{l=1}^ke_r(X_l,Y_l)+|C\setminus E_r(\varphi)|$, and invoking $e_r(\varphi)\geq e_b(\varphi)+\floor{n/2}-\delta n^2$, we compute that
$$\left(\frac{3\epsilon}{8}+\delta\right)n^2\geq|A|+\floor*{\frac{n}{2}}-\sum_{l=1}^ke_r(X_l,Y_l)+|B|\geq|B|+|C\setminus E_r(\varphi)|\,,$$
so using that $\delta<\epsilon/8$, we obtain $|B|+|C\setminus E_r(\varphi)|\leq\epsilon n^2/2$, completing the proof.
\end{proof}

\section*{Acknowledgments}
The authors thank Eoin Hurley for insightful discussions. The first author is supported in part by NSF grant DMS-2348743. The second author is supported by an NSF Graduate Research Fellowship.

\printbibliography

@article{mantel1907vraagstuk,
  title={Vraagstuk xxviii},
  author={Mantel, Willem},
  journal={Wiskundige Opgaven met de Oplossingen},
  volume={10},
  number={2},
  pages={60--61},
  year={1907}
}

@article{erdos1961minimal,
  title={On the minimal number of vertices representing the edges of a graph},
  author={Erd{\H{o}}s, Paul and Gallai, Tibor},
  journal={Publ. Math. Inst. Hungar. Acad. Sci},
  volume={6},
  number={18},
  pages={1--203},
  year={1961}
}

@article{beineke1968derived,
  title={Derived graphs and digraphs},
  author={Beineke, Lowell W},
  journal={Beitr{\"a}ge zur graphentheorie},
  pages={17--33},
  year={1968},
  publisher={Teubner Leipzig}
}

@article{kleitman1975asymptotic,
  title={Asymptotic enumeration of partial orders on a finite set},
  author={Kleitman, Daniel J and Rothschild, Bruce L},
  journal={Transactions of the American Mathematical Society},
  volume={205},
  pages={205--220},
  year={1975}
}

@article{erdos1976enumeration,
  title={Asymptotic enumeration of $K_n$-free graphs},
  author={Erd\H{o}s, Paul and Kleitman, Daniel J. and Rothschild, Bruce L.},
  journal={Colloquio Internazionale sulle Teorie Combinatorie (Rome, 1973)},
  volume={17},
  pages={19--27},
  year={1973}
}

@article{macwilliams1977theory,
  title={The theory of error-correcting codes},
  author={MacWilliams, FJ},
  journal={Elsevier Science Publishers},
  volume={2},
  pages={39--47},
  year={1977}
}

@article{minty1980maximal,
  title={On maximal independent sets of vertices in claw-free graphs},
  author={Minty, George J},
  journal={Journal of Combinatorial Theory, Series B},
  volume={28},
  number={3},
  pages={284--304},
  year={1980},
  publisher={Elsevier}
}

@book{bollobas1985random,
  title={Random Graphs},
  author={Bollob{\'a}s, B{\'e}la},
  year={1985},
  publisher={Academic Press}
}

@article{erdos1986asymptotic,
  title={The asymptotic number of graphs not containing a fixed subgraph and a problem for hypergraphs having no exponent},
  author={Erd\H{o}s, Paul and Frankl, Peter and R\"{o}dl, Vojt\v{e}ch},
  journal={Graphs and Combinatorics},
  volume={2},
  number={1},
  pages={113--121},
  year={1986},
  publisher={Springer}
}

@inproceedings{janson1990exponential,
  title={An exponential bound for the probability of nonexistence of a specified subgraph in a random graph},
  author={Janson, Svante and Luczak, Tomasz and Rucinski, Andrzej},
  booktitle={Random graphs},
  volume={87},
  pages={73--87},
  year={1990}
}

@article{promel1991excluding,
  title={Excluding induced subgraphs: quadrilaterals},
  author={Pr{\"o}mel, Hans J{\"u}rgen and Steger, Angelika},
  journal={Random Structures \& Algorithms},
  volume={2},
  number={1},
  pages={55--71},
  year={1991},
  publisher={Wiley Online Library}
}

@article{promel1992asymptotic,
  title={The asymptotic number of graphs not containing a fixed color-critical subgraph},
  author={Pr{\"o}mel, Hans J{\"u}rgen and Steger, Angelika},
  journal={Combinatorica},
  volume={12},
  pages={463--473},
  year={1992},
  publisher={Springer}
}

@article{promel1992excluding,
  title={Excluding induced subgraphs III: a general asymptotic},
  author={Pr{\"o}mel, Hans J{\"u}rgen and Steger, Angelika},
  journal={Random Structures \& Algorithms},
  volume={3},
  number={1},
  pages={19--31},
  year={1992},
  publisher={Wiley Online Library}
}

@article{alekseev1993entropy,
  title={On the entropy values of hereditary classes of graphs},
  author={Alekseev, Vladimir E},
  year={1993},
  journal={Discrete Mathematics and Applications},
  volume={3},
  number={2},
  parges={191--200},
  publisher={Walter de Gruyter, Berlin/New York Berlin, New York}
}

@article{promel1993excluding,
  title={Excluding induced subgraphs II: extremal graphs},
  author={Pr{\"o}mel, Hans J{\"u}rgen and Steger, Angelika},
  journal={Discrete Applied Mathematics},
  volume={44},
  number={1-3},
  pages={283--294},
  year={1993},
  publisher={Elsevier}
}

@article{promel1996asymptotic,
  title={On the asymptotic structure of sparse triangle free graphs},
  author={Pr{\"o}mel, Hans J{\"u}rgen and Steger, Angelika},
  journal={Journal of Graph Theory},
  volume={21},
  number={2},
  pages={137--151},
  year={1996},
  publisher={Wiley Online Library}
}

@incollection{bollobas1997hereditary,
  title={Hereditary and monotone properties of graphs},
  author={Bollob{\'a}s, B{\'e}la and Thomason, Andrew},
  booktitle={The Mathematics of Paul Erd{\"o}s II},
  pages={70--78},
  year={1997},
  publisher={Springer}
}

@article{faudree1997claw,
  title={Claw-free graphs—a survey},
  author={Faudree, Ralph and Flandrin, Evelyne and Ryj{\'a}{\v{c}}ek, Zden{\v{e}}k},
  journal={Discrete Mathematics},
  volume={164},
  number={1-3},
  pages={87--147},
  year={1997},
  publisher={Elsevier}
}

@article{jaeger1998ehrenfest,
  title={The Ehrenfest classification of phase transitions: introduction and evolution},
  author={Jaeger, Gregg},
  journal={Archive for history of exact sciences},
  volume={53},
  pages={51--81},
  year={1998},
  publisher={Springer}
}

@article{alon2000efficient,
  title={Efficient testing of large graphs},
  author={Alon, Noga and Fischer, Eldar and Krivelevich, Michael and Szegedy, Mario},
  journal={Combinatorica},
  volume={20},
  number={4},
  pages={451--476},
  year={2000},
  publisher={Springer}
}

@article{bollobas2000structure,
  title={The structure of hereditary properties and colourings of random graphs},
  author={Bollob{\'a}s, B{\'e}la and Thomason, Andrew},
  journal={Combinatorica},
  volume={20},
  number={2},
  pages={173--202},
  year={2000},
  publisher={Springer}
}

@article{luczak2000triangle,
  title={On triangle-free random graphs},
  author={{\L}uczak, Tomasz},
  journal={Random Structures \& Algorithms},
  volume={16},
  number={3},
  pages={260--276},
  year={2000},
  publisher={Wiley Online Library}
}

@article{mckay2003asymptotic,
  title={The asymptotic number of claw-free cubic graphs},
  author={McKay, Brendan D. and Palmer, Edgar M. and Read, Ronald C. and Robinson, Robert W.},
  journal={Discrete Mathematics},
  volume={272},
  number={1},
  pages={107--118},
  year={2003},
  publisher={Elsevier}
}

@article{osthus2003densities,
  title={For which densities are random triangle-free graphs almost surely bipartite?},
  author={Osthus, Deryk and Pr{\"o}mel, Hans J{\"u}rgen and Taraz, Anusch},
  journal={Combinatorica},
  volume={23},
  pages={105--150},
  year={2003},
  publisher={Springer}
}

@article{csiszar2004information,
  title={Information Theory and Statistics: A Tutorial},
  author={Csisz{\'a}r, Imre and Shields, Paul C.},
  journal={Foundations and Trends{\textregistered} in Communications and Information Theory},
  volume={1},
  number={4},
  pages={417--528},
  year={2004},
  publisher={Now Publishers, Inc.}
}

@article{chudnovsky2005structure,
  title={The structure of claw-free graphs},
  author={Chudnovsky, Maria and Seymour, Paul},
  journal={Surveys in Combinatorics},
  pages={153--172},
  year={2005}
}

@article{lovasz2006limits,
  title={Limits of dense graph sequences},
  author={Lov{\'a}sz, L{\'a}szl{\'o} and Szegedy, Bal{\'a}zs},
  journal={Journal of Combinatorial Theory, Series B},
  volume={96},
  number={6},
  pages={933--957},
  year={2006},
  publisher={Elsevier}
}

@article{chudnovsky2007roots,
  title={The roots of the independence polynomial of a clawfree graph},
  author={Chudnovsky, Maria and Seymour, Paul},
  journal={Journal of Combinatorial Theory, Series B},
  volume={97},
  number={3},
  pages={350--357},
  year={2007},
  publisher={Elsevier}
}

@article{alon2008characterization,
  title={A characterization of the (natural) graph properties testable with one-sided error},
  author={Alon, Noga and Shapira, Asaf},
  journal={SIAM Journal on Computing},
  volume={37},
  number={6},
  pages={1703--1727},
  year={2008},
  publisher={SIAM}
}

@article{borgs2008convergent,
  title={Convergent sequences of dense graphs I: Subgraph frequencies, metric properties and testing},
  author={Borgs, Christian and Chayes, Jennifer T and Lov{\'a}sz, L{\'a}szl{\'o} and S{\'o}s, Vera T and Vesztergombi, Katalin},
  journal={Advances in Mathematics},
  volume={219},
  number={6},
  pages={1801--1851},
  year={2008},
  publisher={Elsevier}
}

@article{chudnovsky2008clawV,
  title={Claw-free graphs. V. Global structure},
  author={Chudnovsky, Maria and Seymour, Paul},
  journal={Journal of Combinatorial Theory, Series B},
  volume={98},
  number={6},
  pages={1373--1410},
  year={2008},
  publisher={Elsevier}
}

@article{balogh2009typical,
  title={The typical structure of graphs without given excluded subgraphs},
  author={Balogh, J{\'o}zsef and Bollob{\'a}s, B{\'e}la and Simonovits, Mikl{\'o}s},
  journal={Random Structures \& Algorithms},
  volume={34},
  number={3},
  pages={305--318},
  year={2009},
  publisher={Wiley Online Library}
}

@article{borgs2010moments,
  title={Moments of two-variable functions and the uniqueness of graph limits},
  author={Borgs, Christian and Chayes, Jennifer and Lov{\'a}sz, L{\'a}szl{\'o}},
  journal={Geometric and functional analysis},
  volume={19},
  pages={1597--1619},
  year={2010},
  publisher={Springer}
}

@article{janson2010graphons,
  title={Graphons, cut norm and distance, couplings and rearrangements},
  author={Janson, Svante},
  journal={New York Journal of Mathematics},
  year={2013}
}

@article{marchant2010extremal,
  title={Extremal graphs and multigraphs with two weighted colours},
  author={Marchant, Edward and Thomason, Andrew},
  journal={Fete of combinatorics and computer science},
  pages={239--286},
  year={2010},
  publisher={Springer}
}

@article{alon2011structure,
  title={The structure of almost all graphs in a hereditary property},
  author={Alon, Noga and Balogh, J{\'o}zsef and Bollob{\'a}s, B{\'e}la and Morris, Robert},
  journal={Journal of Combinatorial Theory, Series B},
  volume={101},
  number={2},
  pages={85--110},
  year={2011},
  publisher={Elsevier}
}

@article{balogh2011excluding,
  title={Excluding induced subgraphs: critical graphs},
  author={Balogh, J{\'o}zsef and Butterfield, Jane},
  journal={Random Structures \& Algorithms},
  volume={38},
  number={1-2},
  pages={100--120},
  year={2011},
  publisher={Wiley Online Library}
}

@article{chatterjee2011large,
  title={The large deviation principle for the {E}rd{\H{o}}s-{R}{\'e}nyi random graph},
  author={Chatterjee, Sourav and Varadhan, S. R. Srinivasa},
  journal={European Journal of Combinatorics},
  volume={32},
  number={7},
  pages={1000--1017},
  year={2011},
  publisher={Elsevier}
}

@article{marchant2011structure,
  title={The structure of hereditary properties and 2-coloured multigraphs},
  author={Marchant, Edward and Thomason, Andrew},
  journal={Combinatorica},
  volume={31},
  pages={85--93},
  year={2011},
  publisher={Springer}
}

@article{bottcher2012perfect,
  title={Perfect graphs of fixed density: Counting and homogeneous sets},
  author={B{\"o}ttcher, Julia and Taraz, Anusch and W{\"u}rfl, Andreas},
  journal={Combinatorics, Probability and Computing},
  volume={21},
  number={5},
  pages={661--682},
  year={2012},
  publisher={Cambridge University Press}
}

@book{lovasz2012large,
  title={Large Networks and Graph Limits},
  author={Lov{\'a}sz, L{\'a}szl{\'o}},
  volume={60},
  year={2012},
  publisher={American Mathematical Soc.}
}

@article{radin2013phase,
  title={Phase transitions in a complex network},
  author={Radin, Charles and Sadun, Lorenzo},
  journal={Journal of Physics A: Mathematical and Theoretical},
  volume={46},
  number={30},
  pages={305002},
  year={2013},
  publisher={IOP Publishing}
}

@article{radin2014asymptotics,
  title={The asymptotics of large constrained graphs},
  author={Radin, Charles and Ren, Kui and Sadun, Lorenzo},
  journal={Journal of Physics A: Mathematical and Theoretical},
  volume={47},
  number={17},
  pages={175001},
  year={2014},
  publisher={IOP Publishing}
}

@article{aristoff2015asymptotic,
  title={Asymptotic structure and singularities in constrained directed graphs},
  author={Aristoff, David and Zhu, Lingjiong},
  journal={Stochastic Processes and their Applications},
  volume={125},
  number={11},
  pages={4154--4177},
  year={2015},
  publisher={Elsevier}
}

@article{augeri2020nonlinear,
  title={Nonlinear large deviation bounds with applications to {W}igner matrices and sparse {E}rd{\H{o}}s--{R}{\'e}nyi graphs},
  author={Augeri, Fanny},
  journal={The Annals of probability},
  volume={48},
  number={5},
  pages={2404--2448},
  year={2020},
  publisher={JSTOR}
}

@article{cook2020large,
  title={Large deviations of subgraph counts for sparse {E}rd{\H{o}}s--{R}{\'e}nyi graphs},
  author={Cook, Nicholas and Dembo, Amir},
  journal={Advances in Mathematics},
  volume={373},
  pages={107289},
  year={2020},
  publisher={Elsevier}
}

@article{chatterjee2016nonlinear,
  title={Nonlinear large deviations},
  author={Chatterjee, Sourav and Dembo, Amir},
  journal={Advances in Mathematics},
  volume={299},
  pages={396--450},
  year={2016},
  publisher={Elsevier}
}

@article{eldan2018gaussian,
  title={Gaussian-width gradient complexity, reverse log-{S}obolev inequalities and nonlinear large deviations},
  author={Eldan, Ronen},
  journal={Geometric and Functional Analysis},
  volume={28},
  number={6},
  pages={1548--1596},
  year={2018},
  publisher={Springer}
}

@article{lubetzky2017variational,
  title={On the variational problem for upper tails in sparse random graphs},
  author={Lubetzky, Eyal and Zhao, Yufei},
  journal={Random Structures \& Algorithms},
  volume={50},
  number={3},
  pages={420--436},
  year={2017},
  publisher={Wiley Online Library}
}

@article{lubetzky2015replica,
  title={On replica symmetry of large deviations in random graphs},
  author={Lubetzky, Eyal and Zhao, Yufei},
  journal={Random Structures \& Algorithms},
  volume={47},
  number={1},
  pages={109--146},
  year={2015},
  publisher={Wiley Online Library}
}

@article{radin2015singularities,
  title={Singularities in the entropy of asymptotically large simple graphs},
  author={Radin, Charles and Sadun, Lorenzo},
  journal={Journal of Statistical Physics},
  volume={158},
  pages={853--865},
  year={2015},
  publisher={Springer}
}

@article{riordan2015janson,
  title={The Janson inequalities for general up-sets},
  author={Riordan, Oliver and Warnke, Lutz},
  journal={Random Structures \& Algorithms},
  volume={46},
  number={2},
  pages={391--395},
  year={2015},
  publisher={Wiley Online Library}
}

@article{balogh2016typical,
  title={The typical structure of sparse {$K_{r+1}$}-free graphs},
  author={Balogh, J{\'o}zsef and Morris, Robert and Samotij, Wojciech and Warnke, Lutz},
  journal={Transactions of the American Mathematical Society},
  volume={368},
  number={9},
  pages={6439--6485},
  year={2016}
}

@book{diestel2017graph,
author = {Diestel, Reinhard},
title = {Graph Theory},
year = {2017},
isbn = {3662536218},
publisher = {Springer Publishing Company, Incorporated},
edition = {5th},
}

@article{janson2017string,
  title={On string graph limits and the structure of a typical string graph},
  author={Janson, Svante and Uzzell, Andrew J},
  journal={Journal of Graph Theory},
  volume={84},
  number={4},
  pages={386--407},
  year={2017},
  publisher={Wiley Online Library}
}

@article{kenyon2017multipodal,
  title={Multipodal structure and phase transitions in large constrained graphs},
  author={Kenyon, Richard and Radin, Charles and Ren, Kui and Sadun, Lorenzo},
  journal={Journal of Statistical Physics},
  volume={168},
  pages={233--258},
  year={2017},
  publisher={Springer}
}

@article{kenyon2017asymptotics,
  title={On the asymptotics of constrained exponential random graphs},
  author={Kenyon, Richard and Yin, Mei},
  journal={Journal of Applied Probability},
  volume={54},
  number={1},
  pages={165--180},
  year={2017},
  publisher={Cambridge University Press}
}

@article{patel2017deterministic,
  title={Deterministic polynomial-time approximation algorithms for partition functions and graph polynomials},
  author={Patel, Viresh and Regts, Guus},
  journal={SIAM Journal on Computing},
  volume={46},
  number={6},
  pages={1893--1919},
  year={2017},
  publisher={SIAM}
}

@article{zhao2017lower,
  title={On the lower tail variational problem for random graphs},
  author={Zhao, Yufei},
  journal={Combinatorics, Probability and Computing},
  volume={26},
  number={2},
  pages={301--320},
  year={2017},
  publisher={Cambridge University Press}
}

@article{zhu2017asymptotic,
  title={Asymptotic structure of constrained exponential random graph models},
  author={Zhu, Lingjiong},
  journal={Journal of Statistical Physics},
  volume={166},
  pages={1464--1482},
  year={2017},
  publisher={Springer}
}

@article{hatami2018graph,
  title={Graph properties, graph limits, and entropy},
  author={Hatami, Hamed and Janson, Svante and Szegedy, Bal{\'a}zs},
  journal={Journal of Graph Theory},
  volume={87},
  number={2},
  pages={208--229},
  year={2018},
  publisher={Wiley Online Library}
}

@article{pach2020almost,
  title={Almost all string graphs are intersection graphs of plane convex sets},
  author={Pach, J{\'a}nos and Reed, Bruce and Yuditsky, Yelena},
  journal={Discrete \& Computational Geometry},
  volume={63},
  number={4},
  pages={888--917},
  year={2020},
  publisher={Springer}
}

@mastersthesis{kalvari2021typical,
  title={The typical structure of sparse graphs with no induced copy of a given subgraph},
  author={Kalvari, T and Samotij, W},
  year={2021},
  school={Tel Aviv University}
}

@article{harel2022upper,
  title={Upper tails via high moments and entropic stability},
  author={Harel, Matan and Mousset, Frank and Samotij, Wojciech},
  journal={Duke Mathematical Journal},
  volume={171},
  number={10},
  pages={2089--2192},
  year={2022},
  publisher={Duke University Press}
}

@article{kozma2023lower,
  title={Lower tails via relative entropy},
  author={Kozma, Gady and Samotij, Wojciech},
  journal={The Annals of Probability},
  volume={51},
  number={2},
  pages={665--698},
  year={2023},
  publisher={Institute of Mathematical Statistics}
}

@article{neeman2023typical,
  title={Typical large graphs with given edge and triangle densities},
  author={Neeman, Joe and Radin, Charles and Sadun, Lorenzo},
  journal={Probability Theory and Related Fields},
  pages={1--57},
  year={2023},
  publisher={Springer}
}

@book{zhao2023graph,
  title={Graph Theory and Additive Combinatorics: Exploring Structure and Randomness},
  author={Zhao, Yufei},
  year={2023},
  publisher={Cambridge University Press}
}

@article{morris2024asymmetric,
  title={An asymmetric container lemma and the structure of graphs with no induced $4 $-cycle},
  author={Morris, Robert and Samotij, Wojciech and Saxton, David},
  journal={Journal of the European Mathematical Society},
  volume={26},
  number={5},
  pages={1655--1711},
  year={2024}
}

\appendix
\markboth{APPENDIX}{APPENDIX}

\section{Auxiliary Facts}
\begin{lemma}[\normalfont{\cite[p.~5]{bollobas1985random}}]\label{lemma:BinomialCoeffIneqs}
Let $n,m\geq1$ be integers. The following statements are valid.
\begin{enumerate}[(\emph{\roman*})]
	\item If $j$ is an integer such that $j\leq m\leq n$ then
	$$\binom{n-j}{m-j}{\binom{n}{m}}^{-1}\leq e^{-(1-m/n)j}\,.$$
	\item If $j$ is an integer such that $j\geq-m$ and $m\leq n-\max\{1,j\}$ then
	$$\binom{n}{m+j}\binom{n}{m}^{-1}\leq\left(\frac{n-m}{m}\right)^{j}\leq e^{-\big(1-\frac{n-m}{m}\big)j}\,.$$
	\item If $j,k$ are nonnegative integers such that $m+k\leq n-j$ then
	$$\binom{n-j}{m+k}\binom{n}{m}^{-1}\leq\left(\frac{n-m}{n}\right)^j\left(\frac{n-m}{m}\right)^k\,.$$
\end{enumerate}
\end{lemma}

\begin{lemma}[Convergence of hypergeometric to binomial]\label{lemma:ConvHypergeomToBinomial}
Let $n,m,k,l$ be nonnegative integers such that $k\leq n$ and $l\leq m\leq n+\min\{0,k-l\}$. If $m=\Theta(n)$, $k=O(\log n)$, and $l=O(\log n)$ as $n\to\infty$ then
$$\binom{n-k}{m-l}\binom{n}{m}^{-1}\sim\left(\frac{m}{n}\right)^l\left(1-\frac{m}{n}\right)^{k-l}\,.$$
\end{lemma}

\begin{lemma}\label{lemma:erdosgallai}
If $G$ is a graph on $n$ vertices and $m$ edges then $G$ has a matching of size at least $\floor{\min\{\frac{\sqrt{2}}{4}\sqrt{m},\frac{m}{n}\}}$. Hence if $m=\alpha n^2$ then $G$ has a matching of size at least $\frac{1}{8}\alpha n$.
\end{lemma}
\begin{proof}
If $0\leq m\leq\max\{n,7\}$ then the assertion is obvious, so we assume $m>\max\{n,7\}$ (which implies $\min\{\frac{\sqrt{2}}{4}\sqrt{m},\frac{m}{n}\}>1$). Let $\nu(G)$ denote the size of a maximum matching in $G$. \citeauthor{erdos1961minimal} proved in \cite{erdos1961minimal} that every graph $H$ on $k$ vertices with $\nu:=\nu(H)$ satisfies
$$e(H)\leq\max\left\{\binom{2\nu+1}{2},(k-\nu)\nu+\binom{\nu}{2}\right\}\,.$$
Let $\mu:=\min\{\frac{\sqrt{2}}{4}\sqrt{m},\frac{m}{n}\}$. Since $\mu\leq\frac{1}{2}(\sqrt{2m}-1)$, we have
$$m\geq2\mu^2+2\mu-\frac{1}{2}>\frac{(2\mu+1)\cdot2\mu}{2}\geq\binom{2\floor{\mu}+1}{2}\,,$$
where the second inequality uses $\mu>1$. Since $x\mapsto(n-x)x+\frac{1}{2}x(x-1)$ is strictly increasing for $x\in[1,\frac{n}{2}]$, and since (by completing the square) $n-\sqrt{n^2-2m}>\frac{m}{n}\geq\mu$, we have
$$(n-\floor{\mu})\floor{\mu}+\binom{\floor{\mu}}{2}\leq(n-\mu)\mu+\frac{\mu(\mu-1)}{2}\leq n\mu-\frac{\mu^2}{2}<m\,.$$
It follows from the theorem of \citeauthor{erdos1961minimal} that $\nu(G)\geq\floor{\mu}$.
\end{proof}

\begin{lemma}\label{lemma:GnmLeqGnp}
Let $1\leq m\leq n$ be integers and let $p:=m/n$. Let $S$ be the uniformly random subset of $[n]$ of size $m$, and let $T$ be the random subset of $[n]$ where each element is included independently with probability $p$. For any event $A\subseteq2^{[n]}$,
$$\bbP_S\{A\}\leq\sqrt{8np}\cdot\bbP_T\{A\}\,.$$
\end{lemma}
\begin{proof}
Using the fact that
$$\binom{n}{k}\geq\sqrt{\frac{n}{8k(n-k)}}2^{H(k/n)n}\,,$$
for $1\leq k\leq n-1$ (e.g. \cite[p.~309]{macwilliams1977theory}), we compute that
\begin{align*}
\bbP_T\{A\} &= \sum_{k=0}^{n}\bbP_T\{A\,|\,|T|=k\}\bbP_T\{|T|=k\} \\
&\geq \bbP_T\{A\,|\,|T|=np\}\bbP_T\{|T|=np\} \\
&= \bbP_S\{A\}\binom{n}{np}p^{np}(1-p)^{n(1-p)} \\
&\geq \bbP_S\{A\}\sqrt{\frac{1}{8p(1-p)n}}2^{H(p)n}p^{np}(1-p)^{n(1-p)} \\
&= \bbP_S\{A\}\sqrt{\frac{1}{8p(1-p)n}}\,,
\end{align*}
completing the proof.
\end{proof}

\begin{lemma}\label{lemma:Wrandomgraphazuma}
Let $A\subseteq[0,1]^2$ be a measurable set. Let $X_1,\dots,X_n$ be i.i.d. uniformly distributed on $[0,1]$. For all $1\leq i<j\leq n$ let $Y_{ij}$ be the indicator for the event $(X_i,X_j)\in A$. For all $\delta>0$ the random variable $Y:=\sum_{1\leq i<j\leq n}Y_{ij}$ satisfies
$$\bbP\left\{Y\leq(1-\delta)|A|\binom{n}{2}\right\}\leq\exp\left(-\frac{\delta^2|A|^2}{32}\cdot n\right)\,.$$
\end{lemma}
\begin{proof}
The function $f:[0,1]^n\to\N$ defined by
$$f(x_1,\dots,x_n):=\sum_{1\leq i<j\leq n}\bsone\{(x_i,x_j)\in A\}$$
satisfies the bounded differences condition with parameter $n$, that is, whenever two vectors $x,y\in[0,1]^n$ differ on at most one coordinate, we have $|f(x)-f(y)|\leq n$. It follows that the martingale defined by $Z_i:=\E[Y|X_1,\dots,X_i]$ for all $i\in[n]$ is $n$-Lipschitz, that is, $|Z_{i+1}-Z_i|\leq n$ for all $i$ almost surely. Since $\E Y=|A|\binom{n}{2}$, the bounded differences inequality (e.g. \cite[Theorem~4.4.4]{zhao2023graph}) implies that for all $\lambda>0$,
$$\bbP\left\{Y\leq|A|\binom{n}{2}-\lambda\right\}\leq e^{-\lambda^2/(2n^3)}\,,$$
and the result follows by taking $\lambda=\delta|A|\binom{n}{2}$.
\end{proof}

\section{Deferred Proofs}
\label{app:deferred}
\begin{proof}[Proof of \Cref{lemma:limsupentropyupperbound}]
Let $\widehat{\cQ}$ denote the set of graph limits of convergent sequences of graphs in $\cQ(n)$. For all $\gamma\in(0,1)$ let $\widehat{\cQ}_\gamma$ be the set of all $W\in\widehat{\cQ}$ such that $t(K_2,W)=\gamma$. Letting $m_\gamma:=\floor{\gamma\binom{n}{2}}$, we compute
\begin{align*}
&\limsup_{n\to\infty}\frac{1}{\binom{n}{2}}\log_2\bbP\{G\in\cQ(n)\} \\
&\hspace{2cm}= \limsup_{n\to\infty}\frac{1}{\binom{n}{2}}\sup_{\gamma\in[0,1]}\log_2\bbP\{G\in\cQ(n,m_\gamma)\} \\
&\hspace{2cm}= \limsup_{n\to\infty}\frac{1}{\binom{n}{2}}\sup_{\gamma\in[0,1]}\log_2\left(|\cQ(n,m_\gamma)|\left(\frac{p}{1-p}\right)^{m_\gamma}(1-p)^{\binom{n}{2}}\right) \\
&\hspace{2cm}= \sup_{\gamma\in[0,1]}\left\{\limsup_{n\to\infty}\frac{1}{\binom{n}{2}}\log_2|\cQ(n,m_\gamma)|+\gamma\log_2\left(\frac{p}{1-p}\right)+\log_2(1-p)\right\} \\
&\hspace{2cm}\leq \sup_{\gamma\in[0,1]}\left\{\sup_{W\in\widehat{\cQ}_\gamma}H(W)+\gamma\log_2\left(\frac{p}{1-p}\right)+\log_2(1-p)\right\} = \sup_{W\in\overline{\cQ}}I_p(W) \,,
\end{align*}
where we used \cite[Theorem~1]{hatami2018graph} to obtain the inequality.
\end{proof}

We will need the following lemma in the proof of \Cref{lemma:LDPforGNPkl}; note that an analogous result was proven for triangles \cite[Theorem~4.1]{radin2015singularities}.

\begin{lemma}\label{lemma:ChatgammaDgammaEquivOpt}
Let $\gamma\in(0,1)$, $p\in(0,1)$, $n,m\in\N$, and $m\sim\gamma\binom{n}{2}$. Fix a graph $H$ and let $\cF(n,m)$ be the set of induced-$H$-free graphs on $n$ vertices and $m$ edges. Define the set
$$\cF_\gamma:=\{W\in\cW:t(K_2,W)=\gamma,\,t_{\ind}(H,W)=0,\,\rand(W)>0\}$$
and assume $\cF_\gamma\neq\emptyset$. Let $\cX_\gamma$ denote the set of graphons that are limits of convergent sequences of graphs in $\cF(n,m)$. Then $\cF_\gamma\subseteq\cX_\gamma$ and
\begin{equation}
\sup\{I_p(W):W\in\cX_\gamma\}=\sup\{I_p(W):W\in\cF_\gamma\}\,.\label{eqn:SupIpGammaEquality}
\end{equation}
\end{lemma}
\begin{proof}
Let $W\in\cF_\gamma$ and we claim there exists a sequence $G_n\in\cF(n,m)$ of graphs such that $G_n\to W$. The sequence $G(n,W)$ of random graphs (\Cref{defn:GammaRandomGraph}) converges with probability 1 to $W$ in cut metric by \cite[Theorem~4.5]{borgs2008convergent}. For all $n\in\N$ let $m'=m'_n:=e(G(n,W))$. Hence $m'\sim\gamma\binom{n}{2}$ with probability 1. For all $n\in\N$, let $X^n_1,\dots,X^n_n$ and $Y^n_{ij}$, $1\leq i<j\leq n$, be the random variables that are i.i.d. uniformly distributed on $[0,1]$ and that are used to sample $G(n,W)$ as in \Cref{defn:GammaRandomGraph}. For all $n\in\N$ let $Y_n$ denote the number of pairs $1\leq i<j\leq n$ such that $(X^n_i,X^n_j)\in R:=R_W$. Applying \Cref{lemma:Wrandomgraphazuma} with $\delta=1/2$, the random variable $Y_n$ satisfies
$$\bbP\left\{Y_n\leq\frac{|R|n^2}{8}\right\}\leq\exp\left(-\frac{\delta^2|R|^2}{32}\cdot n\right)\,.$$
By the Borel--Cantelli lemma, there are only finitely many $n\in\N$ such that $Y_n\leq|R|n^2/8$ with probability 1. Whenever $Y_n>|R|n^2/8$, there exist $c:=|R|n^2/8$ edges $\{i_1,j_1\},\dots,\{i_c,j_c\}$ such that $(X^n_{i_k},X^n_{j_k})\in R$ for all $k\in[c]$; in this case Chernoff's inequality implies that the random variable $Z_n:=\sum_{k=1}^cY^n_{i_kj_k}$ satisfies
$$\bbP\left\{|Z_n-c/2|\geq n/4\right\}\leq2e^{-n/32}\,.$$
Again by the Borel--Cantelli lemma, with probability 1 there are only finitely many $n\in\N$ such that both $Y_n>|R|n^2/8$ and $|Z_n-c/2|\geq n/4$. For all $n\in\N$ define the random graph $H_n$ on the vertex set $[n]$ as follows:
\begin{enumerate}[(\emph{\roman*})]
	\item If $m'=m$ then let $H_n=G$.
	\item If any of the following three conditions holds, then let $H_n$ be an arbitrary graph in $\cF(n,m)$: (1) $Y_n\leq|R|n^2/8$; (2) $Y_n>|R|n^2/8$ and $|Z_n-c/2|\geq n/4$; or (3) $Y_n>|R|n^2/8$ and $|Z_n-c/2|<n/4$ and $|m-m'|>c/2-n/4$.
	\item Otherwise, $Y_n>|R|n^2/8$ and $|Z_n-c/2|<n/4$ and $d:=|m-m'|\leq c/2-n/4$. In this case there exist edge sets $E_1=\{i_kj_k:k\in[d]\}$ and $E_2=\{i_k'j_k':k\in[d]\}$ such that $E_1\subseteq E(G)$, $E_2\subseteq E(\overline{G})$, and $(X_u,X_v)\in R_W$ for all $uv\in E_1\cup E_2$. If $m>m'$ then let $H_n=(V,E(G)\cup E_2)$; otherwise $m<m'$, and let $H_n=(V,E(G)\setminus E_1)$.
\end{enumerate}
By definition, the graph $H_n$ always belongs to $\cF(n,m)$, and the conditions of case (\emph{ii}) are met for only finitely many $n\in\N$ with probability 1. It follows that $H_n\to W$ in cut metric with probability 1, completing the proof of the inclusion $\cF_\gamma\subseteq\cX_\gamma$.

To prove the second part of the lemma, first note that the supremum of $I_p(W)$ over $\cX_\gamma$ is achieved only if $\rand(W)>0$. Indeed, if $W\in\cX_\gamma$ is a graphon with $\rand(W)=0$ then $I_p(W)=\gamma\log_2(\frac{p}{1-p})+\log_2(1-p)$, and if $W'\in\cX_\gamma$ is a graphon with $r=:\rand(W')>0$ (such a graphon exists by the first part of the lemma) then $H(W')>0$ and we have
$$I_p(W')=H(W')+\gamma\log_2\left(\frac{p}{1-p}\right)+\log_2(1-p)>I_p(W)\,.$$
The result follows if we prove $\{W\in\cX_\gamma:\rand(W)>0\}\subseteq\cF_\gamma$. Indeed, $W$ is the limit of a convergent sequence $G\in\cF(n,m)$ of graphs, and since $\rand(W)>0$, we know $v(G_n)\to\infty$. Hence continuity of the homomorphism density implies $t(F,\Gamma)=\lim_{n\to\infty}t(F,G_n)$ for all fixed graphs $F$, in particular, $t(K_2,W)=\gamma$ and $t_{\ind}(K_{1,3},W)=0$.
\end{proof}

\begin{proof}[Proof of \Cref{lemma:LDPforGNPkl}]
For all $\gamma\in(0,1)$ let $\cF_\gamma$ denote the set of all $W\in\cF$ with $t(K_2,W)=\gamma$. Fix $\gamma\in(0,1)$ such that there exists $W\in\cF_\gamma$ with $\rand(W)>0$, and let $G\sim G(n,W)$. Letting $m_\gamma:=\floor{\gamma\binom{n}{2}}$, we claim
\begin{equation}
H(G)\leq\log_2|\cF(n,m_\gamma,H)|+o(n^2)\,,\label{eqn:HGleqlog2FnmgammaH}
\end{equation}
where $H(G)$ is the entropy of $G$ as a discrete random variable. For all $n,m\in\N$ let $\cS(n)$ be the set of all $H\in\cG(n)$ such that $\bbP\{G=H\}>0$, and let $\cS(n,m):=\cS(n)\cap\cG(n,m)$. We write $\bbP\{H'\}:=\bbP\{G=H'\}$ and $\bbP\{\cH\}:=\bbP\{G\in\cH\}$ for a graph $H'$ and set of graphs $\cH$.

Let $0<L<U<1$ be fixed constants such that the set $R:=\{(x,y):L<W(x,y)<U\}$ has positive measure. For all $\epsilon\in(0,1)$ let $R_\epsilon\subseteq R$ be a fixed subset of measure $\epsilon|R|$. Let $X_1,\dots,X_n$, and $Y_{ij}$, $1\leq i<j\leq n$, denote the random variables associated with $G(n,W)$ (see \Cref{defn:GammaRandomGraph}). Let
$$Y:=\sum_{1\leq i<j\leq n}\bsone\{(X_i,X_j)\in R_\epsilon\}\hspace{6mm}\text{and}\hspace{6mm}Z:=\sum_{\substack{1\leq i<j\leq n\\(X_i,X_j)\in R_\epsilon}}\bsone\{Y_{ij}\leq W(X_i,X_j)\}\,.$$
By \Cref{lemma:Wrandomgraphazuma}, $Y$ is concentrated around its mean $y_\epsilon:=|R_\epsilon|\binom{n}{2}$. Conditioned on $Y$, the random variable $Z$ also concentrates around its mean $LY\leq z_\epsilon\leq UY$ by Chernoff's inequality. Let $\delta=\delta(\epsilon)>0$ be sufficiently small and let
$$\cS'(n):=\left\{H\in\cS(n):\bbP\left\{H\,\bigg|\,\genfrac{}{}{0pt}{0}{y_\epsilon-\delta{\textstyle\binom{n}{2}}\leq Y\leq y_\epsilon+\delta{\textstyle\binom{n}{2}}\vphantom{\displaystyle\bigcup}}{z_\epsilon-\delta{\textstyle\binom{n}{2}}\leq Z\leq z_\epsilon+\delta{\textstyle\binom{n}{2}}}\right\}>0\right\}$$
and $\cS'(n,m):=\cS'(n)\cap\cG(n,m)$. Let $\eta:=\epsilon|R|\min\{L,1-U\}$ and we claim that for all $(\gamma-\eta)\binom{n}{2}\leq m\leq(\gamma+\eta)\binom{n}{2}$,
\begin{equation}
\log|\cS'(n,m)|\leq\log|\cS(n,m_\gamma)|+o(n^2)\,.\label{eqn:SprimeSnmineqClaim}
\end{equation}
Indeed, for all $H\in\cS'(n,m)$, there exist numbers $x_1,\dots,x_n\in[0,1]$ and $y_{ij}\in[0,1]$ in the support of $G$ (in the sense of \Cref{defn:GammaRandomGraph}) such that the number of pairs $(x_i,x_j)\in R_\epsilon$ is in the interval $y_\epsilon\pm\delta\binom{n}{2}$ and the number of $y_{ij}$ such that $y_{ij}\leq W(x_i,x_j)$ is in the interval $z_\epsilon\pm\delta\binom{n}{2}$. If $m<\gamma\binom{n}{2}$ then we can obtain a graph $H'\in\cS(n,m_\gamma)$ in the support of $G$ by changing at most $\delta\binom{n}{2}$ of the values $y_{ij}$ from belonging to the interval $(W(x_i,x_j),1)$ to belonging to the interval $(0,W(x_i,x_j))$, and analogously if $m>\gamma\binom{n}{2}$. Since $|R_\epsilon|=\epsilon|R|<\epsilon$, the mapping $f:\cS'(n,m)\to\cS(n,m_\gamma):H\mapsto H'$ has the following property: for all $H'\in\cS(n,m_\gamma)$ there are at most $2^{\epsilon\binom{n}{2}}$ graphs $H\in\cS'(n,m)$ such that $f(H)=H'$. It directly follows that
$$|\cS'(n,m)|\leq|\cS(n,m_\gamma)|2^{\epsilon\binom{n}{2}}\,,$$
completing the proof of \pref{eqn:SprimeSnmineqClaim}.

Since $e(G)\binom{n}{2}^{-1}\to\gamma$ with probability 1, we use \pref{eqn:SprimeSnmineqClaim} to compute that
\begin{align*}
H(G) &= -\sum_{m=0}^{\binom{n}{2}}\bbP\{\cS(n,m)\}\sum_{H'\in\cS(n)}\bbP\{H'\,|\,\cS(n,m)\}\log_2\bbP\{H'\} \\
&= -(1+o(1))\sum_{m=(\gamma-\eta)\binom{n}{2}}^{(\gamma+\eta)\binom{n}{2}}\bbP\{\cS'(n,m)\}\sum_{H'\in\cS(n)}\bbP\{H'\,|\,\cS'(n,m)\}\log_2\bbP\{H'\} \\
&\leq (1+o(1))\log_2\left(\sum_{m=(\gamma-\eta)\binom{n}{2}}^{(\gamma+\eta)\binom{n}{2}}|\cS'(n,m)|\right) \\[5pt]
&\leq (1+o(1))\log_2|\cS(n,m_\gamma)| \leq (1+o(1))\log_2|\cF(n,m_\gamma,H)| \,,
\end{align*}
where we used concavity of the entropy, completing the proof of \pref{eqn:HGleqlog2FnmgammaH}. Using \pref{eqn:HGleqlog2FnmgammaH} and the fact that $H(G)\binom{n}{2}^{-1}\to H(W)$ (see \cite[Theorem~D.5]{janson2010graphons}), we obtain
\begin{align*}
&\liminf_{n\to\infty}\frac{1}{\binom{n}{2}}\log_2\bbP\{G\in\cF(n,H)\} \\
&\hspace{2cm}\geq \liminf_{n\to\infty}\frac{1}{\binom{n}{2}}\log_2\bbP\{G\in\cF(n,m_\gamma,H)\} \\
&\hspace{2cm}= \liminf_{n\to\infty}\frac{1}{\binom{n}{2}}\log_2\left(|\cF(n,m_\gamma,H)|\left(\frac{p}{1-p}\right)^{m_\gamma}(1-p)^{\binom{n}{2}}\right) \\
&\hspace{2cm}\geq \liminf_{n\to\infty}\frac{1}{\binom{n}{2}}\left(H(G(n,W))+m_\gamma\log_2\left(\frac{p}{1-p}\right)\right)+\log_2(1-p) \\
&\hspace{2cm}= H(W)+\gamma\log_2\left(\frac{p}{1-p}\right)+\log_2(1-p) = I_p(W) \,,
\end{align*}
and taking the supremum over $W\in\cF$ on the right-hand side proves the first inequality.

To prove the opposite inequality first let $\cX_\gamma$ be the set of graphons $W$ that are limits of convergent sequences $G_n\in\cF(n,H)$ such that $t(K_2,W)=\gamma$. Now \Cref{lemma:ChatgammaDgammaEquivOpt} proves
\begin{equation}
\sup_{W\in\cF_\gamma}H(W)=\sup_{W\in\cX_\gamma}H(W)\,.\label{eqn:lemma29FFhatEqualEntropy}
\end{equation}
so the upper bound follows directly by applying \Cref{lemma:limsupentropyupperbound}, completing the proof.
\end{proof}

\begin{proof}[Proof of \Cref{lemma:FnmGraphonOptAsymps}]
Let $W\in\cF_\gamma$ and $G\sim G(n,W)$. The proof of \Cref{lemma:LDPforGNPkl} shows $H(G)\leq\log_2|\cF(n,m,H)|+o(n^2)$, so using the fact that $H(G)\binom{n}{2}^{-1}\to H(W)$,
$$\liminf_{n\to\infty}\frac{1}{\binom{n}{2}}\log_2|\cF(n,m,H)|\geq\liminf_{n\to\infty}\frac{H(G)}{\binom{n}{2}}=H(W)\,.$$
Applying \Cref{lemma:limsupentropyupperbound} to $\cF(n,m,H)$ with $p=1/2$, and again using \Cref{lemma:ChatgammaDgammaEquivOpt}, we obtain the opposite inequality, completing the proof.
\end{proof}

\begin{proof}[Proof of \Cref{lemma:kktopt}]
Define the Lagrangian
$$L(x,y,\lambda,\mu)=yH\left(\frac{c-x}{y}\right)+\lambda\left(x+y-1\right)+\mu(y-x)\,.$$
The equation $\nabla L=0$ implies
\begin{align*}
\frac{\partial L}{\partial x}(x,y) &= \log\left(\frac{c-x}{y}\right)-\log\left(1-\frac{c-x}{y}\right)+\lambda-\mu = 0 \,, \\
\frac{\partial L}{\partial y}(x,y) &= -\log\left(1-\frac{c-x}{y}\right)+\lambda+\mu = 0 \,,
\end{align*}
hence
\begin{equation}
\lambda=\log\left(1-\frac{c-x}{y}\right)-\frac{1}{2}\log\left(\frac{c-x}{y}\right)\,,\hspace{5mm}\mu=\frac{1}{2}\log\left(\frac{c-x}{y}\right)\,.\label{eqn:kktdualvars}
\end{equation}
The Karush--Kuhn--Tucker (KKT) conditions assert that $\lambda(x+y-1)=0$ and $\mu(y-x)=0$, and that the dual variables $\lambda$ and $\mu$ are nonpositive. From the KKT conditions we deduce $x=y$, since otherwise $\mu=0$, implying the strictly suboptimal objective value $f(x,y)=0$. In the remainder we take $x=y$ and analyze the two cases $\lambda=0$ and $\lambda<0$ pertaining to the first KKT equation. If $\lambda=0$ then \pref{eqn:kktdualvars} implies $(c-x)/y=(3-\sqrt{5})/2$, which further implies $x=y=\frac{5+\sqrt{5}}{10}c$\,; in this case $f$ takes the value
$$g_1(c):=\frac{5+\sqrt{5}}{10}H\left(\frac{3-\sqrt{5}}{2}\right)c\,.$$
Note that the inequality $x+y\leq1$ implies $\lambda=0$ is feasible if and only if $0<c\leq(5-\sqrt{5})/4$. Alternatively if $\lambda<0$, then the KKT condition asserts $x+y=1$ and $x=y=1/2$, and $f$ takes the value $g_2(c):=\frac{1}{2}H(2c-1)$. It is easy to verify that $g_1(c)$ is the line tangent to $g_2(c)$ at $c=(5-\sqrt{5})/4$, which also proves the second statement of the lemma. Hence by strict concavity of $g_2(c)$ we have $g_1(c)\geq g_2(c)$ with equality if and only if $c=(5-\sqrt{5})/4$. \Cref{eqn:kktlemma} now follows since we proved that $\lambda<0$ whenever $c>(5-\sqrt{5})/4$.
\end{proof}

\section{Asymptotic Enumeration of Co-Bipartite Graphs}
\label{sec:counting}
In this section we prove the following proposition, which gives a formula for the asymptotic number co-bipartite graphs at edge density between $\frac{1}{2}$ and $1$.

\begin{proposition}\label{prop:asymp-bcnm}
Let $\gamma\in(\frac{1}{2},1)$ and $n,m\in\N$. If $m\sim\gamma\binom{n}{2}$ then
$$|\cB_c(n,m)|\sim\left(\frac{r+1}{2}+\sum_{k=1}^\infty\left(2\gamma-1\right)^{k^2+rk}\right)\binom{n}{\floor{n/2}}\binom{\floor{n^2/4}}{m-\binom{\floor{n/2}}{2}-\binom{\ceil{n/2}}{2}}\,,$$
where $r=n\bmod{2}$.
\end{proposition}

\noindent Note that by taking $\gamma=3/4$ in \Cref{prop:asymp-bcnm}, the asymptotics of $\log|\cB_c(n,m)|$ match the asymptotics of $\log|\cB_c(n)|$, which reflects the fact that almost all co-bipartite graphs have edge density approximately $3/4$. The proof of \Cref{prop:asymp-bcnm} is standard based on methods of earlier works including \cite{erdos1976enumeration,osthus2003densities}.

The \emph{imbalance} of a bipartition $\{A,B\}$ of $V=[n]$ is defined to be $||A|-|B||$, and we say that a bipartition is \emph{almost equitable} if its imbalance is at most $\sqrt{\log n}$. Define the set of 2-clique-covered graphs on $n$ vertices and $m$ edges
$$\Cov_{n,m}:=\{(G,\{A,B\}):G\in\cG(n,m),\,V=A\sqcup B,\,G[A]=K_A,\,G[B]=K_B\}\,.$$
We show in \Cref{lemma:almostallcovnmbalanced} that almost every $(G,\{A,B\})\in\Cov_{n,m}$ is almost equitable. In \Cref{lemma:universalvtxlittleoh1,cor:cbuniquepartition}, we show that almost all co-bipartite graphs $G\in\cB_c(n,m)$ admit a unique 2-clique-cover. By enumerating almost-equitable 2-clique-covered graphs on $n$ vertices and $m$ edges, we obtain an asymptotic formula for $|\cB_c(n,m)|$.

\begin{lemma}\label{lemma:almostallcovnmbalanced}
Let $\gamma\in(\frac{1}{2},1)$ and $m\sim\gamma\binom{n}{2}$. For almost all $(G,\{A,B\})\in\Covnm$, the bipartition $\{A,B\}$ is almost equitable, that is,
$$|\Covnm|\sim\sum_{A,B}\binom{|A|\cdot|B|}{m-\binom{|A|}{2}-\binom{|B|}{2}}\,,$$
where the sum is over almost-equitable bipartitions $\{A,B\}$ of $V$.
\end{lemma}
\begin{proof}
For all $n\in\N$ define the set
$$I_n:=\begin{cases}\{2j:j\in\N,\,\sqrt{\log n}<2j\leq n-2\}&n\text{ is even}\\\{2j+1:j\in\N,\,\sqrt{\log n}<2j+1\leq n-2\}&n\text{ is odd}\end{cases}\,,$$
so that $I_n$ is the set of possible imbalances of a bipartition of $n$ vertices that are \emph{not} almost equitable. Using \Cref{lemma:BinomialCoeffIneqs} and $m\sim\gamma\binom{n}{2}$, we find that for $k\in I_n$,
\begin{align}
&\binom{(n-k)(n+k)/4}{m-\binom{(n-k)/2}{2}-\binom{(n+k)/2}{2}}{\binom{\floor{n/2}\cdot\ceil{n/2}}{m-\binom{\floor{n/2}}{2}-\binom{\ceil{n/2}}{2}}}^{-1} \nonumber\\
&\hspace{3cm}= \binom{\floor{n^2/4}-\floor{k^2/4}}{m-(n^2+k^2)/4+n/2}{\binom{\floor{n^2/4}}{m-\floor{n^2/4}+\floor{n/2}}}^{-1} \nonumber\\
&\hspace{3cm}= \binom{\floor{n^2/4}-\floor{k^2/4}}{m-\floor{n^2/4}-\floor{k^2/4}+\floor{n/2}}{\binom{\floor{n^2/4}}{m-\floor{n^2/4}+\floor{n/2}}}^{-1} \label{eqn:binomratioCovNMBound}\\
&\hspace{3cm}\leq \exp\left(-\left(1-\frac{m-\floor{n^2/4}+\floor{n/2}}{\floor{n^2/4}}\right)\floor*{\frac{k^2}{4}}\right) \nonumber\\
&\hspace{3cm}\sim \exp\left(-2(1-\gamma)\floor*{\frac{k^2}{4}}\right) < e^{-\frac{1}{4}(1-\gamma)k^2} \nonumber\,,
\end{align}
provided $n$ is large enough. Hence the number of $(G,\{A,B\})\in\Covnm$ where $\{A,B\}$ is not almost equitable is at most
\begin{align}
&\sum_{k\in I_n}\binom{n}{(n-k)/2}\binom{(n-k)(n+k)/4}{m-\binom{(n-k)/2}{2}-\binom{(n+k)/2}{2}} \nonumber\\
&\hspace{3cm}\leq \binom{n}{\floor{n/2}}\sum_{k\in I_n}\binom{(n-k)(n+k)/4}{m-\binom{(n-k)/2}{2}-\binom{(n+k)/2}{2}} \nonumber\\
&\hspace{3cm}\leq \left(\sum_{k=\floor{\sqrt{\log n}}}^\infty e^{-\frac{1}{4}(1-\gamma)k^2}\right)\binom{n}{\floor{n/2}}\binom{\floor{n/2}\cdot\ceil{n/2}}{m-\binom{\floor{n/2}}{2}-\binom{\ceil{n/2}}{2}} \nonumber\\
&\hspace{3cm}\leq \frac{n^{-4/(1-\gamma)^2}}{\sqrt{\log n}}\cdot\frac{1}{2}\binom{n}{\floor{n/2}}\binom{\floor{n/2}\cdot\ceil{n/2}}{m-\binom{\floor{n/2}}{2}-\binom{\ceil{n/2}}{2}} \nonumber\\
&\hspace{3cm}\leq \frac{n^{-4/(1-\gamma)^2}}{\sqrt{\log n}}\cdot|\Covnm|\,,\label{eqn:SumInequitableBiparts}
\end{align}
where we used a standard Gaussian tail bound, and the last inequality holds since the second factor on the left-hand side of \pref{eqn:SumInequitableBiparts} is a lower bound on the number of 2-clique-covered graphs $(G,\{A,B\})\in\Covnm$ satisfying $||A|-|B||\leq1$.
\end{proof}

\begin{lemma}\label{lemma:universalvtxlittleoh1}
Let $\gamma\in(\frac{1}{2},1)$ and $m\sim\gamma\binom{n}{2}$. Let $0\leq k\leq\frac{1}{2}\sqrt{\gamma-1/2}\cdot n$ be an integer and let $\{A,B\}$ be a bipartition of $V$ with imbalance $k$. Let $G$ be a uniformly random graph on $n$ vertices and $m$ edges conditioned on the event that $G[A]$ and $G[B]$ are cliques. The probability that $G$ has a universal vertex, i.e. one adjacent to all other vertices, is at most $e^{-Cn}$ for a constant $C=C(\gamma)>0$ and large enough $n$.
\end{lemma}
\begin{proof}
Since $\{A,B\}$ has imbalance $k$, we may assume $|A|=(n-k)/2$ and $|B|=(n+k)/2$. For vertices $v\in A$ and $w\in B$, the inequality $|A|\leq|B|$ implies $\bbP_{G}\{v\text{ is universal}\}$ is at most $\bbP_{G}\{w\text{ is universal}\}$. The upper bound on $k$ ensures that
$$m-\binom{(n-k)/2}{2}-\binom{(n+k)/2}{2}-\frac{n-k}{2}=m-\frac{n^2+k^2-2k}{4}\geq0\,,$$
provided $n$ is large enough. Using \Cref{lemma:BinomialCoeffIneqs} and $m\sim\gamma\binom{n}{2}$, the probability that $w$ is universal in $G$ is at most
\begin{align*}
&\binom{(n-k)(n+k)/4-(n-k)/2}{m-\binom{(n-k)/2}{2}-\binom{(n+k)/2}{2}-(n-k)/2}\binom{(n-k)(n+k)/4}{m-\binom{(n-k)/2}{2}-\binom{(n+k)/2}{2}}^{-1} \\
&\hspace{4cm}\leq \exp\left(-\left(1-\frac{m-\binom{(n-k)/2}{2}-\binom{(n+k)/2}{2}}{(n-k)(n+k)/4}\right)\frac{n-k}{2}\right) \\
&\hspace{4cm}= \exp\left(-\left(1-\frac{4m-n^2-k^2+2n}{n^2-k^2}\right)\frac{n^2-k^2}{2(n+k)}\right) \\
&\hspace{4cm}\leq \exp\left(-\frac{n^2-2m-n}{2n}\right) \sim e^{-\frac{1}{2}(1-\gamma)n+\frac{1}{2}}\,.
\end{align*}
It follows that the probability $G$ contains a universal vertex is at most $n\bbP\{w\text{ is universal}\}\leq e^{-Cn}$ for a constant $C=C(\gamma)>0$.
\end{proof}

\begin{corollary}\label{cor:cbuniquepartition}
Let $\gamma\in(\frac{1}{2},1)$ and $m\sim\gamma\binom{n}{2}$. Almost all $G\in\cB_c(n,m)$ admit a unique covering by two cliques.
\end{corollary}
\begin{proof}
Let $W\subseteq\Covnm$ denote the set of 2-clique-covered graphs $(G,\{A,B\})$ on $n$ vertices and $m$ edges such that $\{A,B\}$ is almost equitable and $G$ has a unique 2-clique-cover. Let
\begin{align*}
X &:= \{(G,\{A,B\})\in\Cov_{n,m}\setminus W:||A|-|B||\leq\sqrt{\log n}\} \\
Y &:= \{(G,\{A,B\})\in\Cov_{n,m}\setminus W:||A|-|B||>\sqrt{\log n}\} \,.
\end{align*}
We clearly have $|W|\leq|\cB_c(n,m)|\leq|\Cov_{n,m}|$ and $\Cov_{n,m}=W\cup X\cup Y$. \Cref{lemma:almostallcovnmbalanced} proves $|W\cup X|\gg|Y|$, and since a co-bipartite graph admits a unique covering by two cliques if and only if it has no universal vertex, \Cref{lemma:universalvtxlittleoh1} proves $|W|\gg|X|$. It follows that $\cB_c(n,m)\sim|W|$, completing the proof.
\end{proof}

\begin{proof}[Proof of \Cref{prop:asymp-bcnm}]
We first calculate the asymptotics of ratios of binomial coefficients. For all $n\in\N$ define the set
$$J_n:=\begin{cases}\{2j:j\in\{0,1,\dots\},\,2j\leq\sqrt{\log n}\}&n\text{ is even}\\\{2j+1:j\in\{0,1,\dots\},\,2j+1\leq\sqrt{\log n}\}&n\text{ is odd}\end{cases}\,,$$
so that $J_n$ is the set of possible imbalances of an almost-equitable bipartition of $n$ vertices. For all $k\in J_n$ we compute that
\begin{equation}
\binom{n}{(n+k)/2}\binom{n}{\floor{n/2}}^{-1} = \binom{n}{(n-k)/2}\binom{n}{\floor{n/2}}^{-1} \sim 1\,,\label{eqn:bcnmasynmpratio1}
\end{equation}
where the asymptotic relation follows from \Cref{lemma:ConvHypergeomToBinomial} and $k=O(\log n)$. We now continue the calculation from \Cref{eqn:binomratioCovNMBound}, this time assuming $k\in J_n$ and using \Cref{lemma:ConvHypergeomToBinomial} to obtain asymptotics:
\begin{align}
&\binom{(n-k)(n+k)/4}{m-\binom{(n-k)/2}{2}-\binom{(n+k)/2}{2}}{\binom{\floor{n/2}\cdot\ceil{n/2}}{m-\binom{\floor{n/2}}{2}-\binom{\ceil{n/2}}{2}}}^{-1} \nonumber\\
&\hspace{3cm}= \binom{\floor{n^2/4}-\floor{k^2/4}}{m-\floor{n^2/4}-\floor{k^2/4}+\floor{n/2}}{\binom{\floor{n^2/4}}{m-\floor{n^2/4}+\floor{n/2}}}^{-1} \nonumber\\
&\hspace{3cm}\sim \left(\frac{m-\floor{n^2/4}+\floor{n/2}}{\floor{n^2/4}}\right)^{\floor{k^2/4}} \sim \left(2\gamma-1\right)^{(k^2-r)/4}\,\label{eqn:bcnmasynmpratio2}
\end{align}
where $r=n\bmod{2}$. \Cref{lemma:almostallcovnmbalanced,cor:cbuniquepartition}, together with (\ref{eqn:bcnmasynmpratio1}) and (\ref{eqn:bcnmasynmpratio2}), imply that if $n$ is even, we have
\begin{align}
|\cB_c(n,m)| &\sim \frac{1}{2}\binom{n}{n/2}\binom{n^2/4}{m-2\binom{n/2}{2}}+\sum_{k\in J_n\setminus\{0\}}\binom{n}{(n+k)/2}\binom{(n-k)(n+k)/4}{m-\binom{(n-k)/2}{2}-\binom{(n+k)/2}{2}} \nonumber\\
&\sim \binom{n}{n/2}\binom{n^2/4}{m-2\binom{n/2}{2}}\Bigg(\frac{1}{2}+\sum_{k\in J_n\setminus\{0\}}\left(2\gamma-1\right)^{k^2/4}\Bigg) \nonumber\\
&\sim \binom{n}{n/2}\binom{n^2/4}{m-2\binom{n/2}{2}}\left(\frac{1}{2}+\sum_{\ell=1}^{\infty}\left(2\gamma-1\right)^{\ell^2}\right) \,,\label{eqn:bcnmasympevencase}
\end{align}
where we substituted $\ell=k/2$. Similarly, if $n$ is odd then
\begin{align}
|\cB_c(n,m)| &\sim \sum_{k\in J_n}\binom{n}{(n+k)/2}\binom{(n-k)(n+k)/4}{m-\binom{(n-k)/2}{2}-\binom{(n+k)/2}{2}} \nonumber\\
&\sim \binom{n}{\floor{n/2}}\binom{\floor{n/2}\cdot\ceil{n/2}}{m-\binom{\floor{n/2}}{2}-\binom{\ceil{n/2}}{2}}\sum_{k\in J_n}\left(2\gamma-1\right)^{(k^2-1)/4} \nonumber\\
&\sim \binom{n}{\floor{n/2}}\binom{\floor{n^2/4}}{m-\binom{\floor{n/2}}{2}-\binom{\ceil{n/2}}{2}}\sum_{\ell=0}^{\infty}\left(2\gamma-1\right)^{\ell^2+\ell} \,,\label{eqn:bcnmasympoddcase}
\end{align}
where we substituted $\ell=(k-1)/2$. Combining \pref{eqn:bcnmasympevencase} and \pref{eqn:bcnmasympoddcase} completes the proof.
\end{proof}

\end{document}